\documentclass[11pt]{amsart}
\usepackage{amsthm}
\usepackage{tikz}
\usepackage{multirow}
\usepackage{caption}
\usepackage{xcolor} 
\usepackage{epstopdf}
\usepackage{subcaption}
\usepackage{csquotes}
\usepackage{tikz-cd}
\usepackage{arydshln}
\usepackage{amsmath}
\usepackage{amsfonts}
\usepackage{algpseudocode}
\usepackage{enumitem}
\usepackage[ruled,vlined]{algorithm2e}
\usetikzlibrary{shapes,arrows}
\usetikzlibrary{calc,arrows.meta,positioning,shapes}

\usepackage{amsfonts,amssymb,amsmath}
\usepackage{mleftright}
\usepackage{amssymb, amsmath}
\usepackage{tikz}
\usepackage[numbers,sort&compress]{natbib}
\usepackage[colorlinks=true, linkcolor=blue, urlcolor=blue, citecolor=blue]{hyperref}
\usepackage{filecontents}
\usetikzlibrary{positioning, arrows.meta}

\usepackage{xcolor}
\usepackage{mathrsfs}
\usepackage{graphicx}
\usepackage{hyperref}
\usepackage{mathabx}
\usepackage[left=2.5cm,top=1cm,right=2.5cm,nofoot]{geometry}

\usepackage{cleveref}
\usepackage{relsize}
\usepackage[mathlines]{lineno}

\newtheorem{thm}{Theorem}[section]

\newtheorem{lem}[thm]{Lemma}

\newtheorem{deff}[thm]{Definition}
\newtheorem{prop}[thm]{Proposition}
\newtheorem{cor}[thm]{Corollary}

\newtheorem{rk}[thm]{Remark}
\newtheorem{ex}[thm]{Example}

\newenvironment{exmp}{\begin{ex}
		\rm }{\end{ex}}
\newenvironment{remar}{\begin{rk}			\rm }{\end{rk}}
\newenvironment{defn}{\begin{deff}
		\rm }{\end{deff}}
\newcommand{\mc}{\mathcal}

\newcommand{\mf}{\mathfrak}

\newcommand{\ol}{\overline}

\def \l2x{L^2(X,\mc H)}

\DeclareMathOperator*{\supp}{supp}

\theoremstyle{definition}

\theoremstyle{remark}

\numberwithin{equation}{section}

\usepackage{scalerel,stackengine}
\stackMath
\newcommand\reallywidehat[1]{%
	\savestack{\tmpbox}{\stretchto{%
			\scaleto{%
				\scalerel*[\widthof{\ensuremath{#1}}]{\kern.1pt\mathchar"0362\kern.1pt}%
				{\rule{0ex}{\textheight}}
			}{\textheight}%
		}{2.4ex}}%
	\stackon[-6.9pt]{#1}{\tmpbox}%
}
\begin{document}
	\sloppy
	\title[Ramanujan sums  in signal recovery ]{
		Ramanujan sums in signal recovery and uncertainty principle inequalities}

	\author{Sahil Kalra}
	\address{Department of Mathematics,
		Indian Institute of Technology Indore,
		 Khandwa Road, Simrol,
		Indore, Madhya Pradesh- 453552}
	\email{sahilkalramath@gmail.com, nirajshukla@iiti.ac.in}
	\thanks{S.K. gratefully acknowledges the financial support from University Grant Commission (UGC)- Ref. No.: 191620003953 and N.K.S. acknowledges the financial support from DST-SERB Project [MTR/2022/000176]. The authors also acknowledge the facilities of the Bhaskaracharya Mathematics Laboratory, IIT Indore, supported by the DST-FIST Project (File No.: SR/FST/MS I/2018/26).}
	\author{Niraj K. Shukla}

	\maketitle
	\begin{abstract}
		This paper explores the perfect reconstruction property of filter banks based on Ramanujan sums and their applications in signal recovery. Originally introduced by Srinivasa Ramanujan, Ramanujan sums serve as powerful tools for extracting periodic components from signals and form the foundation of Ramanujan filter banks. We investigate the perfect reconstruction property of these filter banks and analyze their robustness against erasures for discrete-time signals in a finite-dimensional space $\ell^2(\mathbb Z_N)$ \,(i.e., $\mathbb C^N$). The study is further extended to non-uniform Ramanujan filter banks, showcasing their ability to address the limitations of uniform ones. Employing  the reconstruction properties of uniform Ramanujan filter banks, we present an uncertainty principle associated with a tight frame of shifts of Ramanujan sums. This principle establishes representation inequalities in terms of Euler's totient function $\phi(n)$ that provide sufficient conditions for the perfect recovery of signals in scenarios where signal information is lost during transmission or corrupted by noise. Finally, we illustrate that utilizing the signal’s periodicity information through Ramanujan filter banks significantly improves the efficiency of signal recovery optimization algorithms, resulting in enhanced signal-to-noise ratio (SNR) gains and more precise reconstruction.
		
	\end{abstract}
	
	\section{Introduction}	%
	
	This paper investigates the   perfect reconstruction property of  filter banks based on Ramanujan sums and their applications in  erasure handling/signal recovery. The Ramanujan sum was introduced by 	the  Indian mathematician Srinivasa Ramanujan
	in 1918.	For a positive integer $q$, the  \textit{$q$-th Ramanujan sum} $c_q$ has the following representation
	\cite{ramanujan1918certain}:
	\begin{equation}\label{eq:Ramanujansum}
		c_q(n)=\sum_{\substack{k=1\\(k,q)= 1}}^q e^{2\pi i k n /q},\quad n\geq1,
	\end{equation}
	where $(k,q)$ denotes the greatest common divisor of $k$ and $q.$	
	Ramanujan showed that several   arithmetic functions in number theory, such as the Euler's totient function $\phi(n)$ and the Möbius function $\mu(n)$, can be expressed as linear combinations of Ramanujan sums.  In particular, the collection of \( \phi(q) \)-successive circular shifts of the  $q$-th Ramanujan sum, i.e.,
	$
	c_q, c_q(\cdot- 1), \ldots, c_q(\cdot - \phi(q) + 1)$
	is linearly independent and generates a $\phi(q)$-dimensional subspace, known as a $\textit{Ramanujan subspace}$ For further details, we refer the reader to \cite{vaidyanathan2014ramanujan1, vaidyanathan2014ramanujan2}.
	
	In recent years, Ramanujan sums and Ramanujan subspaces have gained attention in signal processing due to their ability to extract periodic components from discrete-time signals \cite{ planat2002ramanujan,  planat2009ramanujan, vaidyanathan2014ramanujan1}.  The primary objective of our study is to analyze the perfect reconstruction property of the filter bank associated with Ramanujan sums, also known as the \textit{Ramanujan filter bank}, in the finite-dimensional setting of $\ell^2(\mathbb{Z}_N)$ (i.e., $\mathbb{C}^N$), where $\ell^2(\mathbb{Z}_N)$ refers to the $N$-dimensional Hilbert space of functions defined on the cyclic group $\mathbb{Z}_N = \{0, 1, \dots, N-1\}$. This space is equipped with the inner  product 
	$
	\langle x, y \rangle = \sum_{n=0}^{N-1} x(n) \overline{y(n)},
	$
	where \(\overline{y(n)}\) is the complex conjugate of \(y(n)\) and $x=(x(0),x(1), \hdots, x(N-1))$, $y=(y(0),y(1), \hdots, y(N-1)) \in \ell^2(\mathbb Z_N).$ Any signal \(x \in \ell^2(\mathbb{Z}_N)\) is periodic with period \(N\), i.e., \(x(n + N) = x(n)\) for all \(n \in \mathbb{Z}\).  Periodic signals represent a significant class of signals, as they often encapsulate critical information about the observed phenomena. Examples of periodic signals include segments of tonal music, voiced speech, ECG signals, tandem repeats in DNA associated with genetic disorders, and protein repeats in amino acids that influence their structure and binding properties \cite{tenneti2016detecting, cristea2003large, sameni2008multichannel}.
	
	%
	%
	Ramanujan filter banks possess unique properties that make them exceptionally well-suited for detecting periodic patterns in real-time data \cite{vaidyanathan2015properties,   tenneti2016detection, abraham2021design, tenneti2015ramanujan}.  Filter banks have consistently attracted the attention of researchers for constructing frames and wavelets \cite{BOWNIK2023229, BOWNIK20091065, christensen2017explicit,  han2022multivariate, han2000frames, han2007frames, christensen2016frames, han2018framelets, bolcskei1998frame, cvetkovic1998oversampled}. Frames derived from filter banks play a pivotal role in various fields, including signal processing, image processing,  data analysis, wavelets, shearlets, Gabor systems, etc. For instance, filter banks are used to extract distinguishing features from data, thereby improving pattern recognition tasks such as fingerprint recognition and face detection. Additionally, filter banks are integral to the construction of multiresolution analysis (MRA) systems, which enable the representation of signals or data at varying levels of detail and  form the foundation for applications like signal compression, denoising, and feature extraction \cite{BOWNIK2023229, BOWNIK20091065, han2000frames, christensen2016frames}.

	%
	
	
	Vaidyanathan and Tenneti \cite{vaidyanathan2015properties} studied  Ramanujan filter banks and demonstrated their effectiveness in identifying the period of an unknown signal.  However, when the period $P$ is large, their approach of finding the period through Ramanujan  filter banks can be computationally expensive. 
	We observe that the period identification process can be accelerated in such scenarios  if we know that $P$ is a divisor of the signal's length. Our approach   utilizes a $K$-channel filter bank,  where the  $i$-th input channel is based on the $q_i$-th Ramanujan sum $c_{q_i} \in \ell^2(\mathbb Z_N)$ for $1 \leq i \leq K, $ and \(q_1, q_2, \dots, q_K\) are all divisors of \(N\).  
		The analysis phase of this $K$-channel filter bank,  based on  the Ramanujan sums \(c_{q_1}, c_{q_2}, \hdots, c_{q_K}\),  is depicted in Figure~\ref{Filter bank}. Furthermore, we study the perfect reconstruction capabilities of these filter banks under decimation. 
		The motivation for studying decimated Ramanujan filter banks arises from the broader theory of  decimated filter banks and their characterization via frame theory
		~\cite{cvetkovic1998oversampled, fickus2005convolutional,fickusfusion2011}.

	Let \( N = pM \) for some \( p, M \in \mathbb{N} \), and fix \( 1 \leq i \leq K \). For a signal \( x \in \ell^2(\mathbb{Z}_N) \), the output of the \( q_i \)-th Ramanujan filter is given by the discrete circular convolution \( x * \tilde{c}_{q_i} \), where \( \tilde{c}_q(n) := \overline{c_q(-n)} \) denotes the involution of \( c_q \). Sampling this output at indices \( \{0, p, 2p, \ldots, p(M-1)\} \) corresponds to evaluating inner products with translated filters:
	\begin{equation}\label{eq:outputintro}
	y_{q_i}=	(x * \tilde{c}_{q_i})(pk) = \langle x, L_{pk} c_{q_i} \rangle, \quad 0 \leq k \leq M-1.
	\end{equation}
	For $m \in \mathbb Z,$ the \textit{circular shift/translation operator} $L_m$   on $\ell^2(\mathbb Z_N)$ is defined by $L_mx(n)=x(n-m)$ for all $n \in \mathbb Z$ and $x \in \ell^2(\mathbb Z_N).$ 
	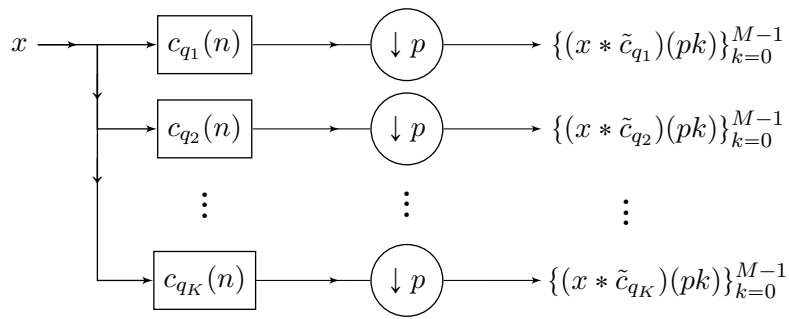
\begin{figure}[!b]\label{fig:rfb}
		\begin{center}
			\begin{tikzpicture}[auto,>=latex',
				mynode/.style={rectangle,fill=white,anchor=center}, scale=1.12]
				\tikzstyle{block} = [draw, shape=rectangle, minimum height=2em, minimum width=2em, node distance=1cm, line width=0.5pt]
				\tikzstyle{circ} = [draw, shape=circle,  node distance=1cm, line width=0.5pt]
				\tikzstyle{sum} = [draw, shape=circle, node distance=1.5cm, line width=0.5pt, minimum width=1.25em]
				\tikzstyle{branch}=[fill,shape=circle,minimum size=0pt,inner sep=0pt]
				\node at (-2.2,1) (input) {$x$};  
				\node at (5.5,1) {$\{(x * \tilde{c}_{q_1})(pk)\}_{k=0}^{M-1}$};
				\node at (5.5,0) {$\{(x * \tilde{c}_{q_2})(pk)\}_{k=0}^{M-1}$};
				\node at ($(5,0)!.5!(5,-1.8)$){\bf{\vdots}};
				\node at (5.5,-1.8) {$\{(x * \tilde{c}_{q_K})(pk)\}_{k=0}^{M-1}$};
				
				\node [block] at (0,1) (h1) {$c_{q_1}(n)$};
				\draw[->] (0.57,1) -- (1.6,1);  
				\node [block] at (0,0) (h2)  {$c_{q_2}(n)$};
				\draw[->] (0.57,0) -- (1.6,0);  
				\node [block] at (0,-1.8) (hs) {$c_{q_K}(n)$};
				\draw[->] (0.596,-1.8) -- (1.6,-1.8);  
				\node at ($(0,0)!.45!(0,-1.8)$){\bf{\vdots}};
				\node at ($(2.4,0.1)!.5!(2.4,-1.7)$){\bf{\vdots}};
				
				\node [circ, minimum size=1.5em] at (2.4,1) (down1) {$\downarrow p$};
				\draw[-] (1.5,1) -- (down1);
				\draw[->] (down1) -- (4,1);
				
				\node [circ, minimum size=1.5em] at (2.4,0) (down2) {$\downarrow p$};
				\draw[-] (1.5,0) -- (down2);
				\draw[->] (down2) -- (4,0);
				
				\node [circ, minimum size=1.5em] at (2.4,-1.8) (downs) {$\downarrow p$};
				\draw[-] (1.5,-1.8) -- (downs);
				\draw[->] (downs) -- (4,-1.8); 
				
				\path (input) -- coordinate(branch1) (h2);
				\begin{scope}[line width=0.5pt]
					\draw[->] (input) -- (-1.5,1); 
					\node at (-1.273, 0.5) {$\downarrow$};
					\node at (-1.273, -0.5) {$\downarrow$};
					\draw[-] (branch1) node[branch] {} |- (input);
					\draw[->] (branch1) node[branch] {} |- (h2);
					\draw[->] (branch1) node[branch] {} |- (h1);
					\draw[->] (branch1) node[branch] {} |- (hs);
				\end{scope}
			\end{tikzpicture}
		\end{center}
		\captionsetup{width=0.9\textwidth}
		\caption{\label{Filter bank}  Analysis phase of the $K$-channel  filter bank with decimation ratio $p$.}
	\end{figure}
	In light of frame theory and the equation \eqref{eq:outputintro},  it can be observed that the
	filter bank based on Ramanujan sums  $c_{q_1}, c_{q_2}, \ldots, c_{q_K}$ possesses the perfect reconstruction property provided that the
	collection
	\begin{equation}\label{eq:rpn}
		\mc R_{p,N}:=\{L_{pk}c_{q_1}\}_{k=0}^{M-1} \cup \{L_{pk}c_{q_2}\}_{k=0}^{M-1} \cup \hdots \cup \{L_{pk}c_{q_K}\}_{k=0}^{M-1},
	\end{equation}  is a \textit{frame} for $\ell^2(\mathbb Z_N)$, i.e., there exist constants $A,B > 0$  (called  \textit{frame bounds})  such that
	\begin{equation*}\label{eq: frameineq}
		A \|x\|^2 \leq \sum_{i=1}^K\sum_{k=0}^{M-1} |\langle x, L_{pk}c_{q_i} \rangle|^2 \leq B \|x\|^2,\,\, \text{ for all $x \in \ell^2(\mathbb Z_N).$}
	\end{equation*}
	If  $A = B,$ the frame is called a \textit{tight frame} with bound $A.$
	Filter banks with the same decimation ratio across each channel are called \textit{uniform filter banks.} Otherwise, they are called \textit{non-uniform filter banks.}   
	The collection $\mc R_{p,N}$ can also be seen as a particular case of a \textit{dynamical system }
	\begin{equation}\label{eq:dynamicalsystem}
		\{T^{pk}c_{q_i}: 0\leq k \leq M-1, \,\,1 \leq i\leq K\},
	\end{equation}
	where $T^mx(n)=L_mx(n)$ for $m,n \in \mathbb Z$ and $x\in \ell^2(\mathbb Z_N).$ Dynamical systems involve recovering signals from space-time samples, which corresponds to identifying conditions under which the system \eqref{eq:dynamicalsystem} forms a frame for \( \ell^2(\mathbb{Z}_N) \). For various instances of the dynamical sampling problem, see, \cite{aceska2013dynamical, aldroubi2013dynamical, aldroubi2017dynamical}.

	Our approach to proving  that the collection \( \mc R_{p,N} \) forms a frame for $\ell^2(\mathbb Z_N)$ relies on analyzing the analysis polyphase matrix whose entries are determined by the Zak transform of the Ramanujan sums. Specifically, the collection \( \mc R_{p,N} \) forms a frame for $\ell^2(\mathbb Z_N)$  if the analysis polyphase matrix satisfies a certain rank condition (see Lemma \ref{lem:mainlem}).  In the uniform filter bank setting, we prove in Theorem \ref{thm:tight} that $\mc R_{1,N}$ in general and $\mc R_{2,N}$ under some conditions  form  tight frames for $\ell^2(\mathbb Z_N).$  The same result also shows that $\mc R_{p,N},$ for $p>2,$ does not form a frame for $\ell^2(\mathbb Z_N).$ This situation can be handled in the non-uniform filter bank setting. Specifically, we prove in Theorem \ref{thm: rpkframe} that by updating the decimation ratios for appropriately selected input channels in the uniform filter bank, originally designed with a fixed decimation ratio \(p > 2\), the resulting non-uniform filter bank configuration exhibits frame properties. 
	The process of identifying the input channels that require decimation ratio updates is guided by an orthogonal decomposition of \( \ell^2(\mathbb{Z}_N) \), given by
	\[
	\ell^2(\mathbb{Z}_N) = S_{r,q_1} \oplus S_{r,q_2} \oplus \dots \oplus S_{r,q_K},
	\]  
	where for $1 \leq i \leq K$ and  appropriately chosen values of \( r \), \( S_{r,q_i} \) is a  \( \phi(q_i) \)-dimensional subspace of \( \ell^2(\mathbb{Z}_N) \), generated by \( N/r \)-consecutive \( r \)-shifts of the Ramanujan sum \( c_{q_i} \) (see Theorem \ref{thm:basis}). This decomposition leads to the construction of orthogonal frames based on Ramanujan sums. 
	For general Hilbert spaces, such orthogonal decompositions are widely studied in the literature and are fundamental in applications including signal processing, computer graphics, robotics, MRA, and super-wavelets, due to their structural and computational benefits \cite{han2000frames, balan2000multiplexing, balan1999density, gumber2018orthogonality, redhu2025construction,  shukla2018super, grochenig2009gabor}.
	
We further extend our investigation to the robustness of filter banks based on Ramanujan sums \( c_{q_1}, c_{q_2}, \ldots, c_{q_K} \) under channel erasures. The robustness of filter banks against channel erasures caused by noise or transmission failures is a crucial aspect of signal processing.  Prior work has established conditions under which filter banks retain their frame properties after the loss of one or more channels~\cite{kovacevic2002filter, xu2007robustness, casazza2003equal, goyal2001quantized, holmes2004optimal}. In contrast, we show that the filter bank based on Ramanujan sums fails to retain its frame properties even after the erasure of a single channel (see Theorem~\ref{thm:filtererasure}). This highlights the need to explore the resilience of these filter banks to the loss of frame coefficients during transmission.
	 
	  For effective reconstruction after erasures—where certain frame coefficients \(\{\langle x, L_{pk}c_{q_i} \rangle\}_{(k,i) \in \mathcal{I}}\) are lost, with \(\mathcal{I}\) denoting the index set of erasures—it is essential that the remaining collection \(\{L_{pk}c_{q_i}\}_{(k,i) \in \mathcal{I}^c}\) forms a frame, where \(\mathcal{I}^c\) denotes the complement of \(\mathcal{I}\). While uniform tight frames are robust to one erasure, general tight frames may fail under such loss~\cite{goyal2001quantized}. In Theorem~\ref{thm:erasure1}, we show that the tight frame \(\mathcal{R}_{p,N}\) is robust to one erasure, despite not being uniform. Furthermore, Theorem~\ref{thm:erasure2} provides explicit conditions under which \(\mathcal{R}_{p,N}\) is robust to two erasures.  We also examine the robustness of the Ramanujan systems from a subspace perspective by exploring their fusion frame structure. In particular, we show that the tight frame systems $\mathcal{R}_{p,N}$ admit interpretations as Parseval fusion frames (see Theorem \ref{thm:fusion}). This perspective provides a framework for analyzing robustness to erasures at the subspace level, which reinforces and extends the earlier robustness guarantees (see Theorem \ref{thm:fusionerasure}).

	In the further study, uniform filter banks based on Ramanujan sums exhibiting tight frames $\mc R_{p,N}$ are shown to have applications in signal recovery in  scenarios where some information from the signal is lost, corresponding to certain members of the tight frame  $\mc R_{p,N}$ involved in the frame expansion of \(x\), and where the signal is contaminated by noise, i.e., the noisy signal becomes \(y = x + \eta\), where \(\eta\) represents the noise. These scenarios are inspired by the seminal work of Donoho and Stark \cite{donoho1989uncertainty}, where the uncertainty principle plays a fundamental role in the exact recovery of a signal.  We show that an uncertainty principle concerning pairs of representations of signals  with respect to the canonical basis $\left\{ e_n \right\}_{n\in \mathbb Z_n}$  and tight frame $\mc R_{p,N}$  for $\ell^2(\mathbb Z_N)$ provides the necessary conditions for the signal recovery in both the scenarios (see, Theorems \ref{thm: uncertaintyrecovery}--\ref{thm:recovfromnoisy}).
	The uncertainty principle for pairs of bases has been extensively studied and has significant applications in sparse signal representations (see, e.g., 
	\cite{donoho2001uncertainty, elad2002generalized, ghobber2011uncertainty,  Kutyniok_2012,Grochenig2003}).
	
The signal recovery conditions provided in Theorems \ref{thm: uncertaintyrecovery}--\ref{thm:recovfromnoisy} may not always hold  due to significant missing data or when the signal is heavily contaminated by noise.  In these situations, 	the filter bank based on Ramanujan sums can help in reconstructing the original signal. The novelty of our work lies in incorporating the periodicity information of signals using filter banks based on Ramanujan sums to further optimize the recovery process. This allows for better reconstruction in cases of significant missing data or when the signal is heavily contaminated by noise.

In the case of missing
	information, the periodic components of the original signal utilizing the periodicity information of the original signal through a filter bank based on Ramanujan sums adds valuable constraints to the optimization problem. This modification results in reconstructions with improved signal-to-noise ratio (SNR) gains compared to those achieved without incorporating periodicity constraints.
For denoising,  we aim for recovering the signal \(x \in \ell^2(\mathbb{Z}_N)\) from the noisy observation \(y = x + \eta\), where \(\eta\) is the noise vector.  
Under modest assumptions on the noise, the Ramanujan sums $c_{q_i}$ will produce a high level of energy only when $q_i$ divides the period of the signal. By applying a  threshold to the energy values produced across all the channels, the significant periodic components of the original signal can be identified.
	This gives an optimized recovery algorithm that minimizes the error between the noisy and denoised signals, leading to improved SNR gain and better signal recovery.
	

Ramanujan filter banks have been employed for denoising periodic signals in~\cite{kulkarni2021, kulkarni2023periodicity}. While the approaches in these works and ours share the common use of the Ramanujan filter bank outputs, they differ significantly in the reconstruction strategy and the  filter bank design. In~\cite{kulkarni2021}, a synthesis dictionary is constructed from the filter bank output, and a regularized optimization is solved to obtain the denoised signal. In contrast, our method filters the noisy signal \( y = x + \eta \) using channels indexed by Ramanujan sums \( c_{q_i} \), \( 1\leq i \leq K \), assuming the true period divides \( N \). The output energies are used to identify a relevant subset of periods, which in turn guides a modification of the optimization framework. The resulting constrained problem yields the denoised signal.

The remainder of the paper is organized as follows. Section~\ref{S2} analyzes the frame properties of Ramanujan filter banks and their applications to period identification, non-uniform designs, robustness to erasures and their fusion frame properties. Section~\ref{S3} develops an uncertainty principle  for signal recovery in \( \ell^2(\mathbb{Z}_N) \), and Section~\ref{S4} presents supporting numerical experiments.
	

	%
	\section{Ramanujan Sums in Signal processing} \label{S2}
	
	Ramanujan sums have found increasing relevance in the field of signal processing due to their various properties, particularly,  to capture periodicities in discrete-time signals.  A significant application of Ramanujan sums is in the construction of \textit{Ramanujan filter banks}. These filter banks use Ramanujan sums as analysis filters, providing a mechanism for detecting and tracking periodic patterns in time series data \cite{ planat2002ramanujan, planat2009ramanujan}. Recall
	the $q$-th Ramanujan sum $c_q,$ defined in \eqref{eq:Ramanujansum}, i.e.,   the  sum of the  $n$th powers of all the $q$-th primitive roots of unity:
	\begin{equation*}\label{eq:2Ramanujansum}
		c_q(n)=\sum_{\substack{k=1\\(k,q)= 1}}^q e^{2\pi i k n /q}, \,\,\, n \geq 1.
	\end{equation*}
	For example, if $q=5,$ then $(k,q)=1$ for $k \in \{1,2,3,4\},$ thus 
	$$c_{5}(n)=e^{2\pi i n /5} + e^{2\pi i 2n /5} + e^{2\pi i 3n /5} + e^{2\pi i 4 n /5},\,\,\, n\geq 1.$$
	Note that
	$c_q$ is $q$-periodic, i.e., 	$c_q(n+q) = c_q(n), \,n \geq 1$ and  
	$c_q(0) = \phi(q),$ where the  \textit{Euler's totient function} $\phi(q)$ counts the positive integers that are coprime to $q.$
	For more properties of Ramanujan sums, we refer to \cite{ramanujan1918certain,vaidyanathan2014ramanujan1,vaidyanathan2014ramanujan2}.
	The first few Ramanujan sums are listed below:

		\begin{minipage}{0.6\textwidth}
	\begin{equation*}
			\begin{array}{ll}
				c_1 &= \{1, 1, \hdots\}; \\
				c_2 &= \{1, -1, 1, -1, \hdots\}; \\
				c_3 &= \{2, -1, -1, 2, \hdots\}; \\
				c_4 &= \{2, 0, -2, 0, 2, \hdots\}; \\
				c_5 &= \{4, -1, -1, -1, -1, 4, \hdots\}; \\
				c_6 &= \{2, 1, -1, -2, -1, 1, 2, \hdots\}.
			\end{array}
		\end{equation*}
		\end{minipage}%
		\hspace{0.5cm}
		\begin{minipage}{0.301\textwidth}
			\includegraphics[width=0.69\textwidth]{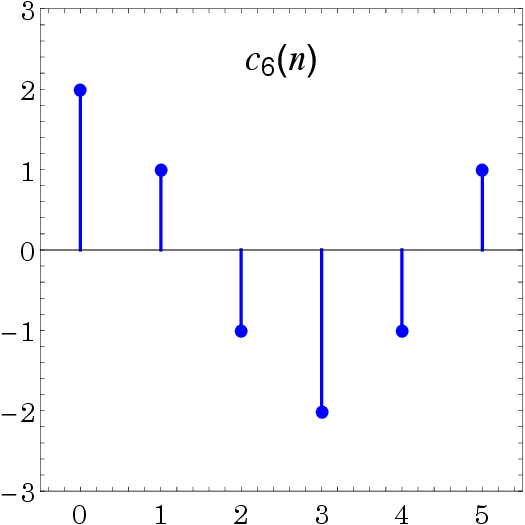}
		\end{minipage}
		
		\noindent The \( \phi(q) \) successively shifted $q$-th Ramanujan sum $c_q,$
		i.e., $
		\{c_q, L_1c_q, \ldots, L_{\phi(q)-1}c_q\},$
		are linearly independent and thus generate a $\phi(q)$-dimensional space, known as the  \textit{Ramanujan subspace}, denoted by  $S_q,$  where for $m \in \mathbb Z,$ the operator $L_m$ is the \textit{circular shift/translation operator}   on $\ell^2(\mathbb Z_N).$    These subspaces enable the development of Ramanujan periodicity dictionaries for representing and extracting periodic components of signals. 
		For example, if \( q =5 \) then \( \phi(5)=4 \), and the four sequences which form the basis for \( S_{5} \) are
	\begin{align*}
			c_{5}&=  \{4, -1, -1, -1, -1, 4, \hdots\},\,\,
			c_{5}(\cdot- 1) = \{-1,4,-1,-1,-1,-1,4,\hdots\}, \\
			c_{5}(\cdot - 2) &= \{-1,-1,4,-1,-1,-1,\hdots\},\,\,
			c_{5}(\cdot- 3) = \{-1,-1,-1,4,-1,-1,-1,\hdots\}.
		\end{align*}
		The  Ramanujan subspaces can be associated with the  $q \times q$  circulant matrix shown below:
		\begin{equation*}
			\mc {C}_q: = 
			\left[
			\begin{array}{cccccc}
				c_q(0) & c_q(q-1) & c_q(q-2) & \cdots & c_q(1) \\
				c_q(1) & c_q(0) & c_q(q-1) & \cdots & c_q(2) \\
				c_q(2) & c_q(1) & c_q(0) & \cdots & c_q(3) \\
				\vdots & \vdots & \vdots & \ddots & \vdots \\
				c_q(q-2) & c_q(q-3) & c_q(q-4) & \cdots & c_q(q-1) \\
				c_q(q-1) & c_q(q-2) & c_q(q-3) & \cdots & c_q(0)
			\end{array}
			\right].
		\end{equation*}
		In fact,  $\mc {C}_q$ has rank $\phi(q)$ and the first $\phi(q),$ columns (or any $\phi(q)$ consecutive columns) are linearly independent and form a basis for $S_q.$ Thus, the column space of $\mc {C}_q$ is the Ramanujan subspace $S_q$.
		
	We summarize below some essential properties of the Ramanujan sums that will be used in the sequel.
\begin{prop} \cite{vaidyanathan2014ramanujan1}\label{prop:properties}
			(Some properties of the Ramanujan sums). For any positive integer \( q \), the Ramanujan sum \( c_q \), as defined in~\eqref{eq:Ramanujansum}, satisfies the following properties:
			\begin{itemize}
				\item [$(i)$]  \( c_q(n) \) is an integer and \( c_q(n) \leq \phi(q) \) for any  \( n \).
				\item [$(ii)$]\( c_q(n) = c_q(-n) \) for any  \( n \).
				\item [$(iii)$] (Autocorrelation) 
				$
				\sum_{n=0}^{q-1} c_q(n)c_q(n-l) = q c_q(l) \quad \text{for any } l \in \mathbb{Z}.
				$
				\item[$(iv)$] (Orthogonality) 
				$
				\sum_{n=0}^{\operatorname{lcm}(p,q)-1} c_{q'}(n)c_q(n-l) = 0 \quad \text{for any } l \in \mathbb{Z} \text{ when } q' \neq q.
				$
					\item[$(v)$] (Sum) $\sum_{n=0}^{q-1} c_q(n) =0$ for $q>1$.
				\item[$(vi)$] (Sum of squares) $\sum_{n=0}^{q-1} c_q^2(n) =q\phi(q).$
			\end{itemize}
		\end{prop}

		In this section, we use the  properties of Ramanujan sums and their associated subspaces to  examine the  reconstruction capabilities of the Ramanujan filter bank and its robustness against erasures.
		In the following subsection, we demonstrate how these filter banks are used to effectively identify  the period of an unknown signal.

	\subsection{Period identification using Ramanujan sums}\label{sub:2period}
		The use of filter banks based on Ramanujan sums for identifying the period of a signal was developed by Vaidyanathan and Tenneti in \cite{vaidyanathan2015properties}. Their work provides a foundational framework for this task. Specifically, Theorem~1 in \cite{vaidyanathan2015properties} identifies which filters in the bank will respond to a periodic input signal.

		\begin{thm}[{\cite[Theorem 1]{vaidyanathan2015properties}}]\label{thm:vaidyanathanthm1}
			Consider a filter bank based on Ramanujan sums $c_1, c_2, \hdots, c_M$ and let $x$ be a period-$P$ input signal with $1 \leq P \leq M$. Then nonzero outputs can only be produced by those filters $c_q,$ $1 \leq q \leq M$, such that the filter index $q$ is a divisor of $P$, that is, $q \mid P$.
		\end{thm}
		
		Furthermore, Theorem~3 of \cite{vaidyanathan2015properties} establishes how the period can be precisely recovered from the set of responding filters.
		
		\begin{thm}[{\cite[Theorem 3]{vaidyanathan2015properties}}]\label{thm:vaidyanathanthm3}
		Consider a filter bank based on Ramanujan sums $c_1, c_2, \hdots, c_M$ and let $x$ be a period-$P$ input signal with $1 \leq P \leq M$. Let nonzero outputs be produced by the subset of filters $c_{q_i}$ with periods $q_1, q_2, \ldots, q_K$. Then the period $P$ is given by
			\[
			P = \mathrm{lcm} \{q_1, q_2, \ldots, q_K\}.
			\]
		\end{thm}
		
		These results show that a filter bank based on Ramanujan sums isolates only those components corresponding to divisors of the period, and the original period can be recovered as the least common multiple of the indices of the non-zero output filters. However, when the period $P$ is large, considering all Ramanujan sums $c_q$ for $1 \leq q \leq M$ with $1\leq P \leq M$ may increase computational complexity and reduce the efficiency of the period identification process. In the special case where $x \in \ell^2(\mathbb Z_N)$ and  $P$ is a known divisor of $N$, the period identification process can be made significantly simpler by considering the filter bank based on Ramanujan sums $c_{q_1},c_{q_2}, \ldots, c_{q_K}$ corresponding to all the divisors $q_1,q_2, \ldots, q_K$ of $N.$

As a consequence of Theorem~2.2 and Theorem~2.3, we derive the following result, which serves as a corollary to these theorems. Assuming that the period \( P \) divides \( N \), the result shows that a reduced filter bank comprising only the Ramanujan sums corresponding to the divisors of \( N \) suffices to identify the period.
	\begin{cor} \label{thm: period} 
				Let \( x \in \ell^2(\mathbb{Z}_N) \) be a period-\( P \) input signal such that \( P \) divides $N$, and let \( q_1, q_2, \ldots, q_K \) be the divisors of \( N \). Then the nonzero outputs \( y_{q_i} \) (as defined in~\eqref{eq:outputintro}),  \( 1 \leq i \leq K \), must be generated by a subset of the Ramanujan sums \( c_{q_1}, c_{q_2}, \ldots, c_{q_m} \) with \( m \leq K \). Moreover, the period \( P \) is given by
				\(
				P = \mathrm{lcm}(q_1, q_2, \ldots, q_m).
				\)
		\end{cor}
		\begin{proof}
		The output \( y_{q_i}(n) \) corresponding to the \( q_i \)-th Ramanujan sum \( c_{q_i} \), for \( 1 \leq i \leq K \) and \( 0 \leq n \leq N-1 \), can be expressed as
		\begin{equation*}
			y_{q_i}(n) = (x * \tilde{c}_{q_i})(n) = \sum_{m=0}^{N-1} x(m) \tilde{c}_{q_i}(n - m) 
			= \sum_{m=0}^{N-1} x(m) \overline{L_n c_{q_i}(m)} = \langle x, L_n c_{q_i} \rangle.
		\end{equation*}
		It follows from \cite[Theorem 1]{vaidyanathan2015properties} that the output $y_{q_i}, 1 \leq i \leq K,$  can be non-zero only if $q_i \mid P.$ Since $P$ divides $N,$ then clearly $\{p_1, p_2, \hdots, p_m\} \subset \{q_1, q_2, \hdots, q_K\},$ where $p_1, p_2, \hdots, p_m$ are all the divisors of $P.$ 
			Therefore,  the filter bank based on  $c_{q_1}, c_{q_2}, \hdots, c_{q_K}$  contains all those Ramanujan sums, i.e.,  $c_{p_1}, c_{p_2}, \hdots, c_{p_m}$ for which $y_{p_i}, 1\leq i \leq m,$ can be non-zero.  Then,  $P$ is known precisely and is given by \( \text{lcm}(p_1, p_2, \dots, p_m).\)
		\end{proof}

		We provide an example to illustrate  Corollary \ref{thm: period}.
		\begin{exmp}\label{exmp1}
		Consider a signal \( x \in \ell^2(\mathbb{Z}_{30}) \), depicted in Figure~\ref{fig:period_signal}(a), with period $P\leq N=30$. 
		The divisors of \( N = 30 \) are given by \( q_1 = 1, q_2 = 2, q_3 = 3, q_4 = 5, q_5 = 6, q_6 = 10, q_7 = 15, q_8 = 30 \), so that \( K = 8 \).
		We construct the filter bank based on the Ramanujan sums \( c_1, c_2, c_3, c_5, c_6, c_{10}, c_{15}, \) and \( c_{30} \), in accordance with Corollary \ref{thm: period}.
		The output from the \( q_i \)-th Ramanujan filter is given by
		\[
		(x * \tilde{c}_{q_i})(k) = \langle x, L_{k} c_{q_i} \rangle, \quad 0 \leq k \leq 29,\; 1 \leq i \leq 8.
		\]
			%
			Figure~\ref{fig:period_signal}(b) illustrates the energy outputs produced by each channel after the signal $x$ is passed through it, where
			the energy $E_i$ corresponding to  the $q_i$-th Ramanujan is given by:
			\begin{equation*}
				E_i=\sum_{n=0}^{N-1}|x\, \asterisk\, \tilde{c}_{q_i}(n)|^2, \,\, 1 \leq i \leq K.
			\end{equation*} 
			From Figure~ \ref{fig:period_signal}(b), it is clear that the outputs corresponding to the Ramanujan sums $c_3,c_5$ and $c_{15}$ are non-zero. therefore by using Corollary \ref{thm: period},  we get $P=\text{lcm}(3,5,15)=15.$
			\begin{figure}[!h]
				\begin{center}
					\centering
					\includegraphics[width=0.313\textwidth]{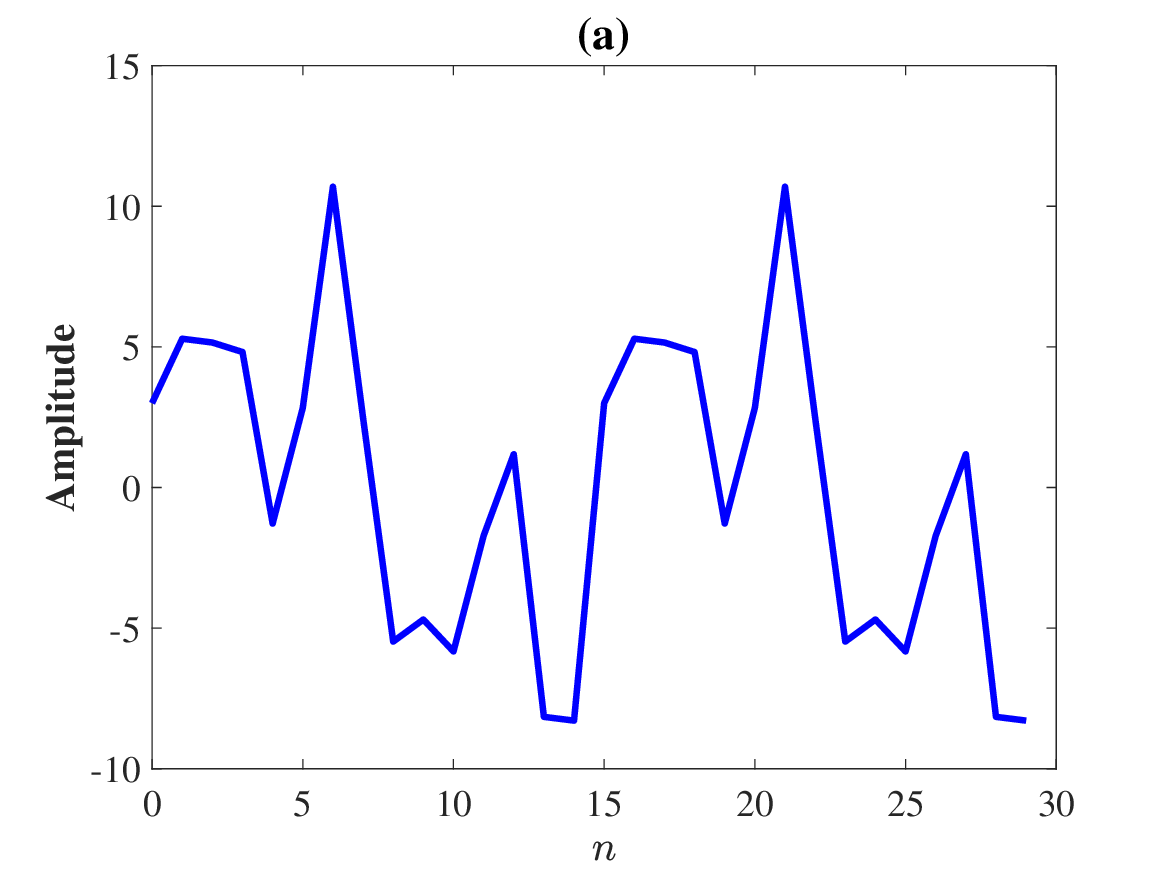}
					\quad
					\includegraphics[width=0.3133\textwidth]{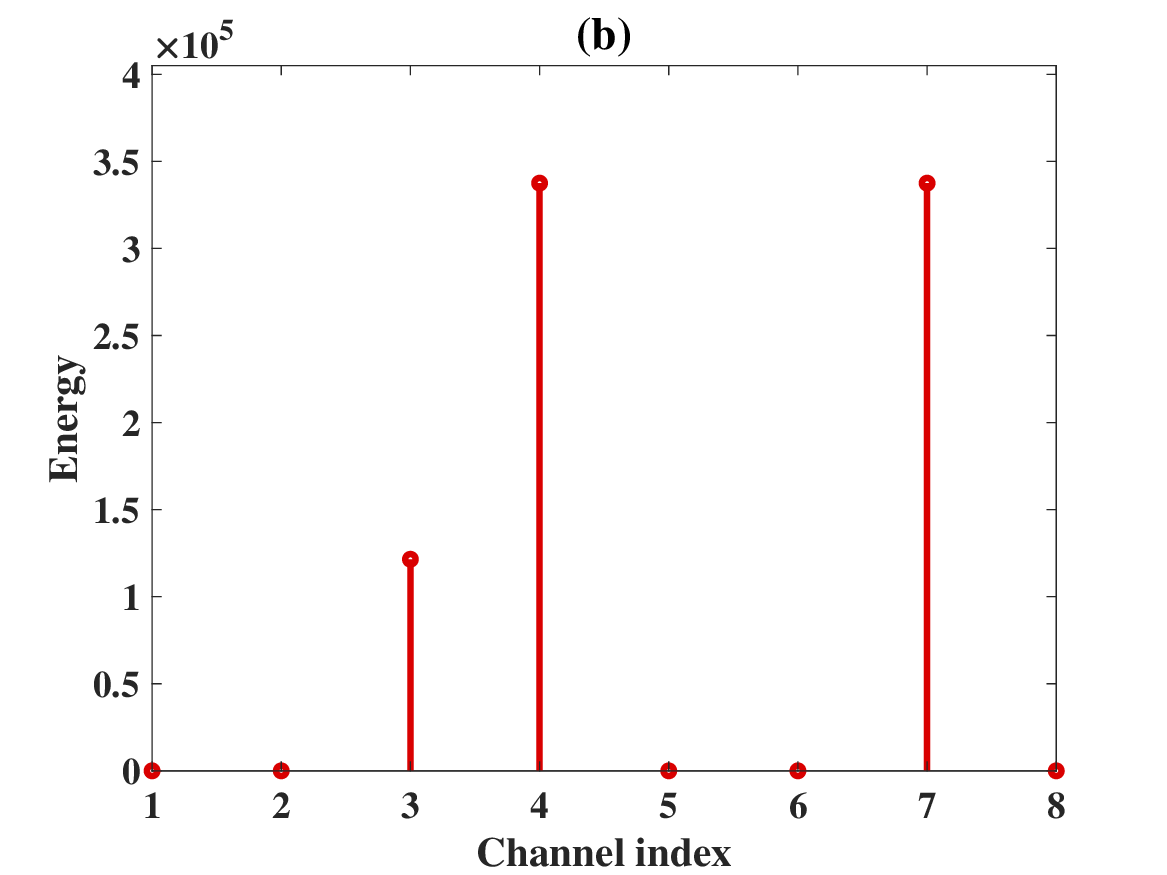}
					\quad
					\includegraphics[width=0.3133\textwidth]{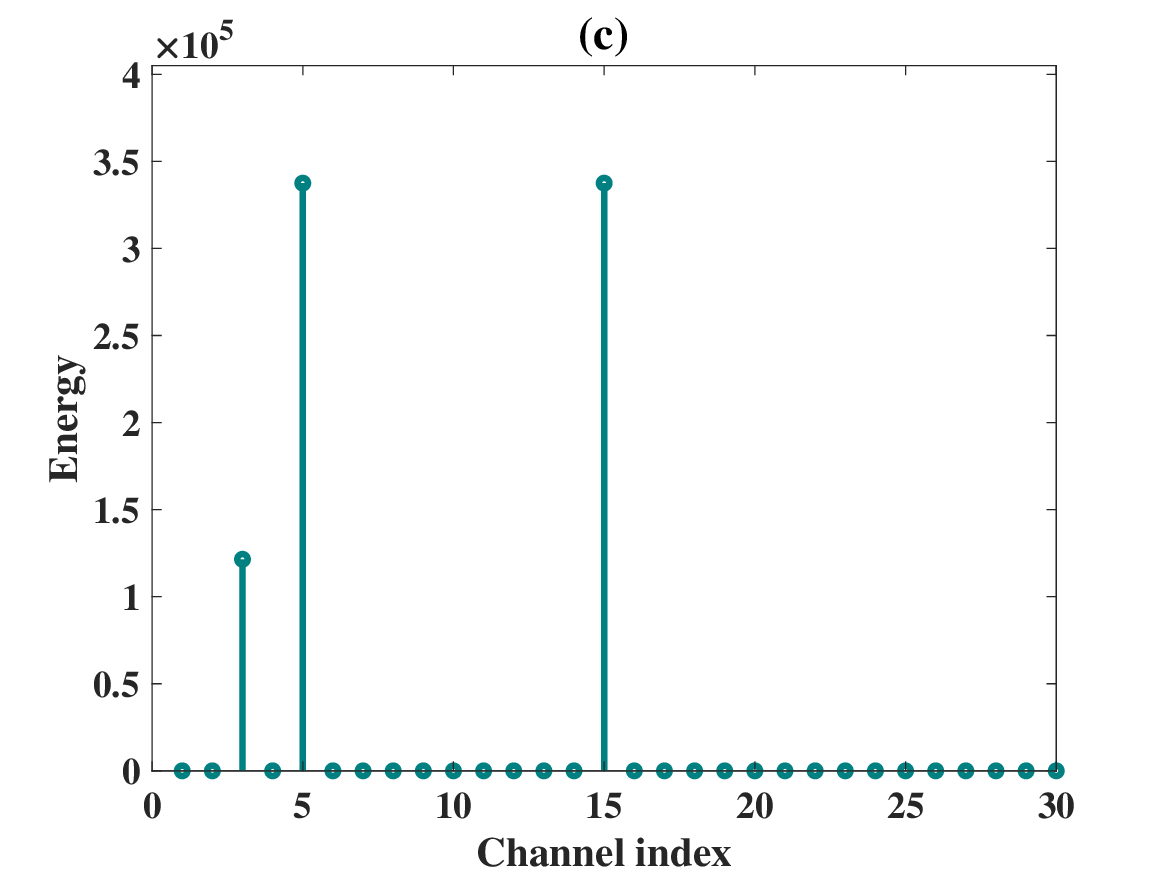}
					\captionsetup{width=0.88\textwidth}
					\caption{
						(a) Signal of length $N=30$ with periodic components 3, 5, and 15, (b) energy outputs of the signal $x$ using Ramanujan sums $c_1,c_2,c_3,c_5,c_6,c_{10},c_{15}$ and $c_{30},$ (according to Corollary \ref{thm: period}) and (c) energy outputs of the signal $x$ using Ramanujan sums $c_1$ to $c_{30}$ (according to Theorem \ref{thm:vaidyanathanthm1}).}
					\label{fig:period_signal}
				\end{center}
			\end{figure}

		On the other hand, identifying \( P \), as described in Theorem~\ref{thm:vaidyanathanthm3}, requires a filter bank based on Ramanujan sums \( c_1, c_2, c_3, \dots, c_{30} \), whose energy outputs are shown in Figure~\ref{fig:period_signal}(c). 
		However, Corollary~\ref{thm: period} shows that the filter bank constructed using only the Ramanujan sums \( c_1, c_2, c_3, c_5, c_6, c_{10}, c_{15}, \) and \( c_{30} \) is sufficient to identify \( P \).
		\end{exmp}
		
			Theorem \ref{thm: period} provides a practical framework for determining the period \( P \) of an input signal using a filter bank composed of Ramanujan sums corresponding to the divisors of \( N \). This highlights the significance of such  filter banks  in signal processing. Consequently,  studying the frame properties  of the collection \(\mathcal{R}_{p,N}\) (defined in \eqref{eq:rpn}), generated  from such  filter banks with a decimation ratio \( p \), becomes essential as it ensures the perfect reconstruction capability of these filter banks.
			
			In the next subsection, we demonstrate that, for  certain  choices of $N$ and the decimation ratio $p,$ the collection $\mc R_{p,N}$ forms a tight frame for $\ell^2(\mathbb Z_N).$

			\subsection{Tight frames involving Ramanujan sums and erasures}\label{sub:2tight}						
			
			Tight frames are useful in applications such as compressed sensing, denoising, and data compression, providing a versatile set of tools for engineers and researchers. See, \cite{donoho2006compressed, shen2006image, cai2014data, benedetto2003finite, christensen2016frames}.
			In this subsection, we prove that \(\mathcal{R}_{1,N}\), in general, and \(\mathcal{R}_{2,N}\), under certain conditions, form a tight frame for \(\ell^2(\mathbb{Z}_N)\) (Theorem \ref{thm:tight}). The same result also establishes that \(\mathcal{R}_{p,N}\), for \(p > 2\), does not form a frame for \(\ell^2(\mathbb{Z}_N)\).

	Throughout this paper, we fix a positive integer \(N = pd\) for some positive integers \(p\) and \(d\), and let \(q_1, q_2, \dots, q_K\) denote all the positive divisors of \(N\), where \(K\) is the number of such divisors. For each \(i \in \mathcal{I}_K\), we fix \(c_{q_i} \in \ell^2(\mathbb{Z}_N)\) to be the Ramanujan sum of length \(N\). We also set \(\mathcal{I}_m := \{1, 2, \dots, m\}\) and \(\mathcal{J}_n := \{0, 1, \dots, n-1\}\) for positive integers \(m\) and \(n\).
			
			%
			%
			We prove the following main result of this section:
			

			\begin{thm}\label{thm:tight}
			Let the system
		\begin{equation}\label{eq:rpk}
			\mc R_{p,N} := 
			\{L_{p k} c_{q_1}\}_{k=0}^{d-1} \cup
			\{L_{p k} c_{q_2}\}_{k=0}^{d-1} \cup
			\cdots \cup
			\{L_{p k} c_{q_K}\}_{k=0}^{d-1}.
		\end{equation}
			be generated from a filter bank based on the Ramanujan sums \(c_{q_1}, c_{q_2}, \ldots, c_{q_K}\), with decimation ratio \(p\). Then, under the assumption \(K \geq p\), the following statements hold:
				\begin{itemize}
					\item[$(i)$] 
					The  system $\mc R_{1,N}$ forms a tight frame for $\ell^2(\mathbb Z_N)$ with bound $N^2.$   
					\item[$(ii)$]   For odd $d, $ the  system $\mc R_{2,N}$ forms a tight frame for $\ell^2(\mathbb Z_N)$ with bound $2d^2.$
					Moreover, in the case of even $d,$ $\mc R_{2,N}$ does not form a frame for $\ell^2(\mathbb Z_N).$
					\item[$(iii)$]   For any $p >2,$ the collection $\mc R_{p,N}$ does not form a frame for $\ell^2(\mathbb Z_N).$  
				\end{itemize}
			\end{thm}

			The condition \( K \geq p \) in the hypothesis of Theorem \ref{thm:tight} is justified by the fact that the total number of elements in \( \mathcal{R}_{p,N} \) is \( Kd \). According to frame theory, for \( \mathcal{R}_{p,N} \) to constitute a frame for \( \ell^2(\mathbb{Z}_N) \), it is necessary that \( Kd \geq N \), which implies \( K \geq p \).

			%
			

	 For each $x \in \mathbb \ell^2(\mathbb Z_N),$ consider the outputs  $\{(x\, \asterisk \,\tilde{c}_{q_i})(pk): k\in \mc J_{d}, \,i \in \mc I_K\}$ from the filter bank based on Ramanujan sums $\{c_{q_1}, c_{q_2}, \cdots, c_{q_K}\}\subset \ell^2(\mathbb Z_N)$.
			We now simplify the outputs using   the Zak transform. The \textit{Zak transform} on $\ell^2(\mathbb Z_N)$ is the map $\mc Z: \ell^2(\mathbb Z_N) \rightarrow \ell^2(\mathbb Z_d \times \mathbb Z_p)$ and is defined by the formula:
		\begin{equation}\label{eq:Zak}
			(\mc Z x)(m,n) = \frac{1}{\sqrt{d}} \sum_{\ell=0}^{d-1} x(p \ell + n) \, e^{-2\pi i m\ell/d}, 
			\quad \text{for all } x \in \ell^2(\mathbb{Z}_N), \; m \in \mathbb{Z}_d, \; n \in \mathbb{Z}_p.
		\end{equation}
			It is also worth noting that the Zak  transform \eqref{eq:Zak} can be obtained as a particular case of the Zak transform for LCA groups (see, for instance, Refs. \cite{barbieri2015zak, kaniuth1998zeros}). 
			The Zak transform $\mc Z: \ell^2(\mathbb Z_N) \rightarrow \ell^2(\mathbb Z_d \times \mathbb Z_p),$ as defined in \eqref{eq:Zak}, is a unitary map on $\ell^2(\mathbb Z_N)$
			and satisfies the intertwining relation:
		\begin{equation}\label{S4:intertwining}
			\mc Z(L_{pk} x) = \frac{1}{\sqrt{d}} \, e^{-2\pi i k \cdot / d} \, \mc Z x, 
			\quad \text{for all } x \in \ell^2(\mathbb{Z}_N) \text{ and } k \in \mc J_d.
		\end{equation}
			Let us now define a $K \times p$ matrix $\mc U(m)$ for each $m\in \mc J_d,$ given by
			\begin{equation}\label{eq:Uk}
				\mc U(m)=\sqrt{d} \begin{pmatrix}
					\ol{\mc Zc_{q_1}(m,0)} && \ol{\mc Zc_{q_1}(m,1)} && \hdots && \ol{\mc Zc_{q_1}(m,p-1)}\\
					\ol{\mc Zc_{q_2}(m,0)} && 	\ol{\mc Zc_{q_2}(m,1)} && \hdots && 	\ol{\mc Zc_{q_2}(m,p-1)}\\
					\vdots && \vdots && \ddots&&\vdots\\
					\ol{\mc Zc_{q_K}(m,0)} && \ol{\mc Zc_{q_K}(m,1)} && \hdots && \ol{\mc Zc_{q_K}(m,p-1)}
				\end{pmatrix}.
			\end{equation}
			We call the matrix $\mc U(m)$  the  \textit{analysis polyphase matrix}.
			Let us denote  \begin{equation}\label{eq: AB}A:= \min_{k \in \mc J_d}\lambda_{\min}[\mc U^\asterisk(k)\mc U(k)] \quad \text{ and } \quad B:= \max_{k \in \mc J_d}\lambda_{\max}[\mc U^\asterisk(k)\mc U(k)],
			\end{equation} 
		where, for each $k \in \mc J_d$, $\mc U^\ast(k)$ denotes the conjugate transpose of $\mc U(k)$, while $\lambda_{\min}[\mc U^\ast(k)\mc U(k)]$ and $\lambda_{\max}[\mc U^\ast(k)\mc U(k)]$ denote, respectively, the smallest and largest eigenvalues of the matrix $\mc U^\ast(k)\mc U(k)$.
			
			The following lemma characterizes the frame property of the collection $\mc R_{p,N}$ in terms of the rank of the polyphase matrix \eqref{eq:Uk} and the sampling expansion of any signal $x \in \ell^2(\mathbb Z_N).$ Various versions of the following result can be found in \cite{bolcskei1998frame, cvetkovic1998oversampled}. For the sake of simplicity, we omit its proof.
			\begin{lem}\label{lem:mainlem}
		The following statements are equivalent:
				\begin{enumerate}
				\item[$(i)$]
				The system $\mc R_{p,N}$, defined in \eqref{eq:rpk}, forms a frame for $\ell^2(\mathbb{Z}_N)$ with frame bounds $A$ and $B$ as given in \eqref{eq: AB}.
					\item [$(ii)$] There exist $g_i \in \ell^2(\mathbb Z_N), i \in \mc I_K$ such that the collection $\{ L_{p k} g_i : k \in \mc J_d, \; i \in \mc I_K \}$ is a frame for $\ell^2(\mathbb Z_N)$ and the following sampling formula holds for any $x \in \ell^2(\mathbb Z_N):$
					\begin{equation*}\label{S3: samplingformulainRN}
						x=\sum_{i\in \mc I_K}\sum_{k\in \mc J_d}	(x\, \asterisk \,\tilde{c}_{q_i})(pk)L_{pk}g_i.
					\end{equation*}
					\item [$(iii)$] There exist  $h_i \in \ell^2(\mathbb Z_N), i\in \mc I_K,$ such that 
					\begin{equation*}\label{Gk}
						\begin{bmatrix}
							\mc Zh_1(m, \cdot) &  	\mc Zh_2(m, \cdot) & \cdots& 	\mc Zh_K(m,\cdot) 
						\end{bmatrix}\mc U(m)=I_p, \,\, \text{ $m\in \mc J_d,$}  
					\end{equation*}
					where $I_p$ denotes the identity matrix of order $p.$
					\item[$(iv)$] rank $\mc U(m)=p$ for each $m \in \mc J_d.$ 
				\end{enumerate}
			\end{lem}

			
			The following lemma is required in the sequel.
		\begin{lem}\label{lem:mknq}
			Let \( m \in \{1, 2, \ldots, N-1\} \). Then, there exists a divisor \( q \) of \( N \) such that \( m \) can be uniquely expressed as
			\(
			m = kN/q,
			\)
			where \( k \) is coprime to \( q \).
			\end{lem}
			\begin{proof}
			The least common multiple of \( m \) and \( N \) can be written as \( \mathrm{lcm}(m,N) = qm = kN \),
			where \( q \) and \( k \) are positive, coprime integers. Since \( q \) divides \( kN \) and \( (q,k) = 1 \), \( q \) divides \( N \).
			Thus \( m = kN/q \) is an expression of the desired form. The uniqueness of this expression follows
			from the fact that \( k/q \) is the unique representation of the rational number \( m/N \) in lowest terms.
			\end{proof} 
			\begin{prop}\label{prop: cqprop}
				Let \( N = 2d \) be a positive even integer, where \( d \) is odd. For $m \in \mc J_{N}$ and a divisor $q$ of $N,$ the following statements hold for the Ramanujan sum $c_q \in \ell^2(\mathbb Z_N):$
				\begin{enumerate}
					\item  [$(i)$] 
					$ 
					\sum_{n=0}^{N-1}c_{q}(n)e^{-2 \pi i mn/N} =
					\begin{cases} 
						N, & \text{if } m=\frac{kN}{q},\,  (k,q)=1, \\
						0, & \text{otherwise}.
					\end{cases}
					$
					\item [$(ii)$] $
					c_q\left(\frac{N}{2}+n\right)e^{-2\pi im\left(\frac{N}{2}+n\right)/N} = c_q(n)e^{-2\pi imn/N}, \,\, n \in \mathbb Z_N.$
					\vspace{2mm}
					\item  [$(iii)$] 
					$
					c_{q/2}(n) = (-1)^n c_q(n) 
					$ for $n \in \mathbb Z_N,$ provided $q$ is even.
				\end{enumerate}
			\end{prop}
			\begin{proof}
				We first prove (i).  Simplifying the left hand side and using \eqref{eq:Ramanujansum}, we have
				\begin{equation}\label{S4: dftcq}
					\begin{aligned}
						\sum_{n=0}^{N-1}c_{q}(n)e^{-2 \pi i mn/N}&=\sum_{n=0}^{N-1}\sum_{\substack{\ell=1 \\ (\ell,q)=1}}^qe^{2\pi i n\ell/q}e^{-2\pi imn/N}=\sum_{\substack{\ell=1 \\ (\ell,q)=1}}^q\bigg(\sum_{n=0}^{N-1}e^{2\pi i n(\ell/q-m/N)}\bigg).
					\end{aligned}
				\end{equation}
				Since $q$ divides $N,$ then we can write $N=kq$ for some $k \in \mathbb Z.$ Then, $e^{2\pi i n(\ell/q-m/N)}=e^{2\pi i n(\ell k-m)/N}.$	If $\ell k \neq m$ or $m \neq \ell N/q$ for any $\ell$ with $(\ell,q)=1,$ then $e^{2\pi i n(\ell k-m)/N} \neq 1$ for $1 \leq \ell k, m \leq N-1$ since $-N < \ell k-m <N.$
				Therefore the inner sum is the partial sum of a geometric series, so	
				$$\sum_{n=0}^{N-1}e^{2\pi i n(\ell k-m)/N}= \frac{1-e^{2\pi i (\ell k-m)N/N}}{1-e^{2\pi i (\ell k-m)/N}}=0,$$
				since $\ell k-m$ is an integer.	Now if, $\ell k = m$ or $m = \ell N/q$ for some $\ell$ with $(\ell,q)=1,$ then the  inner sum is $N.$ 
				By Lemma \ref{lem:mknq}, for a given $m$ and $q,$ the choice of $\ell$ is unique. Thus the value of the expression in \eqref{S4: dftcq} becomes: 
				$\sum_{n=0}^{N-1}c_q(n)e^{-2\pi i mn/N}=\sum_{n=0}^{N-1}1=N.$ 

				Now we prove (ii). 	By using Lemma \ref{lem:mknq} for a  given $m \in \{1, \hdots, N-1\}$, there exists a divisor $q$ of $N$ and a number $k$ coprime to $q$ such that $m=\frac{kN}{q}.$ On simplifying the left-hand side, we have
				\begin{equation}\label{S4: cqnby2}
					\begin{aligned}	c_q&\left(\frac{N}{2}+n\right)e^{-2\pi im\left(\frac{N}{2}+n\right)/N} =\sum_{\substack{\ell=1 \\  (\ell,q)=1}}^q e^{2\pi i \ell (\frac{N}{2}+n)/q}e^{-2\pi im\left(\frac{N}{2}+n\right)/N}=\sum_{\substack{\ell=1 \\  (\ell,q)=1}}^q e^{2\pi i  (\frac{N}{2}+n)(\frac{\ell}{q}-\frac{m}{N})}\\
						&=\sum_{\substack{\ell=1 \\  (\ell,q)=1}}^q e^{2\pi i  \frac{N}{2}(\frac{\ell N - mq}{qN})} e^{2\pi i n \ell/q}e^{-2\pi i m n /N}=\sum_{\substack{\ell=1 \\  (\ell,q)=1}}^q e^{\pi i  (\ell N/q - m)} e^{2\pi i n \ell/q}e^{-2\pi i m n /N}.
					\end{aligned}
				\end{equation}
				Now if $q$ is even, then $N/q$ being the product of odd primes is odd and $k$ being coprime to $q$ is odd. Therefore $m$ being the product of two odd numbers is odd. Similarly $\ell$ is also odd as $\ell$ is coprime to $q.$ Then for any  $\ell$ with $(\ell,q)=1,$ the term $\ell N/q-m$ is always even being the difference of two odd numbers. Finally, \eqref{S4: cqnby2} becomes
				$$c_q\left(\frac{N}{2}+n\right)e^{-2\pi im\left(\frac{N}{2}+n\right)/N}= 	(-1)^{m+1}c_q(n)e^{-2\pi imn/N}.$$
				Similarly, it can be easily shown that if $q$ is odd, then for any  $\ell$ with $(\ell,q)=1,$ the term $\ell N/q-m$ is always even being the difference of two even numbers. In this case also \eqref{S4: cqnby2} becomes:
				$$c_q\left(\frac{N}{2}+n\right)e^{-2\pi im\left(\frac{N}{2}+n\right)/N}= 	c_q(n)e^{-2\pi imn/N}.$$

				Finally, we prove (iii).
				Splitting the sum $c_{q/2}(n)=\sum_{\substack{k=1 \\  (k,q/2)=1}}^{q/2} e^{\frac{2 \pi i k n}{q/2}}$ into odd and even values of $k$  and simplifying, we get 
				\begin{equation}\label{S4: cqby2split}
					\begin{aligned}
						c_{q/2}(n)=\sum_{\substack{k=1 \\  (k,q/2)=1\\ \text{$k$ is odd}}}^{q/2} e^{\frac{2 \pi i k n}{q/2}} + \sum_{\substack{k=1 \\  (k,q/2)=1\\ \text{$k$ is even}}}^{q/2} e^{\frac{2 \pi i kn} {q/2}}=\sum_{\substack{k=1 \\  (k,q/2)=1\\ \text{$k$ is odd}}}^{q/2} e^{\frac{2 \pi i kn} {q/2}} + \sum_{\substack{k=1 \\  (k,q/2)=1\\ \text{$k$ is even}}}^{q/2} e^{\frac{2 \pi in (q/2 +k)} {q/2}}.
					\end{aligned}
				\end{equation}
				We first show that the set $B:=\{k : (k,q/2)=1 \text{ and $k$ is odd}\} \cup \{q/2+k : (k,q/2)=1 \text{ and $k$ is even}\}$ consists of all the numbers coprime to $q.$ If $k$ is odd with $(k,q/2)=1$ and $(k,q)=r>1.$ Then, $r \mid k $ and $r \mid q.$ Since  $k$ is odd and  $r \mid k,$ implies $r$ is odd. Then, $r \mid q$ implies $r \mid q/2$ as well. This contradicts the fact that $(k,q/2)=1.$ Therefore $r=1.$
				Similarly, if $k$ is even with $(k,q/2)=1$ and $(q/2+k,q)=r'>1.$ Then $r' \mid (q/2+k)$ and $r' \mid q.$ Also note that $q/2$ is odd since $N$ contains a single factor of $2.$ As $k$ is even  and $q/2$ is odd, then $q/2+k$ is odd. Therefore, $r' \mid (q/2+k)$ implies $r'$ is odd and hence $r' \mid q$ implies $r' \mid q/2.$ This leads to a contradiction, since $(k,q/2)=1.$ Thus, $r'=1.$
				
				Since $q=2p_1p_2\cdots p_n$  for some primes $p_i, 1\leq i \leq n$  other than 2,  therefore $\phi(q)=\phi(q/2) = |B|.$ This implies that the set  $B$ contains all the numbers coprime to $q.$ Using this fact in \eqref{S4: cqby2split}, we get $c_{q/2}(n)=c_{q}(2n).$ We further prove  $c_{q}(2n)=(-1)^nc_{q}(n)$ for which  the left hand side simplifies to
				\begin{equation}
					\begin{aligned}\label{S4: cq2nton}
						c_{q}(2n)&=\sum_{\substack{k=1 \\  (k,q)=1}}^{q} e^{2 \pi i k (2n)/q}= \sum_{\substack{k=1 \\  (k,q)=1}}^{q} e^{2 \pi i k n/q} e^{\frac{\pi i k n}{q/2}}\\
						&=(-1)^n \sum_{\substack{k=1 \\  (k,q)=1}}^{q} e^{2 \pi i k n/q} e^{\frac{\pi i n (k-q/2)}{q/2}}=(-1)^n \sum_{\substack{k=1 \\  (k,q)=1}}^{q}e^{2\pi i n (2k-q/2)/q}.
					\end{aligned}
				\end{equation}
				We now show that the set $B':=\{2k-q/2: (k,q)=1\}$ contains all the numbers coprime to $q.$	For this purpose,  assume that for any given $k$ with $(k,q)=1,$ we have $(2k-q/2,q)=\ell>1.$ Then, $\ell \mid (2k-q/2)$ and $\ell \mid q.$ We can easily check that $2k-q/2$ is odd which eventually implies  $\ell$ is odd.	Therefore, $\ell \mid q$ implies $\ell \mid q/2.$ As a consequence, we get $\ell \mid 2k.$ Again as $\ell$ is odd, we get $\ell \mid k.$ Thus, we arrive at a contradiction, since $(k,q)=1.$ Therefore, $B'$ contains all the numbers coprime to $q$ and then by using this in \eqref{S4: cq2nton}, we finally get $c_q(2n)=(-1)^n c_q(n).$ Hence the claim follows.
			\end{proof}\begin{prop}\label{prop:zak}
				Let \( N = 2d \) be a positive even integer, where \( d \) is odd. Then, for any \( i, j \in \mathcal{I}_K \), \( n \in \{0,1\} \), and any positive integer \( k \) such that \( (k, q_i) = 1 \), we have
				\[
				\mathcal{Z}c_{q_j}\left( \frac{kN}{q_i},\, n \right) = e^{2\pi i kn/q_i} 
				\begin{cases}
					\sqrt{N/2}, & \text{if } q_j = q_i, \\[4pt]
					(-1)^n \sqrt{N/2}, & \text{if } q_j = q_i/2 \text{ or } q_j = 2q_i, \\[4pt]
					0, & \text{otherwise}.
				\end{cases}
				\]
				\end{prop}
			\begin{proof}
				We divide the proof into the  three cases: (i) $q_j=q_i,$  (ii)  $q_j=q_i/2 \text{ or } 2q_i,$ and (iii)  $q_j\neq q_i, q_i/2 ,2q_i.$
				
				\noindent	\textbf{Case 1: $q_j=q _i.$} 	By using Proposition \ref{prop: cqprop} (ii) and then Proposition \ref{prop: cqprop} (i), we get
				\begin{equation}\label{S4:eveneqodd}
					\sum_{\ell=0}^{d-1}c_{q_i}(2\ell)e^{-2\pi i 2\ell k/q_i}=\sum_{\ell=0}^{d-1}c_{q_i}(2\ell+1)e^{-2\pi i (2\ell +1)k/q_i}=N/2.
				\end{equation}
				Consequently, 	by using \eqref{eq:Zak}, \eqref{S4:eveneqodd}, and  $N=2d,$ we get the following equality for $n \in \{0,1\}:$
				\begin{equation*}
					\begin{aligned}
						\mc Zc_{q_i}(kN/q_i,n)&=\frac{1}{\sqrt{d}}\sum_{\ell=0}^{d-1}c_{q_i}(2\ell+n)e^{-2\pi i 2\ell k/q_i}\\
						&=\frac{e^{2\pi i kn/q_i}}{\sqrt{d}}\sum_{\ell=0}^{d-1}c_{q_i}(2\ell+n)e^{-2\pi i (2\ell+n) k/q_i}=e^{2\pi i kn/q_i}N/(2\sqrt{d})=e^{2\pi i kn/q_i}\sqrt{N/2}.
					\end{aligned}
				\end{equation*}

				%
				%
				\noindent	\textbf{Case 2:} $q_j=q_i/2 \text{ or } 2q_i.$ If $q_i$ is even, then  $q_i/2$ is  a divisor of $N.$  Using Proposition \ref{prop: cqprop} (i) and (iii), we get 
				\begin{equation}\label{eq:cqi2cqi}
					\sum_{n=0}^{N-1}(-1)^nc_{q_i/2}(n)e^{-2 \pi i kn/q_i} =	\sum_{n=0}^{N-1}c_{q_i}(n)e^{-2 \pi i kn/q_i} =N.
				\end{equation}	
				Again, using  Proposition \ref{prop: cqprop} (i), we have 
				$	\sum_{n=0}^{N-1}c_{q_i/2}(n)e^{-2 \pi i kn/q_i} =0.$
				This  equality and \eqref{eq:cqi2cqi} gives
				\begin{equation}\label{eq:cqi2N2}
					\sum_{\ell=0}^{d-1}c_{q_i/2}(2\ell)e^{-2\pi i 2\ell k/q_i}=N/2 \,\,\, \text{ and } \,\,\,	\sum_{\ell=0}^{d-1}c_{q_i/2}(2\ell+1)e^{-2\pi i (2\ell +1)k/q_i}=-N/2.
				\end{equation}  
				Consequently, by using \eqref{eq:Zak}, \eqref{eq:cqi2N2}, and $N=2d,$ we get the following equality for $n \in \{0,1\}:$
				\begin{equation*}\label{eq: zcqi2}
					\mc Zc_{q_i/2}\left(\frac{kN}{q_i},n\right)=\frac{e^{2\pi i kn/q_i}}{\sqrt{d}}\sum_{\ell=0}^{d-1}c_{q_i/2}(2\ell+n)e^{-2\pi i (2\ell+n) k/q_i}=(-1)^ne^{2\pi i kn /q_i} \sqrt{N/2}.
				\end{equation*}			
				Also, if $q_i$ is odd, then $2q_i$ is a divisor of $N$ and the equality $\mc Zc_{2q_i}(kN/q_i,n)=(-1)^ne^{2\pi ikn /q_i} \sqrt{N/2}$ holds similarly for $n \in \{0,1\}.$ 
				
				\noindent
				\textbf{Case 3: $q_j\neq q_i, q_i/2 ,2q_i.$} 	 Using \eqref{eq:Zak} and  $N=2d,$ we get 
				\begin{equation*}
					\begin{aligned}
						&\mc Zc_{q_j}\left(\frac{kN}{q_i},0\right)=\frac{1}{\sqrt{d}}\sum_{\ell=0}^{d-1}c_{q_j}(2\ell)e^{-2\pi i (2 \ell)k/q_i}=\frac{1}{\sqrt{d}}\sum_{\ell=0}^{d-1}\bigg(\sum_{\substack{r=1 \\ (r,q_i)=1}}^{q_j}e^{2\pi i (2\ell)r/q_j}\bigg)e^{-2\pi i (2 \ell)k/q_i}\\
						&=\frac{1}{\sqrt{d}}\sum_{\substack{r=1 \\ (r,q_j)=1}}^{q_j}\left(\sum_{\ell=0}^{d-1}e^{2\pi i\big( \frac{r}{q_j} - \frac{k}{q_i}\big)(2\ell)}\right)= \frac{1}{\sqrt{d}}\sum_{\substack{r=1 \\ (r,q_j)=1}}^{q_j}\left(\sum_{\ell=0}^{d-1}e^{2\pi i (rt_1-kt_2)\ell/d}\right),
					\end{aligned}
				\end{equation*}
				where $N=t_1q_j=t_2q_i$ for some positive integers $t_1$ and $t_2.$		
				Note that the inner sum $\sum_{\ell=0}^{d-1}e^{2\pi i (rt_1-kt_2)\ell/d},$ is zero unless $rt_1-kt_2$ is a multiple of $d$. However, 
				$rt_1-kt_2=N\bigg(\frac{r}{q_j}-\frac{k}{q_i}\bigg)$ can never be an integer since 
				$(r,q_j)=(k,q_i)=1.$
				%
				%
				%
				This gives,
				$\mc Zc_{q_j}(kN/q_i,0)=0.$
				Similarly, we can show $\mc Zc_{q_j}(kN/q_i,1)=0.$ This completes the proof.
			\end{proof}

			We are now ready to prove Theorem \ref{thm:tight}. 
			
			\begin{proof}[Proof of Theorem \ref{thm:tight}]
				First, we prove (ii).  Note that for $m \in \mc J_d,$ $\mc U(m)$ (defined  in \eqref{eq:Uk}) is a $K \times 2$ matrix as $p=2$ in this case. 
				Then for  $m \in \mc J_d,$ we have
				\begin{equation}\label{eq:ustaru}
					\mc U^\asterisk(m)		\mc U(m)={d} \begin{pmatrix}
						\displaystyle\sum_{i=1}^K |\mc Zc_{q_i}(m,0)|^2 & \displaystyle\sum_{i=1}^K \mc Zc_{q_i}(m,0) \, \overline{\mc Zc_{q_i}(m,1)} \\
						\displaystyle\sum_{i=1}^K \overline{\mc Zc_{q_i}(m,0)} \, \mc Zc_{q_i}(m,1) & \displaystyle\sum_{i=1}^K |\mc Zc_{q_i}(m,1)|^2
					\end{pmatrix}.
				\end{equation}
				Let $m=0.$ We show that $\mc Zc_{q_i}(0,n)=0$ for $q_i \neq 1,2$ and $n \in \{0,1\}.$ Indeed, we have 
				\begin{equation*}
					\begin{aligned}
						\mc Zc_{q_i}(0,n)&=\frac{1}{\sqrt{d}}\sum_{\ell=0}^{d-1}c_{q_i}(2\ell+n)=\frac{1}{\sqrt{d}}\sum_{\ell=0}^{d-1}\bigg(\sum_{\substack{k=1 \\ (k,q_i)=1}}^{q_i}e^{2\pi i (2\ell+n)k/qi}\bigg)=\frac{1}{\sqrt{d}}\sum_{\substack{k=1 \\ (k,q_i)=1}}^{q_i}\left(\sum_{\ell=0}^{d-1}e^{2\pi ik t\ell/d}\right)e^{2\pi ikn/q_i},
					\end{aligned}
				\end{equation*}		
				where $N=tq_i$ for some positive integer $t.$
				Note that for $q_i\neq 1,2,$ the inner sum is the partial sum of a geometric series with common ratio $e^{2\pi ikt\ell/d} \neq 1$  since $kt\ell/d=kN\ell/(q_id)=2k\ell/q_i \notin \mathbb Z$  for any $k$ with $(k,q_i)=1.$ Therefore,  $$\sum_{\ell=0}^{d-1}e^{2\pi ikt\ell/d}=\frac{1-e^{2\pi ikt\ell}}{1-e^{2\pi ikt\ell/d}}=0.$$
				This implies, $\mc Zc_{q_i}(0,n)=0$ for $q_i \neq 1,2$ and $n \in \{0,1\}.$ 
				Also, it is easy to check that $\mc Zc_{1}(0,n)=\sqrt{d}$ and $\mc Zc_2(0,n)=(-1)^n\sqrt{d}$ for $n \in \{0,1\}.$ Combining everything and using \eqref{eq:ustaru}, we get
				\begin{equation}\label{eq:ustartu0}
					\mc U^\asterisk(0)		\mc U(0)={d} \begin{pmatrix}
						2d&& 0  \\
						0 && 	2d
					\end{pmatrix}
					=2d^2 I_2. \end{equation}
				Now	let $m \in \{1,\ldots,d-1\}$ be fixed. Then by Lemma \ref{lem:mknq}, there exists a divisor $q_i$ of $N$ and a number $k$ coprime to $q_i$ such that $m=kN/q_i.$ First assume that $q_i$ is even. Then by Proposition \ref{prop:zak}, we get \begin{equation}\label{eq:ustartu1122}
					\sum_{j=1}^K	|\mc Zc_{q_j}(kN/q_i,n)|^2=|\mc Zc_{q_i}(kN/q_i,n)|^2 + |\mc Zc_{q_i/2}(kN/q_i,n)|^2 =N,
				\end{equation}
				for any $n \in \{0,1\}$ and 
				\begin{equation}\label{eq:ustartu1221}
					\begin{aligned}
						\sum_{j=1}^K\mc Zc_{q_j}&(kN/q_i,0)\ol{\mc Zc_{q_j}(kN/q_i,1)} =\mc Zc_{q_i}(kN/q_i,0)\ol{\mc Zc_{q_i}(kN/q_i,1)} \\
						&+ \mc Zc_{q_i/2}(kN/q_i,0)\ol{\mc Zc_{q_i/2}(kN/q_i,1)} =N/2 \,e^{-2\pi ik /q_i} + N/2 (-1) e^{-2\pi ik/q_i} =0.
					\end{aligned}
				\end{equation}
				Then by using \eqref{eq:ustaru}, \eqref{eq:ustartu1122} and \eqref{eq:ustartu1221}, we get
				\begin{equation*}
					\mc U^\asterisk(m)		\mc U(m)=d\begin{pmatrix}
						N && 0   \\
						0 && 	N
					\end{pmatrix}=2d^2 I_2,\,\, m\in \mc J_d.		\end{equation*}
				This, along with \eqref{eq:ustartu0} shows that $	\mc U^\asterisk(m)		\mc U(m)$ has the  same eigenvalues for $m \in \{0,1,\ldots,d-1\}.$ This implies $\operatorname{rank}\, \mc U(m)=2$ for $m \in \mc J_d$ and moreover $A=B=2d^2$ where $A$ and $B$ are defined in \eqref{eq: AB}. 
				Consequently, $\mc R_{2,N}$ is a tight frame with   frame bound  $2d^2.$ 
				
				In order to
				prove the latter claim in (ii),
				we first show that $\mc Zc_{q_i}(1,1)=0$ for $i \in \mc I_K.$ To this end,  consider 
				\begin{equation}\label{eq:Zcqim1}
					\begin{aligned}
						&\mc Zc_{q_i}(1,1)=\frac{1}{\sqrt{d}}\sum_{\ell=0}^{d-1}c_{q_i}(2\ell+1)e^{-2\pi i  \ell/d}=\frac{1}{\sqrt{d}}\sum_{\ell=0}^{d-1}\bigg(\sum_{\substack{k=1 \\ (k,q_i)=1}}^{q_i}e^{2\pi i (2\ell+1)k/qi}\bigg)e^{-2\pi i  \ell/d}\\
						&=\frac{1}{\sqrt{d}}\sum_{\substack{k=1 \\ (k,q_i)=1}}^{q_i}\left(\sum_{\ell=0}^{d-1}e^{2\pi i\left( \frac{2k}{q_i} - \frac{1}{d}\right)\ell}\right)e^{2\pi ik/q_i}=\frac{1}{\sqrt{d}}\sum_{\substack{k=1 \\ (k,q_i)=1}}^{q_i}\left(\sum_{\ell=0}^{d-1}e^{2\pi i\left( rk-1\right)\ell/d}\right)e^{2\pi ik/q_i},
					\end{aligned}
				\end{equation}
				where $N=rq_i$	for some positive integer $r.$ Note that the inner sum in the last term of the above equality is zero unless for some $k$ such that $(k,q_i)=1,$ $rk-1$ is a multiple of $d.$ Thus, if $rk-1=td$ 
				for some $t \in \mathbb Z,$ then we get 
				\begin{equation}\label{eq:td+1}
					\frac{k}{q_i}=\frac{td+1}{N}.
				\end{equation}
				Since $d$ is even, thus $td+1$ is odd and hence $td+1$ and $N$ do not share any factor of $2.$ Also, if $s \neq 2$ is a  prime such that $s \mid N$ and $s \mid td+1$, then $s\mid N=2d$ implies $s\mid d.$ This gives $s \mid td.$ Consequently, we get $s\mid 1,$ which is not possible. Therefore, $td+1$ and $N$ do not share any prime factor as well. This implies $(td+1,N)=1.$ Then, the equality \eqref{eq:td+1} is possible if and only if $q_i=N$ and $k=1,d+1.$ In these cases, the common ratio $e^{2\pi i( rk-1)/d}=1$  and hence \eqref{eq:Zcqim1} becomes:
				$$\mc Zc_{N}(1,1)=\sqrt{d}\sum_{k=1,d+1}e^{2\pi ik/N}=\sqrt{d}\left(e^{2\pi i/N} + e^{2\pi i(d+1)/N} \right)=\sqrt{d}\left(e^{2\pi i/N} -e^{2\pi i/N} \right)=0.$$
				On combining everything, we get $\mc Zc_{q_i}(1,1)=0$ for $i \in \mc I_K.$ This implies that the second column of $\mc U(1)$ is zero, and hence  rank $ \mc U(1)=1 \neq 2,$ violating Lemma  \ref{lem:mainlem}.

				Now we prove (i). In this case, $p=1,$ so $\mc U(m)$
				is a $K \times 1$ matrix and thus similar to the proof of (ii), we show that  				$\mc U^\asterisk(m)\mc U(m)=AI_1 $ for $m\in \mc J_N.$ where $A=\min_{k\in \mc J_N}\lambda_{\min}[\mc U^\asterisk(m)\mc U(m)].$ Let $m \in\{1,2,\ldots,N-1\}$ be fixed. Then by Lemma \ref{lem:mknq}, we have $m=kN/q_i$ for some  $q_i|N$ and $(k,q_i)=1.$
				Consequently,  by Proposition \ref{prop: cqprop} (i), we get
				\begin{equation*}
					\begin{aligned}
						\mc U^\asterisk(m)		\mc U(m)={N}\sum_{j=1}^K	|\mc Zc_{q_i}(m,0)|^2=N\sum_{j=1}^K\left|\frac{1}{\sqrt{N}}\sum_{n=0}^{N-1}c_{q_j}(n)e^{-2\pi i kn/q_i}\right|^2=N^2.
					\end{aligned}
				\end{equation*}
				Also, for $m=0,$ we have
				\begin{equation*}
					\begin{aligned}
						\mc U^\asterisk(0)		\mc U(0)={N}\sum_{j=1}^K	|\mc Zc_{q_j}(0,0)|^2=N\sum_{j=1}^K\left|\frac{1}{\sqrt{N}}\sum_{n=0}^{N-1}c_{q_j}(n)\right|^2=N\left|\frac{1}{\sqrt{N}}\sum_{n=0}^{N-1}c_{1}(n)\right|^2=N^2,
					\end{aligned}
				\end{equation*}
		where we have used the sum property  of the Ramanujan sums.
				Combining everything, we get $A=N^2$ and therefore $\mc R_{1,N}$ is a tight frame with frame bound $N^2.$  
				
				Finally,  we prove (iii). 	
				We first prove that $\mc Zc_{q_i}(0,j) =\mc Zc_{q_i}(0,p-j)$ for $j \in \mc I_{p-1}$ and $i \in \mc I_K.$ To this end, we solve
				\begin{equation*}
					\begin{aligned}
						\mc Zc_{q_i}(0,p-j)&=\frac{1}{\sqrt{d}}\sum_{\ell=0}^{d-1}c_{q_i}(p\ell+p-j)=\frac{1}{\sqrt{d}}\sum_{\ell=0}^{d-1}c_{q_i}(p(\ell+1)-j).
					\end{aligned}
				\end{equation*}
				Substituting $t=\ell+1$ and using the fact that $c_{q_i}$ is $N$-periodic, we obtain
				\begin{equation}\label{eq:Zcqi0p-j}
					\begin{aligned}
						\mc Zc_{q_i}(0,p-j)&=\frac{1}{\sqrt{d}}\sum_{t=1}^{d}c_{q_i}(pt-j)=\frac{1}{\sqrt{d}}\sum_{t=1}^{d-1}c_{q_i}(pt-j) + \frac{1}{\sqrt{d}}c_{q_i}(N-j)\\
						&=\frac{1}{\sqrt{d}}\sum_{t=1}^{d-1}c_{q_i}(pt-j) + \frac{1}{\sqrt{d}}c_{q_i}(0-j)=\frac{1}{\sqrt{d}}\sum_{t=0}^{d-1}c_{q_i}(pt-j).
					\end{aligned}
				\end{equation}
				Expanding the quantity $\sum_{t=0}^{d-1}c_{q_i}(pt-j)$ and using $(q_i-k,q_i)=1$ for any $k$ with $(k,q_i)=1,$ we get
				\begin{equation}\label{eq:cqiptj}
					\begin{aligned}
						&\sum_{t=0}^{d-1}c_{q_i}(pt-j)=\sum_{t=0}^{d-1}\sum_{\substack{k=1 \\ (k,q_i)=1}}^{q_i}e^{2\pi i(pt-j)k/q_i}=\sum_{\substack{k=1 \\ (k,q_i)=1}}^{q_i}\bigg(\sum_{t=0}^{d-1}e^{2\pi i ptk/q_i}\bigg)	e^{-2\pi i kj/q_i}	\\
						&=	\sum_{\substack{k=1 \\ (k,q_i)=1}}^{q_i}\bigg(\sum_{t=0}^{d-1}e^{-2\pi i pt(q_i-k)/q_i}\bigg)	e^{2\pi i (q_i-k)j/q_i}	=\sum_{\substack{k=1 \\ (q_i-k,q_i)=1}}^{q_i}\left(\sum_{t=0}^{d-1}e^{-2\pi i pt(q_i-k)/q_i}\right)	e^{2\pi i (q_i-k)j/q_i}	\\
						&	=\sum_{\substack{h=1 \\ (h,q_i)=1}}^{q_i}\left(\sum_{t=0}^{d-1}e^{-2\pi i pth/q_i}\right)	e^{2\pi i hj/q_i}		=\sum_{\substack{h=1 \\ (h,q_i)=1}}^{q_i}\left(\sum_{t=0}^{d-1}e^{2\pi i pth/q_i}\right)	e^{2\pi i hj/q_i}		\\
						&=	\sum_{\substack{h=1 \\ (h,q_i)=1}}^{q_i}\left(\sum_{t=0}^{d-1}e^{2\pi i (pt+j)h/q_i}\right)	=\sum_{t=0}^{d-1}c_{q_i}(pt+j),
					\end{aligned}
				\end{equation}
				where we have used the fact that the sum $\sum_{t=0}^{d-1}e^{-2\pi i pth/q_i}=\sum_{t=0}^{d-1}e^{-2\pi i rth/d}$	is real. By using \eqref{eq:Zcqi0p-j} and \eqref{eq:cqiptj}, we get
		\begin{equation*}
			\mc Z c_{q_i}(0, p - j) 
			= \frac{1}{\sqrt{d}} \sum_{t=0}^{d-1} c_{q_i}(p t + j) 
			= \mc Z c_{q_i}(0, j), 
			\, i \in \mc I_K, \, j \in \mc I_{p-1}.
		\end{equation*}
				This shows that the $j$-th and $(p-j)$-th column of $\mc U(0)$ are equal for $j \in \mc I_{p-1}.$ Therefore, we get 
				$$ \text{rank}\,\, \mc U(0)\leq	
				\begin{cases} 
					\frac{p-1}{2}+1, & \text{ $p$ is odd,} \\
					(p/2)+1,	 & \text{$p$ is even.} 
				\end{cases}$$
				In both  cases, rank $\mc U(0) <p$ since $p \neq 1,2,$	violating Lemma  \ref{lem:mainlem}. Hence the claim follows.				
			\end{proof}
			\begin{exmp}\label{S4: exampletight6}
				Let $N=6, p=2.$ Then, in view of Theorem \ref{thm:tight}, $d=3,$ $K=4$ and  the collection $$\mc R_{2,6}=\{L_{2k}c_1\}_{k=0}^2 \cup \{L_{2k}c_2\}_{k=0}^2 \cup \{L_{2k}c_3\}_{k=0}^2\cup \{L_{2k}c_6\}_{k=0}^2,$$ forms a tight frame for $\ell^2(\mathbb Z_{6}).$						The $4 \times 2$ matrix $\mc U(m)$ for $ m \in \{0,1,2\}$ is given by:
				\begin{equation*}
					\mc U(m)=\sqrt{3}\begin{pmatrix}
						\sum_{\ell=0}^{2}c_1(2\ell)e^{2\pi i m\ell/3}&& \sum_{\ell=0}^{2}c_1(2\ell+1)e^{2\pi i m\ell/3}\\
						\sum_{\ell=0}^{2}c_2(2\ell)e^{2\pi i m\ell/3}&& 	\sum_{\ell=0}^{2}c_2(2\ell+1)e^{2\pi i m\ell/3} \\
						\sum_{\ell=0}^{2}c_3(2\ell)e^{2\pi i m\ell/3}&& 	\sum_{\ell=0}^{2}c_3(2\ell+1)e^{2\pi i m\ell/3}  \\
						\sum_{\ell=0}^{2}c_6(2\ell)e^{2\pi i m\ell/3}&& 	\sum_{\ell=0}^{2}c_6(2\ell+1)e^{2\pi i m\ell/3} 
					\end{pmatrix}.
				\end{equation*}
				After calculating these matrices, we obtain
				\[
				\mc U(0)=\begin{pmatrix}
					3 & 3 \\
					3 & -3 \\
					0&0 \\
					0& 0
				\end{pmatrix}, \quad \mc U(1)=\begin{pmatrix}
					0 & 0 \\
					0 & 0 \\
					3& -1.5 + i\, 2.5981\\
					3& 1.5 - i\, 2.5981
				\end{pmatrix}, \quad \mc U(2)=\begin{pmatrix}
					0 & 0 \\
					0 & 0 \\
					3& -1.5 - i\, 2.5981\\
					3& 1.5 + i\, 2.5981
				\end{pmatrix}.
				\]
				It can be verified that $\mc U^\asterisk(m)\mc U(m)=18I_2$ for $ m=0,1,2.$ Thus, $\mc R_{2,6}$ is a tight frame with bound 18.
			\end{exmp}
			\begin{exmp}
				Let $N=8$, $p=1$. Then, in view of Theorem \ref{thm:tight}, $d=8,$ $K=4$ and the collection $$\mc R_{1,8} = \{L_kc_1\}_{k=0}^7 \cup \{L_kc_2\}_{k=0}^7 \cup \{L_kc_4\}_{k=0}^7 \cup \{L_kc_8\}_{k=0}^7,$$ forms a tight frame for $\ell^2(\mathbb Z_8)$.
				The $4 \times 1$ matrix $\mc U(m)$ for $m \in \{0,1,\hdots,7\}$ is given by:
				\begin{equation*}
					\mc U(m)=\sqrt{8}\begin{pmatrix}
						\sum_{\ell=0}^{7}c_1(\ell)e^{2\pi i m\ell/8}\\
						\sum_{\ell=0}^{2}c_2(\ell)e^{2\pi i m\ell/8}\\
						\sum_{\ell=0}^{7}c_4(\ell)e^{2\pi i m\ell/8}  \\
						\sum_{\ell=0}^{7}c_8(\ell)e^{2\pi i m\ell/8} 
					\end{pmatrix}.
				\end{equation*}
				After calculating these matrices, we get
				\[
				\mc U(m)=\begin{pmatrix}
					8  \\
					0 \\
					0 \\
					0
				\end{pmatrix},\, m=0, 
				\quad
				\mc U(m)=\begin{pmatrix}
					0  \\
					8 \\
					0 \\
					0
				\end{pmatrix},\, m=4, 
				\quad
				\mc U(k)=\begin{pmatrix}
					0  \\
					0 \\
					8 \\
					0
				\end{pmatrix},\, m=\{2,6\},
				\quad
				\mc U(m)=\begin{pmatrix}
					0  \\
					0 \\
					0 \\
					8
				\end{pmatrix},\, m=\{1,3,5,7\}. 
				\]
				Therefore, $\mc U^\ast(m)\mc U(m) = 64I_1$ for $m \in \{0, 1, \dots, 7\}$. Thus, $\mc R_{1,8}$ is a tight frame with bound 64.
			\end{exmp}
			\begin{exmp}
				Let $N=12$, $p=2$. Then, in view of Theorem \ref{thm:tight}, $d=6$ and $K=6$ and the collection $$\mc R_{2,12} = \{L_{2k}c_1\}_{k=0}^5 \cup \{L_{2k}c_2\}_{k=0}^5 \cup \{L_{2k}c_3\}_{k=0}^5 \cup \{L_{2k}c_4\}_{k=0}^5 \cup \{L_{2k}c_6\}_{k=0}^5 \cup \{L_{2k}c_{12}\}_{k=0}^5,$$ does not form a  frame for $\ell^2(\mathbb Z_{12})$.
				The $6 \times 2$ matrix $\mc U(m),\, m\in \{0,1, \hdots, 5\}$ is given by:
				\begin{equation*}
					\mc U(m)=\sqrt{5}\begin{pmatrix}
						\sum_{\ell=0}^{5}c_1(2\ell)e^{2\pi i m\ell/12} && \sum_{\ell=0}^{5}c_1(2\ell+1)e^{2\pi i m\ell/12}\\
						\sum_{\ell=0}^{5}c_2(2\ell)e^{2\pi i m\ell/12} && \sum_{\ell=0}^{5}c_2(2\ell+1)e^{2\pi i m\ell/12}\\
						\sum_{\ell=0}^{5}c_3(2\ell)e^{2\pi i m\ell/12}  && \sum_{\ell=0}^{5}c_3(2\ell+1)e^{2\pi i m\ell/12}\\
						\sum_{\ell=0}^{5}c_4(2\ell)e^{2\pi i m\ell/12} && 	\sum_{\ell=0}^{5}c_4(2\ell+1)e^{2\pi i m\ell/12} \\
						\sum_{\ell=0}^{5}c_6(2\ell)e^{2\pi i m\ell/12} && 	\sum_{\ell=0}^{5}c_6(2\ell+1)e^{2\pi i m\ell/12} \\
						\sum_{\ell=0}^{5}c_{12}(2\ell)e^{2\pi i m\ell/12} && 	\sum_{\ell=0}^{5}c_{12}(2\ell+1)e^{2\pi i m\ell/12} 
					\end{pmatrix}.
				\end{equation*}
		It can be verified that rank $\mc U(m)=1 \neq 2$ for $m \in \{1,3,5\}.$ Then, by Lemma \ref{lem:mainlem}, the collection $\mc R_{2,12}$ is not a  frame for $\ell^2(\mathbb Z_{12}).$ 
			\end{exmp}

			In our analysis, we demonstrated that \( \mathcal{R}_{p,N} \) for \( p > 2 \) does not span the full space \( \ell^2(\mathbb{Z}_N) \), as it does not form a frame for \( \ell^2(\mathbb{Z}_N) \). This is because, for \( p > 2 \), certain individual collections \( \{L_{pk}c_{q_i}\}_{k \in \mc J_d} \) within \( \mathcal{R}_{p,N} \) fail to span the specific subspaces of \( \ell^2(\mathbb{Z}_N) \) to which they are associated. In the following section, we show that by implementing a non-uniform filter bank structure, where decimation ratios in each channel are chosen appropriately according to the properties of the Ramanujan sums and Ramanujan subspaces, the collection \( \mathcal{R}_{p,N} \) can indeed be modified to form a frame for \( \ell^2(\mathbb{Z}_N) \).
			\subsection{Non-uniform filter banks with Ramanujan sums}\label{sub:3nonuniform}

			In this subsection, we prove that by updating the decimation ratios for appropriately selected channels in the uniform filter bank based on Ramanujan sums \(c_{q_1}, c_{q_2}, \dots, c_{q_K}\), originally designed with a fixed decimation ratio \(p > 2\), the resulting non-uniform filter bank configuration exhibits frame properties.  Non-uniform filter banks are signal processing frameworks that allow variable decimation rates across channels, unlike traditional filter banks which use a uniform rate. See, \cite{hoang1989non, li1997simple} for more details.

		  To this end, we first observe that any two individual collections, \( \{L_{pk}c_{q_i}\}_{k=0}^{d-1} \) and \( \{L_{pk}c_{q_j}\}_{k=0}^{d-1} \) for \( j \neq i \), within \( \mathcal{R}_{p,N} \)  lie in distinct, orthogonal subspaces of \( \ell^2(\mathbb{Z}_N) \), as shown in Theorem \ref{thm:basis}.
			The introduction of a non-uniform filter bank is necessitated by the fact that some individual collections within \( \mathcal{R}_{p,N},  \) for $p>2,$ fail to span their respective subspaces, thereby preventing \( \mathcal{R}_{p,N} \) for $p>2$ from forming  frame for \( \ell^2(\mathbb{Z}_N) \).
			By carefully adjusting the decimation ratios for channels corresponding to these non-spanning collections and utilizing the orthogonal decomposition of \( \ell^2(\mathbb{Z}_N) \), we demonstrate that this modified form of \( \mathcal{R}_{p,N} \), arising from the resulting non-uniform filter bank with $p>2$, forms a frame for \( \ell^2(\mathbb{Z}_N) \) (Theorem \ref{thm: rpkframe}).
			
For our further analysis, we first discuss the spanning properties of the Ramanujan subspace $S_q$.	The authors in \cite[Theorem 4]{vaidyanathan2014ramanujan1}  proved that any $\phi(q)$ consecutive shifts of the Ramanujan sum $c_q$ are linearly independent. In particular, 
 the collection $
\{c_q, L_1c_q, \ldots, L_{\phi(q)-1}c_q\}$ is a basis for 
 the Ramanujan subspace $S_q=\operatorname{span} \left\{ L_{k} c_{q} : 0 \leq k \leq \phi(q) - 1 \right\}$.

Since $c_{q_i} \in \ell^2(\mathbb Z_N)$ for $i \in \mc I_K$, the space $S_{q_i}$ is a $\phi(q_i)$-dimensional subspace of $\ell^2(\mathbb Z_N)$.	Also observe that the subspaces $S_{q_i}$ and $S_{q_j}$ are orthogonal for $q_i, q_j |N$ and   $q_i \neq q_j$ due to the orthogonality property of the Ramanujan sums (see, Proposition \ref{prop:properties}(iv)).
Now by using Gauss's theorem on sums of  Euler's totient functions, we have dim $\ell^2(\mathbb Z_N)=N=\sum_{i=1}^{K}\phi(q_i)=\text{dim } S_{q_1} + \text{dim } S_{q_2}+ \cdots + \text{dim } S_{q_K}$. This gives the orthogonal decomposition   $$\ell^2(\mathbb Z_N)=S_{q_1} \oplus S_{q_2} \oplus \cdots \oplus S_{q_K}.$$
As a consequence, any \( x \in \ell^2(\mathbb{Z}_N) \) admits the representation
\begin{equation}\label{eq:rptrep}
	x = \sum_{i=1}^K \sum_{\ell=0}^{\phi(q_i)-1} \alpha_{i,\ell} \, c_{q_i}(\cdot - \ell),
\end{equation}
for suitable coefficients $\alpha = \{ \alpha_{i,\ell} :  0 \leq \ell < \phi(q_i)-1, i \in \mc I_K\}$. This expansion serves as the foundation of the Ramanujan Periodic Transform (RPT) proposed in \cite{vaidyanathan2014ramanujan2}. The RPT is well-suited for extracting the periodic structure of finite-length signals and identifying hidden periods. It is also useful for denoising periodic signals by utilizing their underlying periodic information captured in \eqref{eq:rptrep}.

The next result establishes that the space $\ell^2(\mathbb{Z}_N)$ admits an orthogonal decomposition into Ramanujan subspaces generated by even consecutive shifts of Ramanujan sums corresponding to the divisors of \(N\). To this end, for each pair of divisors \(p\) and \(q_i\) of \(N\), where \(i \in \mathcal{I}_K\), we define the Ramanujan subspace
\begin{equation}\label{eq:ramanujansubspaces}
	S_{p,q_i} := \operatorname{span} \left\{ L_{pk} c_{q_i} : 0 \leq k \leq \phi(q_i) - 1 \right\}.
\end{equation}
Note that for $p=1$, $S_{1,q_i}=S_{q_i}$ for $i \in \mc I_K$. For \(p = 2\), the following result holds.

	\begin{thm}\label{thm:basis} 
	Let 	$p_i,\, i \in \mc I_n$, be primes and $\alpha_i,\, i \in \mc I_n$, be positive integers such that $N=2p_1^{\alpha_1}p_2^{\alpha_2}\hdots p_n^{\alpha_n}, p_i \neq 2$.  Let \(S_{2,q_i}\) denote the Ramanujan subspace corresponding to the divisor $q_i$ of \(N\), for \(i \in \mathcal{I}_K\), as defined in \eqref{eq:ramanujansubspaces}. Then, for any $\ell \in \mathbb Z$ and $i \in \mc I_K$, the collection  $\{L_{2k}c_{q_i}\}_{k=\ell}^{\ell+\phi(q_i)-1}$ is a basis of $S_{2,q_i}$ and thus  dim $S_{2,q_i}=\phi(q_i)$. Moreover, the system 
	$$\mc B_{N}:=\{L_{2k}c_{q_1}\}_{k=0}^{\phi(q_1)-1} \cup \{L_{2k}c_{q_2}\}_{k=0}^{\phi(q_2)-1} \cup \{L_{2k}c_{q_3}\}_{k=0}^{\phi({q_3})-1} \cup \hdots \cup \{L_{2k}c_{{q_K}}\}_{k=0}^{\phi({q_K})-1}$$
	forms  a  basis for $\ell^2(\mathbb Z_N)$   and the orthogonal decomposition of  $\ell^2(\mathbb Z_N)$ is given by:	$$\ell^2(\mathbb Z_N)=S_{2,q_1} \oplus S_{2,q_2}  \oplus \cdots \oplus S_{2,q_K}.$$
		\end{thm}
	\begin{remar}
		The collection \( \mc B_N \)  is written as a union of shifted Ramanujan sums for notational convenience. However, we implicitly view each vector \( L_{2k}c_{q_i} \) as a block vector of the form \( (0, \ldots, L_{2k}c_{q_i}, \ldots, 0) \), where the nonzero component appears in the \( i \)th summand of the orthogonal decomposition
		\[
		\ell^2(\mathbb{Z}_N) = S_{2,q_1} \oplus S_{2,q_2} \oplus \cdots \oplus S_{2,q_K}.
		\]
	From now on, all such unions in an orthogonal decomposition of a Hilbert space will be implicitly interpreted as block-wise unions within the corresponding direct sum decomposition.
	\end{remar}
	The following lemma, concerning frames and bases in orthogonal subspaces, will be used in the sequel. While the idea that the union of frames (or bases) in mutually orthogonal subspaces yields a frame (or basis) for the direct sum is discussed in \cite[Section~5.1]{Waldrontight2018} and \cite[Chapter~1, page~12]{han2000frames}, we include a precise formulation here for completeness and omit the proof for brevity.
		\begin{lem}\label{lem:frameunion}
			Let \(\{\mathcal{H}_i\}_{i \in \mathcal{I}_J}\) be mutually orthogonal subspaces of a finite-dimensional Hilbert space \(\mathcal{K}\), i.e., \(\langle x, y \rangle_{\mathcal{K}} = 0\) for all \(x \in \mathcal{H}_i\), \(y \in \mathcal{H}_j\) with \(i \neq j\). Assume that \(\mathcal{K} = \bigoplus_{i \in \mathcal{I}_J} \mathcal{H}_i\), and for each \(i \in \mathcal{I}_J\), let \(\mathcal{D}_i\) be a finite collection of vectors in \(\mathcal{H}_i\). 
			
			Then, the union \(\mathcal{D} := \bigcup_{i \in \mathcal{I}_J} \mathcal{D}_i\) forms a frame (respectively, a basis) for \(\mathcal{K}\) if and only if each \(\mathcal{D}_i\) is a frame (respectively, a basis) for \(\mathcal{H}_i\). Moreover, \(\mathcal{D}\) is a tight frame for \(\mathcal{K}\) if and only if each \(\mathcal{D}_i\) is a tight frame for \(\mathcal{H}_i\) with the same frame bound.
		\end{lem}
		We now prove Theorem \ref{thm:basis}. The arguments of its proof are adopted from \cite[Theorem 4]{vaidyanathan2014ramanujan1}.
		\begin{proof}[Proof of Theorem \ref{thm:basis}.]
Let \( q_i \), for \( i \in \mathcal{I}_K \), be fixed. We begin by proving the basis part of  \( S_{2,q_i} \). Let \( M_{q_i}^{\ell} \) denote the matrix formed by the elements of the collection \( \{L_{2k} c_{q_i}\}_{k=\ell}^{\ell + \phi(q_i) - 1} \), that is, $M_{q_i}^{\ell}=\begin{bmatrix}
			L_{2\ell} c_{q_i} & L_{2(\ell+1)}c_{q_i} & \hdots & L_{2(\ell+\phi(q_i)-1)}c_{q_i}
		\end{bmatrix}.$ Then $M_{q_i}^{\ell}$ has the form:
		\[
			M_{q_i}^{\ell} = \left[
			\begin{array}{c}
				U_{q_i}  V_{q_i}^{\ell} \\[2pt]
				\hline 
				U_{q_i}  V_{q_i}^{\ell}\\[2pt]					\hline 
				\vdots\\[2pt]
				\hline
				U_{q_i}  V_{q_i}^{\ell}
			\end{array}
			\right]_{N \times \phi(q_i)}\quad (N/q_i\, \text{times}),
			\] 
		where 
		\begin{equation}\label{eq:Uqi}
			U_{q_i}=\begin{bmatrix}
				1 && 1 && \hdots && 1\\
				e^{2\pi i /q_i}	&& e^{2\pi i k_2/q_i}	&& \hdots && e^{2\pi i k_{\phi(q_i)}/q_i}	\\
				e^{2\pi i 2k_1/q_i}	&& e^{2\pi i 2k_2/q_i}	&& \hdots && e^{2\pi i 2k_{\phi(q_i)}/q_i}	\\
				\vdots && \vdots && \ddots&& \vdots\\
				e^{2\pi i (q_i-1)k_1/q_i}	&& e^{2\pi i (q_i-1)k_2/q_i}	&& \hdots && e^{2\pi i (q_i-1)k_{\phi(q_i)}/q_i}	
			\end{bmatrix}_{q_i \times \phi(q_i)}
		\end{equation}
		and 
		\begin{equation*}\label{eq:Vqi}
			V_{q_i}^{\ell}=\begin{bmatrix}
				e^{-2\pi i 2k_1\ell/q_i} && e^{-2\pi i 2k_1(\ell+1)/q_i}&&\cdots && e^{-2\pi i 2k_1(\ell+\phi(q_i)-1)/q_i}\\
				e^{-2\pi i 2 k_2\ell/q_i} && e^{-2\pi i 2k_2(\ell+1)/q_i}&&\cdots && e^{-2\pi i 2k_2(\ell+\phi(q_i)-1)/q_i}\\
				\vdots && \vdots && \ddots&& \vdots\\
				e^{-2\pi i 2 k_{\phi(q_i)}\ell/q_i} && e^{-2\pi i 2k_{\phi(q_i)}(\ell+1)/q_i}&&\cdots && e^{-2\pi i 2k_{\phi(q_i)}(\ell+\phi(q_i)-1)/q_i}\\
			\end{bmatrix}_{\phi(q_i)\times \phi(q_i)},
		\end{equation*}
	where \( k_j \in \mathbb{Z}_{q_i} \), with \( 1 \leq j \leq \phi(q_i) \), are such that \((k_j, q_i) = 1 \) for each \( i \in \mathcal{I}_K \).	It can be observed that $U_{q_i}$ is a row-Vandermonde matrix with distinct elements in its second row. Indeed, if $e^{-2\pi i \ell_1/q_i}=e^{-2\pi i \ell_2/q_i}$ for $1 \leq \ell_1<\ell_2 < \phi(q_i),$ then $e^{-2\pi i (\ell_1-\ell_2)/q_i}=1.$ This gives $\ell_1-\ell_2=mq_i$ for some integer $m.$ Since $|\ell_1-\ell_2| < \phi(q_i),$  the last equality is possible only if $m=0,$ which contradicts $\ell_1 \neq \ell_2.$ Therefore, each element in the second row of $U_{q_i}$ is distinct, and hence $U_{q_i}$ has rank $\phi(q_i).$	Also notice that $V_{q_i}^{\ell}$ can be written in the form: $V_{q_i}^{\ell}=DV_{q_i}^{0},$ where $D$ is a diagonal matrix with the  first column of $V_{q_i}^{\ell}$ as its diagonal entries.
		We now prove that the column-Vandermonde matrix $V_{q_i}^{0}$ has distinct elements in its second column.  Indeed, if $e^{-2\pi i 2k_j/q_i}=e^{-2\pi i 2k_m/q_i}$ for $1 \leq k_j<k_m< q_i,$ then $e^{-2\pi i 2(k_j-k_m)/q_i}=1.$ This gives $2(k_j-k_m)=nq_i$ for some $n \in \mathbb Z.$
		
	If $q_i=2,$ then the equality $2(k_j-k_m)=nq_i$ gives $(k_j-k_m)=n,$ which is possible only when $n=0$ since $|k_j-k_m| < q_i.$ If $q_i \neq 2,$ then since $N=2p_1^{\alpha_1}p_2^{\alpha_2}\cdots p_n^{\alpha_n}, p_i \neq 2$  implies that $q_i$ is odd. Consequently, $k_j-k_m=nq_i/2$ with $(q_i,2)=1.$ Again, this is possible only when $n=0.$ Thus, both the cases leads to the equality $k_j=k_m$, which is a contradiction.
		Therefore each element in the second column of $V_{q_i}^{0}$  is distinct and hence $V_{q_i}^{0}$ has rank $\phi(q_i).$  Since $D$ is a diagonal matrix with  non-zero elements on the diagonal,  $V_{q_i}^{\ell}$ has rank $\phi(q_i).$ Consequently, the product $U_{q_i}V_{q_i}^{\ell}$ has rank $\phi(q_i).$ Since each block in $M_{q_i}^{\ell}$  is a duplicate of the first block, increasing blocks does not increase the number of linearly independent rows in the matrix. Thus, the total number of linearly independent rows in  $M_{q_i}^{\ell}$ remains $\phi(q_i)$, and consequently, $M_{q_i}^{\ell}$ has rank $\phi(q_i).$  In particular, for $\ell=0$, the collection $\{L_{2k}c_{q_i}\}_{k=0}^{\phi(q_i)-1}$ is  linearly independent and hence  $\dim S_{2,q_i}=\phi(q_i).$ Since shifts are interpreted modulo $q_i,$  the space $S_{2,q_i}$ is invariant under shifts and hence $\{L_{2k}c_{q_i}\}_{k=\ell}^{\ell+\phi(q_i)-1} \subset S_{2,q_i}.$ Consequently, 
		$\{L_{2k}c_{q_i}\}_{k=\ell}^{\ell+\phi(q_i)-1}$ is a basis for $S_{2,q_i}$ for any $q_i|N$ and $\ell \in \mathbb Z.$

We now turn to the latter claim. The fact that $\mathcal{B}_N$ forms a basis follows directly from the previous discussion together with Lemma~\ref{lem:frameunion}. The orthogonal decomposition  follows by arguments analogous to those used for the case \(p = 1\). This completes the proof.
	\end{proof}

			\begin{remar}
				Whenever $N=2^m$ for some $m,$ it can be shown that the collection $\mc B_{1,N}$ forms  an orthogonal basis for $\ell^2(\mathbb Z_N)$, and the collection
				\[
				\left\{ \frac{1}{\sqrt{N\phi(1)}} \, L_k c_{2^0} \right\}_{k=0}^{\phi(1)-1}
				\cup 
				\left\{ \frac{1}{\sqrt{N\phi(2)}} \, L_k c_{2^1} \right\}_{k=0}^{\phi(2)-1}
				\cup 
				\cdots
				\cup 
				\left\{ \frac{1}{\sqrt{N\phi(2^m)}} \, L_k c_{2^m} \right\}_{k=0}^{\phi(2^m)-1}
				\]
				 forms  an orthonormal basis for $\ell^2(\mathbb Z_N).$ Given the orthogonal decomposition of $\ell^2(\mathbb Z_N),$ it is enough to show that the collection $\{L_{k}c_{2^j}\}_{k=0}^{\phi(2^j)-1}$ for $0 \leq \ell \leq m$  forms an orthogonal basis of $S_{1,2^j}.$ The proof of this fact can be found in \cite[Lemma 2]{vaidyanathan2014ramanujan1} and orthonormality follows from the norm equality: $	\|c_{2^\ell}\|^2_{\ell^2(\mathbb Z_N)}=\sum_{n=0}^{N-1}c_{2^\ell}^2(n)=\frac{N}{2^\ell}2^\ell\phi(2^\ell)=N \phi(2^\ell),\,\, 0 \leq \ell \leq m.$
			\end{remar}		
			%
			We now move our discussion to non-uniform filter banks. The following result demonstrates that for a  prime divisor \(p>2\) of $N$, a non-uniform filter bank constructed using Ramanujan sums \(c_{q_i}, i \in \mathcal{I}_K\), exhibits frame properties under specific conditions. Specifically,
			when the decimation ratio is set to \(p\) for \(c_{q_i}\) if \(p \nmid q_i\) or \(p > q_i\), and to \(1\) if \(p \mid q_i\) and \(p \leq q_i\), the collection $\mc R_{p,N}$ is modified to form a frame for $\ell^2(\mathbb Z_N).$
		\begin{thm}\label{thm: rpkframe}
				Let \( p \) be a prime divisor of \( N \), and let \( \{q_1, q_2, \ldots, q_K\} \) denote the set of all positive divisors of \( N \). For \( p > 2 \) and \( r \in \{1,2\} \), associate to each \( i \in \mathcal{I}_K \) the integer
				\[
				p_i := 
				\begin{cases}
					p, & \text{if } q_i \notin \mf D_p, \\
					r, & \text{if } q_i \in \mf D_p,
				\end{cases}
				\]
				where \( \mf D_p := \{q_i \mid N: p \mid q_i \text{ and } p \leq q_i\} \). Consider the system
			\begin{equation}\label{eq:R'pN_uniform}
				\mathcal{R}^r_{p,N} := \left\{ L_{p_1 k} c_{q_1} \right\}_{k \in \mathcal{J}_{N/p_1} } \cup \left\{ L_{p_2 k} c_{q_2} \right\}_{k \in \mathcal{J}_{N/p_2} }\cup \cdots \cup \left\{ L_{p_K k} c_{q_K} \right\}_{k \in \mathcal{J}_{N/p_K} }, 
			\end{equation}
				generated by a non-uniform Ramanujan filter bank. The following assertions hold:
				\begin{itemize}
					\item[$\mathrm{(i)}$] For \( r = 1 \), the system \( \mathcal{R}^1_{p,N} \) forms a frame for \( \ell^2(\mathbb{Z}_N) \) for all positive integers \( N \).
					\item[$\mathrm{(ii)}$] For \( r = 2 \), assume \( N = 2p_1^{\alpha_1} p_2^{\alpha_2} \cdots p_n^{\alpha_n} \) with \( p_i \neq 2 \) for all \( i \in \mathcal{I}_n \). Then, the system  \( \mathcal{R}^2_{p,N} \) is a frame for \( \ell^2(\mathbb{Z}_N) \).
				\end{itemize}
			\end{thm}
		
To prove Theorem~\ref{thm: rpkframe}, we require the following supporting result, which characterizes the set \( \mathfrak{D}_p \) as the collection of those divisors \( q_i \) of \( N \) for which the associated Ramanujan sum \( c_{q_i} \) requires an update in its decimation ratio.
			\begin{prop}\label{prop:rankmqi}
		Let \( p  \) be a prime divisor of $N$ and	let $\mc Q_{q_i}$ be  the matrix   formed by the columns $\{L_{pk}c_{{q_i}}\}_{k\in \mc J_d},$ i.e., $\mc Q_{q_i}=\begin{bmatrix}
					c_{q_i} & L_{p}c_{q_i} &L_{2p}c_{q_i}& \cdots & L_{p(d-1)}c_{q_i}
				\end{bmatrix}$.  Then,
				\[
				\text{rank}\,\mc Q_{q_i} =
				\begin{cases} 
					\phi(q_i), & \text{if } p \nmid q_i \text{ or } p > q_i, \\[6pt]
					{\phi(q_i/p)}, & \text{if } p \mid q_i \text{ and } p \leq q_i.
				\end{cases}
				\]
			\end{prop}
			\begin{proof}
				Let $q_i, i \in \mc I_K$, be fixed. Consider the column-Vandermonde matrix
				\begin{equation}\label{eq:Wqi}
					W_{q_i}=\begin{bmatrix}
						1 && e^{-2\pi i pk_1/q_i}&&\cdots && e^{-2\pi i (d-1)pk_1/q_i}\\
						1 && e^{-2\pi i pk_2/q_i}&&\cdots && e^{-2\pi i (d-1)pk_2/q_i}\\
						\vdots && \vdots && \ddots&& \vdots\\
						1 && e^{-2\pi i pk_{\phi(q_i)}/q_i}&&\cdots && e^{-2\pi i (d-1)pk_{\phi(q_i)}/q_i}\\
					\end{bmatrix}_{\phi(q_i)\times d},
				\end{equation}
			where \( k_j \in \mathbb{Z}_{q_i} \), with \( 1 \leq j \leq \phi(q_i) \), are such that \((k_j, q_i) = 1 \) for each \( j \).	Then, it can be observed that the matrix $\mc Q_{q_i}$ can be written as follows:
			\[
				\mc Q_{q_i} = \left[
				\begin{array}{c}
					U_{q_i}  W_{q_i} \\[2pt]
					\hline
					U_{q_i}  W_{q_i}\\[2pt]
					\hline 
					\vdots\\[2pt]
					\hline
					U_{q_i}  W_{q_i}
				\end{array}
				\right]_{N \times d}
				\quad (N/q_i\, \text{times}),
				\] 
				where $U_{q_i}$ is the matrix defined in \eqref{eq:Uqi}.
				
We divide the remaining proof into two cases depending on whether \( q_i \in \mathfrak{D}_p \) or \( q_i \notin \mathfrak{D}_p \).

\medskip

\noindent \textbf{Case 1:} \( q_i \notin \mathfrak{D}_p \). In this case, we show that all entries in the second column of \( W_{q_i} \) are distinct. Suppose, for contradiction, that
$
e^{-2\pi i p k_j / q_i} = e^{-2\pi i p k_m / q_i}
$
for some \( 1 \leq j < m \leq\phi(q_i) \). This implies
$
e^{-2\pi i p(k_j - k_m)/q_i} = 1,
$
and hence \( p(k_j - k_m) = nq_i \) for some integer \( n \). This gives
$
k_j - k_m = \frac{nq_i}{p}.
$
Since \( p \) is prime, and \( p \nmid q_i \) or \( p > q_i \), we have \( (p, q_i) = 1 \). Therefore, \( nq_i/p \in \mathbb{Z} \) only if \( p \mid n \), which forces \( n = 0 \) because \( |k_j - k_m| < q_i \). Thus, \( k_j = k_m \), a contradiction. Hence, all entries in the second column of \( W_{q_i} \) are distinct. Using the same argument as in the proof of Theorem~\ref{thm:basis}, it follows that \( \operatorname{rank} \mathcal{Q}_{q_i} = \phi(q_i) \) in this case.

\medskip

\noindent \textbf{Case 2:} \( q_i \in \mathfrak{D}_p \). Let \( q_i = p \cdot r \). Then the entries in the second column of the Vandermonde matrix \( W_{q_i} \) are of the form
$
e^{-2\pi i p k_j / q_i} = e^{-2\pi i k_j / r}, \quad \text{where } (k_j, q_i) = 1, \quad 1 \leq j \leq \phi(q_i).
$
Since \( (k_j, q_i) = 1 \) and \( q_i = p \cdot r \), it follows that \( (k_j, r) = 1 \) as well. Consequently, the values
$
e^{-2\pi i k_j / r}, \quad 1 \leq j \leq \phi(q_i),
$
are \(\phi(r)\) distinct \( r \)th roots of unity corresponding to residue classes coprime to \( r \). Therefore, the second column of \( W_{q_i} \) contains only \( \phi(r) \) distinct values, and so rank \( W_{q_i}= \phi(r) \). As a result, the matrix \( \mathcal{Q}_{q_i} \) also has rank \( \phi(r) = \phi(q_i/p) \), as claimed.
			\end{proof}
			
			Finally, we prove Theorem \ref{thm: rpkframe}.
			
				\begin{proof}[Proof of Theorem \ref{thm: rpkframe}]
We first prove (i).	Consider the case when 	$q_i \in \mathfrak{D}_p $ and	$r=1$. Then the corresponding individual  collection \( \{L_k c_{q_i}\}_{k=0}^{N-1} \) contains \( \phi(q_i) \)  consecutive elements since \( N \geq \phi(q_i) \) for any \( i \in \mathcal{I}_K \). Then by \cite[Theorem 4]{vaidyanathan2014ramanujan1},   the  collection \( \{L_k c_{q_i}\}_{k=0}^{N-1} \) contains a basis of $S_{1,q_i}$ and hence forms a frame for $S_{1,q_i}$.
	Next, consider the case when $q_i \notin \mf D_{p}$ and the decimation ratio is $p$.
	Since $S_{1,q_i}$ is invariant under shifts, thus each of the individual collection $\{L_{pk}c_{q_i}\}_{k=0}^{d-1}$ for any $p$ lies inside $S_{1,q_i}.$	
	Furthermore, $\{L_{pk}c_{q_i}\}_{k=0}^{d-1}$ contains a basis of $S_{1,q_i}$  since rank $\mc Q_{q_i} = \phi(q_i)= \dim S_{1,q_i}$ and hence forms a frame for $S_{1,q_i}.$
Thus for each $i\in \mc I_K$, the corresponding collection from the given non-uniform filter bank forms a frame for $S_{1,q_i}$. Then by using the orthogonal decomposition of $\ell^2(\mathbb Z_N)$ in terms of Ramanujan subspaces $S_{1,q_i}$ for $i \in \mc I_K$ and Lemma \ref{lem:frameunion}, it follows that \( \mathcal{R}^1_{p,K}\) forms a frame for \( \ell^2(\mathbb{Z}_N) \).

The proof of (ii) follows analogously by using Theorem~\ref{thm:basis} and Lemma~\ref{lem:frameunion}.  This completes the proof.
			\end{proof}
			\begin{exmp}\label{exmp:nonuniform}
				Let \( N = 12 \) and choose \( p = 3 \), so that \( d  = 4 \). In this case, the set \( \mathfrak{D}_3 \), as defined in Theorem~\ref{thm: rpkframe}, is given by
				\(
				\mathfrak{D}_3 = \{3, 6, 12\}.
				\)
				Then, by Theorem~\ref{thm: rpkframe}, the collection
				\[
				\mathcal{R}_{3,12}^{1} = \{L_{3k} c_1\}_{k=0}^{3} \cup \{L_{3k} c_2\}_{k=0}^{3} \cup \{L_k c_3\}_{k=0}^{11} \cup \{L_{3k} c_4\}_{k=0}^{3} \cup \{L_k c_6\}_{k=0}^{11} \cup \{L_k c_{12}\}_{k=0}^{11}
				\]
				generated from a non-uniform filter bank with Ramanujan sums \( c_1, c_2, c_3, c_4, c_6 \), and \( c_{12} \), and with respective decimation ratios \( 3, 3, 1, 3, 1, 1 \), forms a frame for \( \ell^2(\mathbb{Z}_{12}) \).
			\end{exmp}
			\begin{remar}
	It is worth noting that a non-uniform filter bank based on Ramanujan sums can also be constructed when \( d \) is even, such that the resulting collection \( \mathcal{R}_{2,N} \) can be modified to form a frame for \( \ell^2(\mathbb{Z}_N) \). In particular, consider \( p = 2 \) and a divisor \( q_i \mid N \) satisfying \( 2 \mid q_i \) but \( 2^2 \nmid q_i \). In this case, we observe that
	\(
	\varphi(q_i) = \varphi\left(q_i/2\right),
	\)
	which implies that the decimation ratio for \( c_{q_i} \) need not be modified, as the matrix \( \mathcal{Q}_{q_i} \) already has rank \( \varphi(q_i) \). Therefore, the set \( \mathfrak{D}_2 \) consists precisely of those divisors \( q_i \mid N \) for which \( 2^2 \mid q_i \). With this characterization of \( \mathfrak{D}_2 \), a non-uniform filter bank can be constructed in this setting as well.
			\end{remar}

			\subsection{Erasures in the context of Ramanujan sums}\label{sub:erasure}
	In this subsection, we investigate the conditions for the robustness of the filter bank and the associated frame systems, $\mc R_{p,N}$, under erasures in the context of Ramanujan sums. \textit{Filter bank erasures} refer to the removal of one or more channels from a filter bank when certain channels become noisy and fail to preserve the original signal's information. Many researchers have studied the effect of such erasures on filter banks and obtained conditions under which the underlying collection \(\{L_{pk}\tilde{h}_j\}_{k \in \mathcal{J}_{N/p}}^{j \in \mathcal{I}_s \backslash \{j_1,j_2,\hdots,j_n\}}\) (after removing the channels corresponding to \(h_{j_1},h_{j_2}, \hdots, h_{j_n}\)) remains a frame, where  $h_i, i \in \mc I_s$ are the signals representing the channels of the given filter bank.  When this condition holds, the filter bank is said to be \textit{robust to \(n\) filter bank erasures}. See, for example, \cite{casazza2003equal, kovacevic2002filter, holmes2004optimal}. 
			
	In practical scenarios, it may happen that a subset of frame elements $\{L_{pk} \tilde{h}_i\}_{(k,i) \in \mc{I}}$ is lost due to erasures, where $\mc{I}$ denotes the corresponding index set and $\#_{\mc I} = e$. For stable reconstruction in the presence of such $e$ erasures, it is essential that the remaining collection $\{L_{pk} \tilde{h}_i\}_{(k,i) \in \mc{I}^c}$ still forms a frame. If this condition holds, the original frame $\{L_{pk} \tilde{h}_i\}_{(k,i) \in \mc{I} \cup \mc{I}^c}$ is said to be \emph{robust to $e$ erasures}. It is shown in~\cite{goyal2001quantized} that uniform tight frames are robust to 1 erasure, whereas general tight frames may not possess this property. In our work, we show that the tight frames \( \mathcal{R}_{1,N} \) and \( \mathcal{R}_{2,N} \) are robust to 1 erasure, despite not being uniform. 	Furthermore, we establish conditions under which these tight frames are robust to 2 erasures.
			\subsubsection{Filter bank erasures} 	
		Let \(\mathcal{U}_j(m)\), for \(m \in \mathcal{J}_d\) and \(j \in \mathcal{I}_K\), denote the analysis polyphase matrix obtained by removing the \(j\)-th Ramanujan sum \(c_{q_j}\) from the filter bank based on the Ramanujan sums \(c_{q_1}, c_{q_2}, \ldots, c_{q_K}\). This is a \((K-1) \times p\) matrix obtained by removing the \(j\)-th row from \(\mathcal{U}(m)\).  The following result provides an equivalent condition under which the filter bank based on the Ramanujan sums \( c_{q_1}, c_{q_2}, \ldots, c_{q_K} \) is robust to 1 filter bank erasure, under the assumption that it forms a tight frame.
			
			\begin{prop}\label{prop: Ramanujanerasure}
				Let the collection \( \mc R_{p,N} \) (defined in \eqref{eq:rpk}) be a tight frame for \( \ell^2(\mathbb{Z}_N) \) with frame bound \( A \). Then, the filter bank based on Ramanujan sums \( c_{q_1}, c_{q_2}, \ldots, c_{q_K} \) is robust to 1 filter bank erasure if and only if, for all \( m \in \mc J_d \) and \( j \in \mc I_K \), the following condition holds:
				\begin{equation*}\label{eq:robusterasure}
					1 - \frac{d}{A} \left( |\mc Z c_{q_j}(m,0)|^2 + |\mc Z c_{q_j}(m,1)|^2 + \cdots + |\mc Z c_{q_j}(m,p-1)|^2 \right) \neq 0.
				\end{equation*}
			\end{prop}
			\begin{proof}
				By Lemma \ref{lem:mainlem}, it is clear that the filter bank, after the  removal of the channel corresponding to $c_{q_j},$ exhibits a frame only if \(\operatorname{rank}\,\,\mathcal{U}_j(m) = p\).	We therefore show that $\operatorname{rank}\, \mc U_j (m)^\asterisk \mc U_j (m) = \text{rank}\,\, \mc U_j (m)=p$ for $j \in \mc I_K.$ Note that, for any matrix $M$ with rows $R_i, i\in \mc I_L,$ we can write
				$M^\asterisk M=\sum_{i=1}^LR_i^\asterisk R_i.$
			Thus,	if $C_{q_i}(m)$ denotes the $i$-th row of $\mc U(m),$ i.e., $$C_{q_i}(m)=\sqrt{d}\begin{bmatrix}
					\ol{\mc Zc_{q_i}(m,0)}&	\ol{\mc Zc_{q_i}(m,0)}&
					\cdots&
					\ol{\mc Zc_{q_i}(m,p-1)}
				\end{bmatrix},\quad m \in \mc J_d,$$  then for $k \in \mc I_K$ and $m \in \mc J_d$, we have
				\begin{equation*}
					\begin{aligned}
						\mc U_j (m)^\asterisk \mc U_j (m)&= \sum_{i=1}^p C_{q_i}(m)^\asterisk C_{q_i}(m) - C_{q_j}(m)^\asterisk C_{q_j}(m)
						&=\mc U(m)^\asterisk\mc U(m)-C_{q_j}(m)^\asterisk C_{q_j}(m).
					\end{aligned}
				\end{equation*}
					Taking the  inverse on both sides of the above equation and using the identity $(P-QS)^{-1}=P^{-1}+ P^{-1}Q(1-SP^{-1}Q)^{-1}SP^{-1}$ along with  $\mc U^\asterisk(m)\mc U(m)=A I_p$ for $m\in \mc J_d,$  we get 
				\begin{equation}
					\begin{aligned}
						(	\mc U_j (m)^\asterisk \mc U_j (m))^{-1}& =( \mc U(m)^\asterisk\mc U(m)-C_{q_j}(m)^\asterisk C_{q_j}(m))^{-1}=\left( AI_p -C_{q_j}(m)^\asterisk C_{q_j}(m) \right)^{-1}\\
						&=\frac{1}{A}I_p+ \frac{1}{A}I_p C_{q_j}(m)^\asterisk \left(1-C_{q_j}(m)\frac{1}{A}I_p C_{q_j}(m)^\asterisk\right)^{-1}C_{q_j}(m)\frac{1}{A}I_p\\
						&=\frac{1}{A}I_p + \frac{1}{A^2}\left(1-\frac{1}{A}C_{q_j}(m) C_{q_j}(m)^\asterisk\right)^{-1}C_{q_j}(m)^\asterisk C_{q_j}(m),
					\end{aligned}
				\end{equation}
				$j \in \mc I_K$ and $m \in \mc J_d.$
				Note that $\mc U_j (m)^\asterisk \mc U_j (m)$ is invertible if and only if $1-\frac{1}{A}C_{q_j}(m) C_{q_j}(m)^\asterisk \neq 0,$ which is equivalent to
				$1-\frac{d}{A}\sum_{n=0}^{p-1}|\mc Zc_{q_j}(m,n)|^2 \neq 0$ for $j \in \mc I_K$ and $m \in \mc J_d.$ 
				Hence the claim follows. 
			\end{proof}
			We have the following result.
		\begin{thm}\label{thm:filtererasure}
		The  filter bank based on Ramanujan sums $c_{q_1},c_{q_2}, \hdots, c_{q_K}$  is
				not robust to one filter bank erasure for any decimation ratio.
			\end{thm}
			\begin{proof}
				Note that for $p>2,$ and for $p=2$ with even $d$,  $\mc R_{p,N},$ as shown in Theorem \ref{thm:tight}, does not form a frame for $\ell^2(\mathbb{Z}_N),$ and hence the corresponding filter bank  is trivially not robust to one filter bank erasure for these cases.
				
				Let  $p=2$ with $d$ odd.  Note that for $q_i=1$  and $m=0,$  we have
				$$|\mc Zc_{1}(0,0)|^2 + |\mc Zc_{1}(0,1)|^2 =\frac{1}{d}\bigg(\big|\sum_{\ell=0}^{d-1}c_{1}(2\ell)\big|^2 + \big|\sum_{\ell=0}^{d-1}c_{1}(2\ell+1)\big|^2\bigg)=\frac{2d^2}{d}=2d=N, $$
				where we have used $c_1=(1,1, \hdots, 1).$ The above equation implies 
				$
				N\big(1-\frac{d}{2d^2}\sum_{n=0}^1|\mc Zc_1(0,n)|^2\big)=0$. 
			Similarly, for $p=1, q_i=1$,  and $m=0,$ we have 
				$|\mc Zc_{1}(0,0)|^2  =\frac{1}{N}\big|\sum_{\ell=0}^{N-1}c_{1}(\ell)\big|^2 =\frac{N^2}{N}=N$.  Then by Proposition \ref{prop: Ramanujanerasure}, the corresponding filter banks are not robust 1 filter bank erasure for $p = 2$ with $d$ odd, and for $p = 1$.  On combining everything, the claim follows.
			\end{proof}
			Theorem \ref{thm:filtererasure} suggests that 	filter banks associated with the tight frames $\mc R_{1,N}$ and $\mc R_{2,N}$ do not allow perfect reconstruction  after  the erasure of even one channel. This raises the question: if instead of  one channel, several vectors are deleted from \(\mc R_{1,N}\) or \(\mc R_{2,N}\), does the remaining collection still retain its frame properties? We address this question in the following subsection.
		\subsubsection{Frame erasures}
			In this subsection,	we prove that the tight frames $\mc R_{1,N}$ and $\mc R_{2,N}$ are robust to 1 erasure. Moreover, we have also obtained the conditions under which tight frames $\mc R_{1,N}$ and $\mc R_{2,N}$ are robust to 2 erasures.
			
	\begin{thm}\label{thm:erasure1}
	Let \( N \geq 2 \). Then, the following statements hold:
				\begin{itemize}
					\item[$(i)$] The system $\mc R_{1,N}$ is robust to 1 erasure.
					
					\item[$(ii)$] For odd d, the system $\mc R_{2,N}$ is robust to 1 erasure. 
				\end{itemize}
			\end{thm}

			\begin{proof}
				We begin by proving part~(i). For any \( (\ell, j) \in \mathcal{J}_{N} \times \mathcal{I}_{K} \), define
				\[
				\mathcal{R}_{1,N}^{(\ell,j)} := \mathcal{R}_{1,N} \setminus \{L_{\ell} c_{q_j}\}.
				\]
				The system \( \mathcal{R}_{1,N} \) is robust to one erasure if and only if \( \mathcal{R}_{1,N}^{(\ell,j)} \) forms a frame for \( \ell^2(\mathbb{Z}_N) \) for every such index pair \( (\ell, j) \). 
				By Lemma~\ref{lem:frameunion}, it suffices to verify that the collection
				\(
				\{L_k c_{q_j}\}_{\substack{k \in \mathcal{J}_N \\ k \neq \ell}}
				\)
				forms a frame for the subspace \( S_{1, q_j} \), since for all \( i \neq j \), the collections \( \{L_k c_{q_i}\}_{k \in \mathcal{J}_N} \) continue to be frames for \( S_{1, q_i} \) by Theorem~\ref{thm:tight} and Lemma~\ref{lem:frameunion}.
				
				Observe that for any \( N \geq 2 \), we have \( N - 1 \geq \phi(q_i) \) for all \( i \in \mathcal{I}_K \). Thus, after one erasure, the collection \(  \{L_k c_{q_j}\}_{\substack{k \in \mathcal{J}_N \\ k \neq \ell}} \) contains at least \( \phi(q_j) \) elements. Moreover, due to periodicity, it must include a set of \( \phi(q_j) \) consecutive shifts of \( c_{q_j} \), i.e.,
				\(
				\{L_k c_{q_j}\}_{k = \ell+1}^{\ell+\phi(q_j)} \subseteq \{L_k c_{q_j}\}_{\substack{k \in \mathcal{J}_N \\ k \neq \ell}}.
				\)
	Therefore, this collection contains a basis of \( S_{1,q_j} \), and hence forms a frame for \( S_{1,q_j} \). Consequently, by Lemma~\ref{lem:frameunion}, the system \( \mathcal{R}_{1,N}^{(\ell,j)} \) is a frame for \( \ell^2(\mathbb{Z}_N) \) for every \( (\ell, j) \in \mathcal{J}_N \times \mathcal{I}_K \), completing the proof of (i).
				
				The proof of part~(ii) follows analogously, by applying Theorem~\ref{thm:basis} and Lemma~\ref{lem:frameunion}.
			\end{proof}
\begin{thm}\label{thm:erasure2}
		The following statements hold:
		\begin{itemize}
			\item[$(i)$] If \( N - 1 \geq 2\phi(N) \), then the system \( \mathcal{R}_{1,N} \) is robust to 2 erasures.
			
			\item[$(ii)$] If \( d \) is odd and \( d - 1 \geq 2\phi(N) \), then the system \( \mathcal{R}_{2,N} \) is robust to 2 erasures.
		\end{itemize}
	\end{thm}	\begin{proof}
				We first prove (i). For indices $(k_i,i), (k_j,j) \in \mc J_{N} \times \mc I_K$ and $p \in \{1,2\}$, we define  $\mc R_{p,N}^{\{(k_i,i), (k_j,j)\}}:=\mc R_{p,N}\backslash \{L_{k_i}c_{q_{i}}, L_{k_j}c_{q_{j}}\}$.	If $i \neq j,$ then the claim follows from Theorem \ref{thm:erasure1} (i) as there is at most one erasure in each of the individual collections of $\mc R_{1,N}.$ We divide the case $i=j$ into two subcases: $ |k_i-k_j|> \phi(q_i)$ and $|k_i-k_j|\leq  \phi(q_i).$ 
				Now	if $|k_i-k_j|>\phi(q_i),$ then $\{L_{k_i+1}c_{q_i}, L_{k_i+2}c_{q_i}, \hdots, L_{k_i-1}c_{q_i}\}$ is a subcollection of $\{L_{k}c_{q_i} : k \in \mc J_{N}, k \neq k_i, k_j\},$ containing at least $\phi(q_i)$ consecutive elements and hence contains a basis of $S_{1,q_i}$. Also if $|k_i-k_j|\leq \phi(q_i),$ then $\{L_{k_j+1}c_{q_i}, L_{k_j+2}c_{q_i}, \hdots, L_{N-1}c_{q_i}, \hdots, L_{k_i-1}c_{q_i}\}$ is a subcollection of $\{L_{k}c_{q_i} : k \in \mc J_{N}, k \neq k_i, k_j\},$ containing at least $\phi(q_i)$ consecutive elements since the total number of  elements in this subcollection~is:
				$$(N-1) -(k_j+1)+1 + k_i=N-1+k_i-k_j \geq N-1-\phi(q_i)   \geq \phi(q_i),$$ where the last inequality follows from the hypothesis, using $\phi(q_i) \leq \phi(N)$ for $i \in \mathcal{I}_K$.
				Therefore this subcollection also	contains a basis of $S_{1,q_i}$.
				Thus, in both the cases $\{L_{k}c_{q_i} : k \in \mc J_{N}, k \neq k_i, k_j\}$ is a frame for $S_{1,q_i}$ and hence each of the individual collections in $\mc R_{1,N}^{\{(k_i,i), (k_j,j)\}}$  is a frame for their corresponding subspaces. Then, the claim follows from Lemma \ref{lem:frameunion}.

We now prove (ii). As in the proof of (i),  if $i \neq j,$ the claim follows from Theorem \ref{thm:erasure1} (ii). For $i=j,$ we divide the proof into two subcases: $ |k_i-k_j|> \phi(q_i),$ and  $|k_i-k_j|\leq  \phi(q_i).$ For the first case, the subcollection
				$
				\{L_{2(k_i+1)}c_{q_i}, L_{2(k_i+2)}c_{q_i}, \ldots, L_{2(k_j-1)}c_{q_i}\}
				$
				of the collection 	$
				\{L_{2k}c_{q_i} : k \in \mc J_{N}, k \neq k_i, k_j\},
				$ contains at least \(\phi(q_i)\) consecutive elements and hence contains a basis of $S_{2,q_i}$ by using Theorem \ref{thm:basis}. For the second case, the subcollection $
				\{L_{2(k_j+1)}c_{q_i}, L_{2(k_j+2)}c_{q_i}, \ldots, L_{2(d-1)}c_{q_i}, \ldots, L_{2(k_i-1)}c_{q_i}\}
				$	of the collection $
				\{L_{2k}c_{q_i} : k \in \mc J_d, k \neq k_i, k_j\}
				$
				contains at least \(\phi(q_i)\) consecutive elements since 
				$$(d-1)-(k_i+1) +1 + k_i=d-1+k_i-k_j>d-1 -\phi(q_i)\geq \phi(q_i),$$
				using reasoning similar to that in (i). Consequently, each collection in $\mc R_{2,N}^{\{(k_i,i), (k_j,j)\}}$ is a frame for its corresponding subspace. The result then follows from Lemma~\ref{lem:frameunion}.
			\end{proof}
			From the proofs of Theorems~\ref{thm:erasure1} and~\ref{thm:erasure2}, it follows that a tight frame system $\mathcal{R}_{p,N}$ (for $p \in \{1,2\}$) is robust to $k$ erasures if and only if each of its constituent tight frames $\{L_{pk} c_{q_i}\}_{k = 0}^{d-1}$ is robust to $k$ erasures, for all $i \in \mathcal{I}_K$.  
			Thus, the problem of determining the robustness of $\mathcal{R}_{p,N}$ reduces to analyzing the robustness properties of these individual collections.  
			The next remark shows that, for $p \in \{1,2\}$, each such collection can be interpreted both as a dynamical frame and as a harmonic frame. These frames have a wide range of applications and are studied widely in the literature; see, for example,~\cite{aldroubi2017dynamical, CHRISTENSEN2018dynamical, waldronharmonic2005}.

		\begin{remar}\label{rem:dynamical}
		Let $p\in\{1,2\}$.	Then, the tight frame $\{L_{pk} c_{q_i}\}_{k=0}^{d-1}$ can be realized as the orbit of a single vector under the action of a unitary operator.  
			Indeed, let $T : \ell^2(\mathbb{Z}_N) \to \ell^2(\mathbb{Z}_N)$ be defined by
			\[
			(Tx)(n) = x\big((n - p) \bmod N\big), \quad n \in \mathbb{Z}_N,
			\]
			so that $T = L_p$ is the circular shift by $p$. Then
			\[
			\{L_{pk} c_{q_i}\}_{k=0}^{d-1} = \{T^n c_{q_i}\}_{n=0}^{d-1},
			\]
			showing that it is a dynamical frame generated by $T$.
			
			Furthermore, the same frame can also be interpreted as a harmonic frame for  the Ramanujan subspace $S_{p,q_i}$.  
		For that,	let $U(S_{p,q_i})$ denote the group of unitary operators on $S_{p,q_i}$.  
			Since $S_{p,q_i}$ is invariant under circular shifts, the restriction $L_m|_{S_{p,q_i}}$ belongs to $U(S_{p,q_i})$ for all $m \in \mathbb{Z}_N$.  
			Define the subgroup
			\[
			G_p := \{L_{pk}\mid_{S_{p,q_i}} : k = 0,1,\dots, d-1\} \subset U(S_{p,q_i}),
			\]
			which is abelian because circular shifts commute.  
			The collection $\{L_{pk} c_{q_i}\}_{k=0}^{d-1}$ is precisely the $G_p$–orbit of $c_{q_i}$,
			\[
			\{L_{pk} c_{q_i}\}_{k=0}^{d-1} = \{g\,c_{q_i} : g \in G_p\},
			\]
			and satisfies $\|L_{pk}c_{q_i}\| = \|c_{q_i}\|$ for all $k$.  
			By \cite[Theorem 5.4]{waldronharmonic2005}, such a frame is harmonic.
		\end{remar}
	
			\begin{remar}\label{rem:erasures}
				Note that for $2$ erasures, it is necessary that for any $q_i |N, $ $N-2\geq \phi(q_i)$ 	 in case of the erasures from $\mc R_{1,N}$ and $d-2\geq \phi(q_i)$ in case of erasures from $\mc R_{2,N}$ since to accommodate a basis after erasures, we need at least $\phi(q_i)$ elements in each of the individual collections of the erased versions of $\mc R_{1,N}.$ and $\mc R_{2,N}.$ Both of these conditions are implied by the conditions in the hypotheses of  Theorem \ref{thm:erasure2} (i) and Theorem \ref{thm:erasure2} (ii), respectively. Indeed, we have $N-1 \geq 2 \phi(N)$ from Theorem \ref{thm:erasure2} (i). This implies $N-2 \geq 2 \phi(N)-1 \geq \phi(N) \geq \phi(q_i)$ for any $q_i|N.$ Similarly, the necessary condition $d-2 \geq \phi(q_i)$  for  $q_i|N$ also follows from the condition in Theorem \ref{thm:erasure2} (ii). 
			\end{remar}
			\begin{exmp}
				$(a).$ Let \( N = 4 \). The matrix \( M \) representing \( \mathcal{R}_{1,4} \) is:
		\[
				M = \begin{bmatrix}
					L_0c_1 & \cdots & L_3c_1 &
					L_0c_2 & \cdots & L_3c_2 &
					L_0c_4 & \cdots & L_3c_4
				\end{bmatrix}_{4 \times 12},
				\]
				where \( c_1 = (1,1,1,1)^T \), \( c_2 = (1,-1,1,-1)^T \), and \( c_4 = (2,0,-2,0)^T \). Due to the orthogonal decomposition of $\ell^2(\mathbb Z_4),$ the set of columns $D_{q_i}:=\{L_0c_{q_i}, L_1c_{q_i}, L_2c_{q_i},L_3c_{q_i}\}$ is orthogonal to $D_{q_j}:=\{L_0c_{q_j}, L_1c_{q_j}, L_2c_{q_j},L_3c_{q_j}\}$ for $q_i,q_j \in \{1,2,4\}, \text{ and } j \neq i.$ After removing one column from some \( D_{q_i} \), each \( D_{q_i} \) still contains at least 3 columns, preserving \( \phi(q_i) \) consecutive columns for each divisor \( q_i = 1, 2, 4 \), since  \( \phi(q_i) < 3 \) holds for each $q_i.$ Then, by Lemma \ref{lem:frameunion}, $R_{1,4}$ is robust to 1 erasure. But note that after the  erasure of columns $L_0c_4$ and $L_2c_4$ from $D_4,$ it will be left with $\{L_1c_4,L_3c_4\},$ which are linearly dependent and thus not able to span $S_{1,4}$ since $\dim S_{1,4}=\phi(4)=2 > 1.$ This shows $\mc R_{1,4}$ is not robust to 2 erasures. It is due to the fact that the necessary condition  $N-2 \geq \phi(q_i)$ for $q_i|4$ (discussed in Remark \ref{rem:erasures})  fails here for $q_i=4,$ since $ 2\ngeq 2 \times \phi(4)=4.$
				
				$(b).$ Let \( N = 210 \). Then for $p=2, d=105$ is odd. Also, $105-1>2\times \phi(210)=96.$   Thus, both the conditions given in Theorem \ref{thm:erasure1} (ii) and Theorem \ref{thm:erasure2} (ii) are satisfied and hence $\mc R_{2,N}$ is robust to 1 erasure and 2 erasures.
			\end{exmp}

\subsection{Fusion frames based on Ramanujan subspaces}\label{sub:fusion} Fusion frames, also known as frames of subspaces, arise naturally in the context of distributed processing. Unlike classical frames, which analyze signals via one-dimensional projections, fusion frames perform analysis by projecting signals onto multidimensional subspaces.  For further details on fusion frames, we refer the reader to~\cite{candes2016robust, casazza2012}.

In this subsection, we show that the tight frames $\mathcal{R}_{1,N}$ and $\mathcal{R}_{2,N}$ admit an interpretation as tight fusion frames in terms of Ramanujan subspaces.

To that end, we begin by recalling the definition of a fusion frame adapted to our setting.
\begin{defn}\label{def:fusionframe}
Let \(S_{p,q_i}\) denote the Ramanujan subspace corresponding to the divisors $p$ and  \(q_i\) of \(N\), for \(i \in \mathcal{I}_K\), as defined in \eqref{eq:ramanujansubspaces}, and let \((w_i)_{i\in \mc I_K} \subseteq \mathbb{R}^{+}\) be a corresponding set of positive weights. Then the  family $\left\{(S_{p,q_i}, w_i)\right\}_{i\in \mc I_K}$ is called a \emph{fusion frame} for \(\ell^2(\mathbb Z_N)\) if there exist constants \(0 < A \leq B < \infty\) such that
	\[
	A\|x\|^2 \leq \sum_{i\in \mc I_K} w_i^2 \|P_{S_{p,q_i}}(x)\|^2 \leq B \|x\|^2 \quad \text{for all } x \in \ell^2(\mathbb Z_N),
	\]
	where for a subspace $W\subseteq \ell^2(\mathbb Z_N)$, \(P_W\) denotes the orthogonal projection onto \(W\).
	
	The constants \(A\) and \(B\) are referred to as the \emph{lower} and \emph{upper fusion frame bounds}, respectively. If \(A = B\), the fusion frame is said to be  a \emph{tight fusion frame with bound $A$}. When \(A = B = 1\), the fusion frame is called a \emph{Parseval fusion frame}. If the weights satisfy \(w_i = 1\) for all \(i\in \mc I_K\), we simply write \(\{S_{p,q_i}\}_{i\in \mc I_K}\).
\end{defn}
We have the following result.
\begin{thm}\label{thm:fusion}
The  following statements hold:
\begin{itemize}
\item [$(i)$] The family $\left\{S_{1,q_i}\right\}_{i\in \mc I_K}$ is a Parseval fusion frame.
\item [$(ii)$] For odd $d$, the family $\left\{S_{2,q_i}\right\}_{i\in \mc I_K}$ is a Parseval fusion frame.
\end{itemize}
\end{thm}

The following proposition, adapted from~\cite[Corollary 13.2]{casazza2012} to our setting, establishes a connection between the tight frame property of $\mathcal{R}_{p,N}$ and the tight fusion frame structure of the Ramanujan subspaces \( S_{p,q_i} \), for \( i \in \mathcal{I}_K \).  For brevity, the proof is omitted.
\begin{prop}\label{prop:fusion}
Let the collection $\{L_{pk}c_{q_i}\}_{k \in \mc J_{d}}$ be a tight frame for \(S_{p,q_i}\) with frame bound $A$ for each \(i \in \mc I_K\). Then, the following conditions are equivalent.
	\begin{itemize}
		\item [$(i)$] The family $\left\{S_{p,q_i}\right\}_{i\in \mc I_K}$ is a tight fusion frame for \(\ell^2(\mathbb Z_N)\) with  bound $C$.
		\item [$(ii)$] The system $\mc R_{p,N}$ is a tight frame for \(\ell^2(\mathbb Z_N)\) with  bound $AC$.
	\end{itemize}
\end{prop}
\begin{proof}[Proof of Theorem \ref{thm:fusion}]
We first prove part~(i). By Theorem~\ref{thm:tight}(i), the system \(\mathcal{R}_{1,N}\) forms a tight frame for \(\ell^2(\mathbb{Z}_N)\) with frame bound \(N^2\). Consequently, Lemma~\ref{lem:frameunion} implies that each individual collection \(\left\{L_{k}c_{q_i}\right\}_{k \in \mathcal{J}_N}\) forms a tight frame for the subspace \(S_{1,q_i}\)  with the same frame bound \(N^2\).  Since \(AC = N^2\) and \(A = N^2\), Proposition~\ref{prop:fusion} ensures that the family \(\left\{S_{1,q_i}\right\}_{i\in \mc I_K}\) forms a tight fusion frame for \(\ell^2(\mathbb{Z}_N)\) with frame bound $C=1$. Hence, \(\left\{S_{1,q_i}\right\}_{i\in \mc I_K}\) is a Parseval fusion frame for \(\ell^2(\mathbb{Z}_N)\), completing the proof of part~(i).

The proof of part~(ii) follows analogously by applying Theorem~\ref{thm:tight}(ii), Lemma~\ref{lem:frameunion}, and Proposition~\ref{prop:fusion}.
\end{proof}
As previously discussed, if the frame system $\mathcal{R}_{p,N}$ is robust to $k$ erasures, then each individual collection $\{L_{pk}c_{q_i}\}_{k\in \mathcal{J}_d}$, for $i \in \mathcal{I}_K$, must also be robust to $k$ erasures. In light of this, Theorems~\ref{thm:erasure1} and~\ref{thm:erasure2} (for $p \in \{1,2\}$) provide sufficient conditions under which the collection $\{L_{pk}c_{q_i}\}_{k\in \mathcal{J}_d}$ remain robust to one or two erasures. The following result, motivated by~\cite[Theorem 4.1]{casazza2008}, builds upon these conditions and shows that the deletion of a prescribed number of local frame vectors from each individual collection preserves the fusion frame property.
\begin{thm}\label{thm:fusionerasure}
	Let \( L_i \subset \mathcal{J}_d \) for each \( i \in \mathcal{I}_K \). For the divisors $p$ and $q_i$ of $N$, define the subspaces
	\[
	\widetilde{S}_{p,q_i} := \operatorname{span} \left\{ L_{pk} c_{q_i} : k \in \mathcal{J}_d \setminus L_i \right\}, \quad i \in \mathcal{I}_K.
	\]
	Then the following statements hold:
	
	\medskip
	
	\noindent (a) Single erasure case: Assume \( \# L_i \leq 1 \) for all \( i \in \mathcal{I}_K \).
	\begin{itemize}
		\item[$(i)$] If \( p = 1 \) and \( N \geq 2 \), then the family \( \{ \widetilde{S}_{1,q_i} \}_{i \in \mathcal{I}_K} 
		\) forms a fusion frame for \( \ell^2(\mathbb{Z}_N) \) 
		
		\item[$(ii)$] If \( p = 2 \),  \( N \geq 2 \),   \( d \) is odd,  then the family \( \{ \widetilde{S}_{2,q_i} \}_{i \in \mathcal{I}_K} \) forms a fusion frame for \( \ell^2(\mathbb{Z}_N) \). 
	\end{itemize}
	
	\medskip
	
	\noindent (b) Double erasure case: Assume \( \# L_i \leq 2 \) for all \( i \in \mathcal{I}_K \).
	\begin{itemize}
		\item[$(i)$] If \( p = 1 \) and \( N - 1 \geq 2\phi(N) \),  then the family \( \{ \widetilde{S}_{1,q_i} \}_{i \in \mathcal{I}_K} \) forms a fusion frame for \( \ell^2(\mathbb{Z}_N) \). 
			
		\item[$(ii)$] If \( p = 2 \),  \( d \) is odd, and \( d - 1 \geq 2\phi(N) \), then the family \( \{ \widetilde{S}_{2,q_i} \}_{i \in \mathcal{I}_K} \) forms a fusion frame for \( \ell^2(\mathbb{Z}_N) \). 
	\end{itemize}

Moreover, in all cases above, the collections \( \{L_{pk} c_{q_i}\}_{k \in \mathcal{J}_d \setminus L_i} \) form frames for $S_{p,q_i}$ for $i \in \mc I_K$ for respective values of $p$, and the associated fusion frame bounds are given by
\[
\frac{\min\limits_{i \in \mathcal{I}_K} A_{p,i}}{p d^2} \quad \text{and} \quad 1,
\]
where \( A_{p,i} \) denotes the lower frame bound of \( \{L_{pk} c_{q_i}\}_{k \in \mathcal{J}_d \setminus L_i} \).
\end{thm}
\begin{proof}
We begin with proving (a). Let \( x \in \ell^2(\mathbb{Z}_N) \) be arbitrary. 
	For each \( i \in \mathcal{I}_K \), set \( y_i := P_{S_{1,q_i}} x \). Since \( \widetilde{S}_{1,q_i} \subseteq S_{1,q_i} \), we decompose
	\[
	y_i = u_i + v_i, \quad \text{where } u_i := P_{\widetilde{S}_{1,q_i}} y_i \in \widetilde{S}_{1,q_i}, \quad v_i \in S_{1,q_i} \ominus \widetilde{S}_{1,q_i}.
	\]
	For any \( k \in \mathcal{J}_N \setminus L_i \), we have \( L_k c_{q_i} \in \widetilde{S}_{1,q_i} \), and thus
	\[
	\langle y_i, L_k c_{q_i} \rangle = \langle u_i, L_k c_{q_i} \rangle,
	\]
	since \( v_i \perp \widetilde{S}_{1,q_i} \). Using the fact that the  collection \( \{L_k c_{q_i}\}_{k \in \mathcal{J}_N} \) forms a tight frame for \( S_{1,q_i} \) with bound \( N^2 \) (due to Theorem \ref{thm:tight}(i) and Lemma \ref{lem:frameunion}), we estimate:
	\begin{equation*} \label{eq:u_i-energy}
		\sum_{k \in \mathcal{J}_N \setminus L_i} \left| \langle y_i, L_k c_{q_i} \rangle \right|^2 
		= \sum_{k \in \mathcal{J}_N \setminus L_i} \left| \langle u_i, L_k c_{q_i} \rangle \right|^2 
		\leq \sum_{k \in \mathcal{J}_N} \left| \langle u_i, L_k c_{q_i} \rangle \right|^2 = N^2 \|u_i\|^2.
	\end{equation*}
Summing over \( i \in \mathcal{I}_K \) and using the fact that for $N\geq 2$ and $\#_{L_i}\leq 1$, the collection \( \{L_{k} c_{q_i}\}_{k \in \mathcal{J}_N \setminus L_i} \) forms a frame for $S_{1,q_i}$ for $i \in \mc I_K$ due to Theorem \ref{thm:erasure1}(i) with bound \( A_{1,i} \), we obtain:
	\[
	N^2 \sum_{i=1}^K \| P_{\widetilde{S}_{1,q_i}} x \|^2 
	= N^2 \sum_{i=1}^K \| u_i \|^2 
	\geq \sum_{i=1}^K \sum_{k \in \mathcal{J}_N \setminus L_i} \left| \langle y_i, L_k c_{q_i} \rangle \right|^2\geq \sum_{i=1}^K A_{1,i} \| P_{S_{1,q_i}} x \|^2 \geq \left( \min_{i \in \mathcal{I}_K} A_{1,i} \right) \sum_{i=1}^K \| P_{S_{1,q_i}} x \|^2.
	\]
Finally, since \( \{ S_{1,q_i} \}_{i \in \mathcal{I}_K} \) forms a Parseval fusion frame for \( \ell^2(\mathbb{Z}_N) \) by Theorem~\ref{thm:fusion}(i), we conclude that
	\(
	\sum_{i=1}^K \| P_{S_{1,q_i}} x \|^2 = \|x\|^2,
	\)
	and thus
	\[
	\sum_{i=1}^K \| P_{\widetilde{S}_{1,q_i}} x \|^2 \geq \frac{ \min_{i \in \mathcal{I}_K} A_{1,i} }{N^2} \|x\|^2.
	\]
The upper frame bound \( 1 \) is immediate, since projections are norm-contracting and \( \widetilde{S}_{1,q_i} \subseteq S_{1,q_i} \). This proves (a)(i).
	
When \( p = 2 \), the same argument applies, as the collection \( \{L_{2k} c_{q_i}\}_{k \in \mathcal{J}_d} \) forms a tight frame for \( S_{2,q_i} \) with frame bound \( 2d^2 \) (Theorem~\ref{thm:tight}(ii) and Lemma~\ref{lem:frameunion}). Moreover, under the assumptions \( N \geq 2 \), odd \( d \), and \( \#_{L_i} \leq 1 \), the collection \( \{L_{2k} c_{q_i}\}_{k \in \mathcal{J}_d \setminus L_i} \) also forms a frame for \( S_{2,q_i} \) for all \( i \in \mathcal{I}_K \), by Theorem~\ref{thm:erasure1}(ii). This establishes part~(a)(ii).

The proof of part~(b) follows analogously by applying Theorem~\ref{thm:tight}, Lemma~\ref{lem:frameunion}, Theorem~\ref{thm:erasure2}, and Proposition~\ref{thm:fusion}.
\end{proof}

	In the following section, we introduce an uncertainty principle for signal representations in $\ell^2(\mathbb{Z}_N)$, comparing the canonical basis $\{e_n\}_{n \in \mathbb Z_N}$ with a tight frame $\mc R_{p,N}.$ We further explore the utility of this principle in addressing practical challenges, such as recovering signals with missing information  and signal denoising.

		\section{Uncertainty Principle and Signal recovery}\label{S3}

		Uncertainty principles describe limits on the simultaneous time- and frequency-localization of signals.
		A signal that is well-localized in frequency can be only so well-localized in time, and vice versa.  In this section, we present an uncertainty principle for pairs of signal representations in $\ell^2(\mathbb{Z}_N)$, with respect to the canonical basis $\{e_n\}_{n \in \mathbb{Z}_N}$ of $\ell^2(\mathbb{Z}_N)$ and a tight frame $\mc R_{p,N}$ involving Ramanujan sums. 
		We show that our uncertainty principle has applications in scenarios where some samples of the signal are lost during transmission from the analysis phase to the synthesis phase of a filter bank and in signal denoising.  Our signal recovery approach primarily focuses on utilizing filter banks with Ramanujan sums as filters and using their properties in signal recovery algorithms, which will be illustrated in Section~\ref{S4}.
 		
		In practical scenarios, some samples of the signal may be lost during its transmission. This is one of the common problems in signal processing and has been discussed by many researchers (see, for references, \cite{candes2006stable, montefusco2009nonlinear, tropp2007signal}). In our case, we assume that the samples $\{(x\, \asterisk \,c_{q_i})(pk) : (k,i) \in  M\}$ are lost for some fixed $M \subset \mc J_{d} \times \mc I_{\mc K}.$ We then look for the conditions under which the original signal $x$ can be recovered completely from the remaining samples.
		\begin{remar}
			Note that, if $N=kq$ for some positive integer $k,$ then by using the symmetric property of the Ramanujan sums, i.e., $c_q(n) = c_q(-n), n \in \mathbb Z_N,$  we get $$\tilde{c_q}(n)=c_q(-n)=c_q(N-n)=c_q(kq-n)=c_{q}(n),\,\,\, n \in \mathbb Z_N.$$
			Thus the collection $\{(x\, \asterisk \,c_{q_i})(pk) : (k,i) \in  \mc J_{d} \times \mc I_{\mc K}\},$ without the involution of $c_{q_i}, i \in \mc I_K,$ is the output from the filter bank based on Ramanujan sums $c_{q_1}, c_{q_2}, \hdots, c_{q_K}.$  
		\end{remar}
		
		We define \(M^c := (\mathcal{J}_{d} \times \mathcal{I}_{K}) \setminus M\) for any $M \subset \mc J_d \times \mc I_{K}$ and use \(\#_S\) to denote the cardinality of the set \(S\). 	We focus on the following problems:
		
		\textbf{Problem 1:} Under the assumption that $\mc R_{p,N}$ is a tight frame for $\ell^2(\mathbb Z_N),$  then under what conditions any signal $x\in \ell^2(\mathbb Z_N)$ can be reconstructed uniquely from the  truncated sum:
		\begin{equation}\label{eq:TMx}
			\mc T_{\mc J}x:=\frac{1}{A}\sum_{(k,i) \in \mc J}(x\, \asterisk \,c_{q_i})(pk)L_{pk}c_{q_i},
		\end{equation}
		where 	$\mc J \subset \mc J_d \times \mc I_{K}$ and $A$ is the tight-frame bound of  $\mc R_{p,N}.$
		
		Note that for a tight frame $\mc R_{p,N},$   a signal $x \in \ell^2(\mathbb Z_N)$ has the representation: $$x=\frac{1}{A}\sum_{i \in \mc I_K}\sum_{k \in \mc J_d}(x\, \asterisk \,c_{q_i})(pk)L_{pk}c_{q_i},$$ and thus \eqref{eq:TMx} is the truncated version of this representation. We prove the following result: 
	\begin{thm}\label{thm: uncertaintyrecovery}
Let \( x \in \ell^2(\mathbb{Z}_N) \) be a signal supported on a set \( \mc C \subseteq \mathbb{Z}_N \), and let \( \mc J \subset \mc J_d \times \mc I_K \) be such that the samples \( \{(x \ast c_{q_i})(pk) : (k,i) \in \mc J^c \} \) are lost or unavailable. Suppose that the collection \( \mc R_{p,N} \) (defined in \eqref{eq:rpk}) forms a tight frame for \( \ell^2(\mathbb{Z}_N) \). If
		\begin{equation}\label{uncertrecovcondtn}
			2 \#_{\mc J^c} \#_{\mc C} < p \left( \frac{d}{\phi(N)} \right)^2,
		\end{equation}
		then \( x \) can be uniquely recovered from \( \mc T_{\mc J} x \) (defined in~\eqref{eq:TMx}).
	\end{thm}
		
		Next, we present an application of our uncertainty principle in denoising.  Signal denoising is a widely studied problem in signal processing \cite{  elad2006image, protter2008image, kulkarni2023periodicity}. It refers to the process of removing or reducing noise from a signal to recover the original or a cleaner version of the signal,  which might be corrupted due to external interference, transmission errors, or system imperfections.  
		
		We begin with the  assumption that for some \( \mc M \subset \mc J_d \times \mc I_{K} \), the signal \( x \) belongs to the set:
		\begin{equation}\label{eq:SM}
			\mc S_{\mc M} := \left\{ h \in \ell^2(\mathbb Z_N) \mid (h \, \asterisk \, c_{q_i})(pk) = 0 \text{ for } (k,i) \in \mc M^c \right\}.
		\end{equation}Suppose the signal $x$ is sent to a receiver that captures it accurately, except over a set $\mc N \subset \mathbb Z_N$ where the signal is corrupted by noise $\eta.$  Then the received signal $y$ for  $k \in \mathbb Z_N$ is given by
		\begin{equation}\label{eq:noisysignal}
			y(k)=\begin{cases} 
				x(k)+n(k),& \text{ $k \in \mc N,$} \\
				x(k),	 & \text{$k \in \mc N^c,$}  
			\end{cases}
		\end{equation}
		where $\mc N^c:=\mathbb Z_N\backslash \mc N.$
		Our goal is to reconstruct \( x \) from the noisy signal \( y \), assuming some sparsity on the noise vector \( \eta \), i.e.,  the noise components \( \eta(k), k \in \mathbb Z_N\), are zero outside a given set.  
		
		\textbf{Problem 2}: Under what conditions is it possible to reconstruct any signal \( x \in \mc S_{\mc M} \) for some fixed \( \mc M \subset \mc J_d \times \mc I_{K} \), from the noisy signal \( y = x + \eta \) with some sparsity on the noise vector \( \eta \)?
		
In many practical situations, such as sensor failures, impulsive disturbances, or transmission errors, the noise typically affects only a few locations in the signal. This motivates modeling the noise vector \( \eta \) as sparse. See, for example, \cite{cheriandenoising2011}, where recovery under sparse noise assumptions is systematically analyzed. Problem 2 is addressed by the following theorem:
\begin{thm}\label{thm:recovfromnoisy}
	Let 
		\( x \in \mc S_{\mc M} \), as defined in~\eqref{eq:SM}, for some fixed set \( \mc M \subset \mc J_d \times \mc I_K \), and let \( \mc N = \{ m \in \mathbb{Z}_N : \eta(m) \neq 0 \} \) denote the support of the noise vector \( \eta \). Suppose that the collection \( \mc R_{p,N} \) (defined in \eqref{eq:rpk}) forms a tight frame for \( \ell^2(\mathbb{Z}_N) \). If
		\begin{equation}\label{condtn:2nmnt}
			2 \#_{\mc M} \#_{\mc N} < p \left( \frac{d}{\phi(N)} \right)^2,
		\end{equation}
		then the signal \( x \) can be exactly recovered from the noisy observation \( y = x + \eta \).
	\end{thm}
		
		Before proving Theorems \ref{thm: uncertaintyrecovery} and  \ref{thm:recovfromnoisy}, we provide the uncertainty principle concerning pairs of representations of signals  with respect to canonical basis $\left\{ e_n \right\}_{n\in \mathbb Z_n}$  and a tight frame $\mc R_{p,N}$ of $\ell^2(\mathbb Z_N).$
		\subsection{ Uncertainty principle using tight frames}\label{sub: uncer}	 For a non-zero signal $x \in \ell^2(\mathbb Z_N),$ we have the following representations in terms of canonical basis $\left\{ e_n \right\}_{n \in \mathbb Z_N}$  and a tight frame $\mc R_{p,N}:$ 
		\begin{equation}
			\label{Sunc: rep}
			x=\sum_{n=0}^{N-1}\alpha_ne_n= 		\frac{1}{A}\sum_{i=1}^K\sum_{m=0}^{d-1}	(x\, \asterisk \,c_{q_i})(pm)L_{pm}c_{q_i},
		\end{equation}
		where $A$ is the  frame bound of $\mc R_{p,N}.$ Then it is clear that for any $x \in \ell^2(\mathbb Z_N):$ $$\|x\|^2=\sum_{n=0}^{N-1}|\alpha_n|^2=\frac{1}{A}\sum_{i=1}^K\sum_{k=0}^{d-1}|	(x\, \asterisk \,c_{q_i})(pk)|^2,$$	
		where the last equality follows by noting that \(\big(x * c_{q_i}\big)(n) = \big(x * \tilde{c}_{q_i}\big)(n) = \langle x, L_n c_{q_i} \rangle\) for all \(n \in \mathbb{Z}\) and \(i \in \mc I_K.\)  We have the following result:
	\begin{thm}\label{thm:thmuncertaintycq}
Let the collection \( \mc R_{p,N} \) (defined in \eqref{eq:rpk}) form a tight frame for \( \ell^2(\mathbb{Z}_N) \). Then, for any signal \( x \in \ell^2(\mathbb{Z}_N) \), the representations given in~\eqref{Sunc: rep} with respect to the tight frame \( \mc R_{p,N} \) and the canonical basis \( \{ e_n \}_{n \in \mathbb{Z}_N} \) yield the following uncertainty inequalities:
		\begin{equation}\label{eq:uncerrelations}
			S_x + B_x \geq \frac{2d\sqrt{p}}{\phi(N)}, \quad \text{and} \quad S_x B_x \geq p \left( \frac{d}{\phi(N)} \right)^2,
		\end{equation}
		where \( S_x = \#\{ (x \ast c_{q_i})(pk) \mid k \in \mc J_d,\; i \in \mc I_K \} \) and \( B_x = \#\{ \alpha_n \mid n \in \mc J_N \} \).
	\end{thm}
		Before proving Theorem \ref{thm:thmuncertaintycq}, we first provide the following preliminary results.
	\begin{prop}\label{prop:tight}
			Let the collection \( \mc R_{p,N} \) (defined in \eqref{eq:rpk}) form a tight frame for \( \ell^2(\mathbb{Z}_N) \) with frame bound $A$. Then $A=pd^2.$
		\end{prop}
		\begin{proof}
			In view of \eqref{eq:Uk}, the trace $tr(\mc U^\asterisk(m)\mc U(m))$ can be simplified as follows:
			\begin{equation}\label{eq: trZaklambda}
				tr(\mc U^\asterisk(m)\mc U(m))=d\sum_{i=1}^{K}\sum_{n=0}^{p-1}|\mc Zc_{q_i}(m,n)|^2=\sum_{n=0}^{p-1}\lambda_n(m),
			\end{equation} where for each $m,$  $\{\lambda_n(m)\}_{n=0}^{p-1}$ is the set of eigenvalues of the matrix  $\mc U^\asterisk(m)\mc U(m).$ Since $\mc R_{p,N}$ is a tight frame with frame bound $A,$ therefore  $\mc U^\asterisk(m)\mc U(m)=A I_p$ for $m\in \mc J_d.$ From \eqref{eq: trZaklambda}, we get
			$
			\frac{pA}{d}=\sum_{i=1}^{K}\sum_{n=0}^{p-1}|\mc Zc_{q_i}(m,n)|^2.
			$
			Taking summation from $m=0 \text{ to } d-1$ on both sides of the last equality, the left  side become $pA$ and the right quantity becomes $\sum_{i=1}^K\|\mc Zc_{q_i}\|^2=\sum_{i=1}^K\|c_{q_i}\|^2.$ Finally we have,
			$
			A = \frac{1}{p}\sum_{i=1}^{K}\|c_{q_i}\|^2=\frac{N}{p}  \sum_{i=1}^K\phi(q_i)=\frac{N^2}{p}=pd^2,
			$
			where the second equality follows using the sum of squares property of Ramanujan sums (see, Proposition \ref{prop:properties}(v)), according to which  $\|c_{q_i}\|^2=N \times \phi(q_i).$	Hence the claim follows.
		\end{proof}	
		\begin{lem}\label{lem: modcqlem}
			Let $N$ be any positive integer. Then, for any  divisor $q$ of N, we have $|c_q(n)| \leq \phi(N), \,\, n \in \mathbb Z_N,$
			and hence 
			$\max\{|c_q(n)| : n \in \mathbb Z_N, \text{ $q$ divides $N$}\}=\phi(N).$ 
		\end{lem}
		\begin{proof}
			By the definition of $c_q(n),$ we have
			\begin{equation}\label{S5: modcq}
				\left| c_q(n)\right|= \bigg|  \sum_{\substack{k=1\\(k,q)= 1}}^q e^{2\pi i k n /q}\bigg| \leq \phi(q).
			\end{equation}	
			Note that
			$\phi(N)=\phi(k \cdot q)\geq \phi(k)\phi(q) \geq \phi(q),$
			where $N=k.q$ for some integer $k\geq1$ and 	the  inequality $\phi(k.q)\geq \phi(k)\phi(q)$ holds since for any positive integers $m$ and $n,$ $\phi(mn)=\phi(m)\phi(n)\ell/\phi(\ell)$ with $\ell=(m,n).$  
			Consequently  from \eqref{S5: modcq}, we  get $\left| c_q(n)\right| \leq \phi(N)$ for any divisor $q$ of $N.$ Since $c_{N}(0)=\phi(N),$ therefore $\max\{|c_q(n)| : n \in \mathbb Z_N, $\text{ $q$ divides $N$}$\}=\phi(N).$   Hence the claim follows.
		\end{proof}
		Now we are ready to provide the proof  of Theorem \ref{thm:thmuncertaintycq}.
		\begin{proof}[Proof of Theorem \ref{thm:thmuncertaintycq}]
			By using \eqref{Sunc: rep} and Proposition \ref{prop:tight} for $x \in \ell^2(\mathbb Z_N),$ we get
			\begin{equation*}
				\begin{aligned}
					\|x\|^2&=\bigg|\bigg\langle \sum_{n=0}^{N-1}\alpha_ne_n,  \frac{1}{A}\sum_{i=1}^K\sum_{m=0}^{d-1}	(x\, \asterisk \,c_{q_i})(pm)L_{pm}c_{q_i}\bigg\rangle\bigg|
					=\frac{1}{pd^2}\bigg|\sum_{n=0}^{N-1}\sum_{i=1}^K\sum_{m=0}^{d-1}\alpha_n\ol{(x\, \asterisk \,c_{q_i})(pm)}\langle e_n,L_{pm}c_{q_i} \rangle \bigg|\\
					&\leq \frac{1}{pd^2}\sum_{n=0}^{N-1}\sum_{i=1}^K\sum_{m=0}^{d-1}|\textcolor{blue}{\alpha_n}| |\ol{(x\, \asterisk \,c_{q_i})(pm)}| |\langle e_n,L_{pm}c_{q_i}\rangle| 
					\leq \frac{1}{pd^2}\sum_{n=0}^{N-1}|\alpha_n| \bigg(\sum_{i=1}^K\sum_{m=0}^{d-1}|\ol{(x\, \asterisk \,c_{q_i})(pm)}| \beta_o\bigg)
				\end{aligned}
			\end{equation*}
			where $\beta_o=\max \{|\langle e_n,L_{pm}c_{q_i} \rangle| : n \in \mathbb Z_N,   m \in \mathbb Z_d,  i \in \mc I_K\}.$ By the Cauchy-Schwarz inequality, we get
			\begin{equation*}
				\begin{aligned}
					\|x\|^2& \leq \frac{1}{pd^2}\bigg(\sum_{\substack{n=0\\ \alpha_n \neq 0}}^{N-1}|\alpha_n|^2\bigg)^{1/2} \times \bigg ( \sum_{\substack{n=0\\ \alpha_n \neq 0}}^{N-1}\big(\sum_{i=1}^K\sum_{m=0}^{d-1}|(x\, \asterisk \,c_{q_i})(pm)| \beta_o\big)^{2}\bigg)^{1/2}\\
					&\leq \frac{\beta_o}{pd^2}\|x\| \times \bigg(\sqrt{B_x} \times  \sum_{i=1}^K\sum_{m=0}^{d-1}1.|(x\, \asterisk \,c_{q_i})(pm)|\bigg)\\
					&\leq \frac{\beta_o}{pd^2}\|x\| \times \sqrt{B_x} \times \bigg(\sum_{i=1}^K\sum_{\substack{m=0 \\ (x\, \asterisk \,c_{q_i})(pm) \neq 0}}^{d-1}1^2\bigg)^{1/2} \times \bigg(\sum_{i=1}^K\sum_{\substack{m=0 \\ (x\, \asterisk \,c_{q_i})(pm)\neq 0}}^{d-1}|(x\, \asterisk \,c_{q_i})(pm)|^2\bigg)^{1/2}\\
					&	\leq \frac{\beta_o}{pd^2} \|x\| \times \sqrt{B_x} \times \sqrt{S_x} \times \sqrt{pd^2}\,\|x\|=\frac{\beta_o}{\sqrt{pd^2}} \sqrt{Sx B_x}\, \|x\|^2,
				\end{aligned}
			\end{equation*}
			where the last inequality follows using \eqref{Sunc: rep}. We finally get, $\sqrt{S_x W_x} \geq \frac{\sqrt{pd^2}}{\beta_o}.$ Using the fact that the arithmetic mean always
			dominates the geometric mean, we also get
			$$S_x +W_x \geq 2 \sqrt{S_x. W_x} \geq \frac{2\sqrt{pd^2}}{\beta_o}.$$
			Note that by fundamental theorem of group homomorphism, we have $\mathbb Z_N/\mathbb Z_d \cong \mathbb Z_p.$
			Thus for any $n \in \mathbb Z_N,$ there exist unique $k\in \mathbb Z_d$ and $\ell\in \mathbb Z_p$ such that  $n=pk+\ell.$ This implies  $e_n=e_{pk+\ell}=L_{pk}e_\ell$ and vice-versa. Therefore,  both the collections $\{e_n: n \in \mathbb Z_N\}$ and $\{L_{pk}e_\ell: k \in \mathbb Z_d, \ell \in  \mathbb Z_p\}$ are identical.			Then, for any   
			$m,k \in \mathbb Z_d,$  $i \in \mc I_K,$ and $\ell \in \mathbb Z_p,$ we have
			\begin{equation*}
				\begin{aligned}
					|\langle L_{pm}c_{q_i}, L_{pk}e_\ell \rangle_{\ell^2(\mathbb Z_N)}|& = |\langle  c_{q_i}, L_{p(k-m)}e_\ell\rangle_{\ell^2(\mathbb Z_N)}| = \bigg|\sum_{r=0}^{N-1}c_{q_i}(r)\ol{L_{pk^{'}}e_\ell(r)}\bigg|\\
					&=\bigg|\sum_{r=0}^{N-1}c_{q_i}(r)\ol{e_{\ell+{pk'}}(r)}\bigg|=\bigg|c_{q_i}(\ell+pk^{'})\bigg| \text{ for some $k^{'}\in \mathbb Z_d.$ }
				\end{aligned}
			\end{equation*}
			Using the above relation, it is clear that
			\begin{equation}\label{S5: maxeqmax}
				\begin{aligned}
					\max \{|\langle L_{pk}e_\ell,L_{pm}c_{q_i} \rangle|& :  m,k \in \mathbb Z_d,  i \in \mc I_K \text{ and } \ell \in \mathbb Z_p\}\\
					&=\max \{|c_{q_i}(\ell+pk)| : k \in \mathbb Z_d,  \ell \in \mathbb Z_p, i \in \mc I_K\}.
				\end{aligned}
			\end{equation}
			Then, $\beta_o$  is given by 
			\begin{equation*}
				\begin{aligned}
					\beta_o&=\max \{|\langle e_n,L_{pm}c_{q_i} \rangle| : n \in \mathbb Z_N,   m \in \mathbb Z_d,  i \in \mc I_K\}\\
					&=\max \{|c_{q_i}(\ell+pk)| : k \in \mathbb Z_d,  \ell \in \mathbb Z_p, i \in \mc I_K\},
				\end{aligned}
			\end{equation*}
			where the last equality follows from \eqref{S5: maxeqmax}. Now, by using Lemma \ref{lem: modcqlem}, we finally get $\beta_o =\phi(N).$
		\end{proof}
		
		In the following subsection, we demonstrate an application of our uncertainty relations \eqref{eq:uncerrelations} in the context of signal transmission with information loss. Specifically, we prove Theorem \ref{thm: uncertaintyrecovery}.
		\subsection{Proof of Theorem \ref{thm: uncertaintyrecovery}}\label{sub:missing} 	To prove Theorem \ref{thm: uncertaintyrecovery}, we use the $\ell_1$ reconstruction algorithm.  The $\ell_1$ norm of any signal $x \in \ell^2(\mathbb Z_N)$ is given by  $\|x\|_1=\sum_{n=0}^{N-1}|x(n)|.$ Consider the following optimization problem:
		\begin{equation}\label{opti1}
			\tilde{x} = \arg \min\limits_{x'} \left\{\|x'\|_1 \,:\, \mc T_{\mc J}x' =\mc T_{\mc J}x\right\},
		\end{equation}
		where for $x\in \ell^2(\mathbb Z_N),$ $\mc T_{\mc J}x$  is defined in \eqref{eq:TMx} and $\tilde{x}$ is the signal satisfying $\mc T_{\mc J}\tilde{x} =\mc T_{\mc J}x$ with minimum $\ell_1$ norm. This problem can be solved easily with less complexity and requires $O(N\hash_{\mc J}^2 log(N))$ operations. See, for reference, \cite{levy1981reconstruction, oldenburg1983recovery, santosa1986linear}.
		Before proving the above result, we first provide a lemma, which is  another form of discrete Logan phenomenon \cite{donoho1989uncertainty} for oversampled filter banks.
		
	\begin{lem}\label{loganlem}
		Let \( x \in \ell^2(\mathbb{Z}_N) \) be a signal supported on a set \( \mc C \subseteq \mathbb{Z}_N \), and let \( \mc J \subset \mc J_d \times \mc I_K \) be a set satisfying
		\begin{equation}\label{logancondtn}
			2 \#_{\mc J} \, \#_{\mc C} < p \left( \frac{d}{\phi(N)} \right)^2.
		\end{equation}
		Then, the best approximation to \( x \) by signals in
		\[
		\mc S_{\mc J} = \left\{ h \in \ell^2(\mathbb{Z}_N) \,\middle|\, (h \ast c_{q_i})(pk) = 0 \text{ for all } (k,i) \in \mc J^c \right\}
		\]
		is the zero signal.
		\end{lem}	
		\begin{proof}
			Note that if $\#_{\mc C}=0,$ then $x=0.$ In this case, the result trivially holds as $0 \in \mc S_{\mc J}.$
			
			Now assume that $\#_{\mc C} \neq 0.$  
			Then  by application of Theorem \ref{thm:thmuncertaintycq}, we first prove $x \notin \mc S_{\mc J}.$ For that, assume that the set of samples $\{(x\, \asterisk \,c_{q_i})(pk): k \in \mc J_d, i \in \mc I_K\}$ of $x$  has at most $\#_{\mc L}$  non-zero samples for some $\mc L \subset \mc J_d \times \mc I_K$ i.e., $(x\, \asterisk \,c_{q_i})(pk)=0$  for $(k,i) \in \mc L^c.$  By  Theorem \ref{thm:thmuncertaintycq} and \eqref{logancondtn}, we have
			$$\#_{\mc L}\#_{\mc C} \geq p\left(\frac{d}{\phi(N)}\right)^2> 2\#_{\mc J}\#_{\mc C}.$$
			This implies, $\#_{\mc L} > 2\#_{\mc J}.$ Since the signals in $\mc S_{\mc J}$ have at most $\#_{\mc J}$ non-zero samples, it follows that $x \notin \mc S_{\mc J}.$ 		
			
			To prove that the zero signal is the best approximation to $x,$ consider $0 \neq y \in \mc S_\mc J$ be arbitrary. Since $\mc R_{p,N}$ is a tight frame for $\ell^2(\mathbb Z_N)$ with frame bound $pd^2,$ then for $n \in \mathbb Z_N,$ we can write
			\begin{equation*}
				\begin{aligned}
					&\left|y(n)\right|=\frac{1}{pd^2}\left|\sum_{(k,i) \in \mc J}(y\, \asterisk \,c_{q_i})(pk)L_{pk}c_{q_i}(n)\right|=\frac{1}{pd^2}\left | \sum_{(k,i) \in \mc J}\sum_{\ell=0}^{N-1}y(\ell) \ol{L_{pk}c_{q_i}(\ell)} L_{pk}c_{q_i}(n) \right |\\
					&	\leq\frac{1}{pd^2}\sum_{\ell=0}^{N-1}\left |y(\ell) \right | \left(\sum_{(k,i) \in \mc J}\left |\ol{c_{q_i}(\ell-pk)}c_{q_i}(n-pk) \right | \right) \leq\frac{1}{pd^2}\sum_{\ell=0}^{N-1}\left |y(\ell) \right | \phi(N)^2 \#_{\mc J},
				\end{aligned}
			\end{equation*}
			where the last inequality is due to Lemma \ref{lem: modcqlem}.	This implies, $\max_{n\in \mathbb Z_N} |y(n)| \leq \frac{\phi(N)^2}{pd^2}\#_{\mc J}\|y\|_1.$  Let $P_\mc C$ denotes the projection operator on $\ell^2(\mathbb Z_N)$ onto $\mc C \subseteq \mathbb Z_N.$ Then,  using \eqref{logancondtn}, we have 
			\begin{equation*}
				\begin{aligned}
					\|P_\mc C y\|_1=\sum_{n \in \mc C}|y(n)| \leq \max_{n \in \mathbb Z_N}|y(n)| \#_\mc C \leq   \frac{\phi(N)^2}{pd^2}\#_{\mc J} \#_\mc C\|y\|_1 < \frac{1}{2} \|y\|_1.
				\end{aligned}
			\end{equation*}
			From this, we get $\|(I-P_\mc C)y\|_1 = \|y\|_1 - \|P_\mc C y\|_1 > \frac{1}{2}\|y\|_1 > \|P_\mc Cy\|_1.$ Consequently, we get
			\begin{equation*}
				\begin{aligned}
					\|x-y\|_1 &= \|P_\mc C(x-y)\|_1+ \|(I-P_\mc C)y\|_1 \\
					&\geq \|P_\mc Cx\|_1 -\|P_\mc Cy\|_1 +\|(I-P_\mc C)y\|_1  > \|P_\mc Cx\|_1=\|x\|_1.
				\end{aligned}
			\end{equation*} 
			This shows that any non-zero signal in $\mc S_{\mc J}$ is not approximating $x$ well. 	Hence, the claim follows.
		\end{proof} 
		Now we are ready to provide the proof of Theorem \ref{thm: uncertaintyrecovery}.\begin{proof}[Proof of Theorem \ref{thm: uncertaintyrecovery}]
			We first show that $x$ is the unique signal satisfying \eqref{uncertrecovcondtn} that generates $D=\mc T_{\mc J}x.$ For this purpose, consider $x'$ to be another signal which satisfies \eqref{uncertrecovcondtn} and  generates $D,$ i.e., $\mc T_{\mc J}x'=D=\mc T_{\mc J}x.$ Define $h:=x'-x,$ then $\mc T_{\mc J}h=0.$ This implies $h \in \mc S_{\mc J^c}.$  Therefore, $\#_{\supp(h)}\leq \#_{\supp(x')}+ \#_{\supp(x)}=2\#_\mc C,$ i.e, $h$ has at most $2\#_\mc C$ non-zero elements. Then,  using Theorem \ref{thm:thmuncertaintycq}, we get $(2\#_\mc C).\#_{\mc J^c}\geq p\left(\frac{d}{\phi(N)}\right)^2,$ which is a contradiction to \eqref{uncertrecovcondtn}. Hence $x$ is the unique signal satisfying \eqref{uncertrecovcondtn} that generates $D=\mc T_{\mc J}x.$

			We now provide an algorithm to reconstruct $x$ from $\mc T_{\mc  J}x.$ For this purpose,  let $\tilde{x}$ be the signal satisfying $\mc T_{\mc J}\tilde{x} =\mc T_{\mc J}x$ with the  minimum $\ell_1$ norm. In other words,
			$$\tilde{x} = \arg \min\limits_{x'} \left\{\|x'\|_1 \,:\, \mc T_{\mc J}x' =\mc T_{\mc J}x\right\}.$$ We show that $\tilde{x}=x.$ First, note that any $h \in \mc S_{\mc J^c}$ can be written in the form $h=x'-x,$ where $x'$ satisfies $T_{\mc J}x' =\mc T_{\mc J}x.$ Indeed, consider $x'=h+x$. Then, we have  $\mc T_{\mc J}x'=\mc T_{\mc J}(x+h)=\mc T_{\mc J}x.$
			It is given that $2\#_{\mc J^c}\#_\mc C<p\left(\frac{d}{\phi(N)}\right)^2,$ then by Lemma \ref{loganlem}, the best approximation  to $x$ by signals in $S_{\mc J^c}$ uses $h=0.$ In other words,
			$\min\limits_{h \in \mc S_{\mc J^c}}\, \|x+h\|_1=\|x\|_1$
			or
			$\min \left\{\|x+(x'-x)\|_1 : \mc T_{\mc J}x' =\mc T_{\mc J}x \right\} =\|x\|_1$
			or
			$\arg \min\limits_{x'}\left\{\|x+(x'-x)\|_1 : \mc T_{\mc J}x' =\mc T_{\mc J}x\right\} =x.$
			This proves our claim.
		\end{proof} 
		In the next subsection, we prove Theorem \ref{thm:recovfromnoisy}, which is another application of Theorem \ref{thm:thmuncertaintycq}. 
		\subsection{Proof of Theorem \ref{thm:recovfromnoisy}}\label{sub:nois}	
		To prove Theorem \ref{thm:recovfromnoisy}, consider the following optimization problem: 
		\begin{equation}\label{eq:optinoisy}
			\tilde{x} = \arg \min\limits_{x'\in \mc S_{\mc M}}\|y-x'\|_1,
		\end{equation}
		where \( \mc M \subset \mc J_d \times \mc I_{K} \) is such that $x \in \mc S_{\mc M}$ and $y$ is the noisy signal defined in \eqref{eq:noisysignal}.  This optimization problem can be solved efficiently
		by using convex optimization algorithms such as \cite{koh2008l1}.
		\begin{proof}[Proof of Theorem \ref{thm:recovfromnoisy}]
			We prove that the solution of   \eqref{eq:optinoisy} is $x,$ i.e., $\tilde{x}=x.$ For this purpose,  consider $x' \in \mc S_{\mc M},$ then we get
			$\|y-x^{'}\|_1=\|x+\eta-x^{'}\|_1
			=\|\eta+(x-x^{'})\|_1.$
			Note that for $(k,i) \in \mc M^c,$ 
			$\big((x-x')\, \asterisk\, c_{q_i}\big)(pk)=\langle x-x',L_{pk}c_{q_i} \rangle=\langle x,L_{pk}c_{q_i} \rangle\,-\,\langle x',L_{pk}c_{q_i} \rangle=0.$ This implies $x-x' \in \mc S_{\mc M}.$ Since $2 \#_{\mc M} \#_{\mc N} < p\left(\frac{d}{\phi(N)}\right)^2,$ then by Lemma \ref{loganlem}, we get 
			$	\min\limits_{x'\in \mc S_{\mc M}}\|y-x'\|_1=\min\limits_{x'\in \mc S_{\mc M}}\|\eta+(x-x')\|_1=\|\eta\|_1$
			or
			$\arg	\min\limits_{x'\in \mc S_{\mc M}}\|y-x'\|_1=\arg \min\limits_{x'\in \mc S_{\mc M}}\|\eta+(x-x')\|_1=x.$
			This proves our claim.
		\end{proof}
		The conditions \eqref{uncertrecovcondtn} and \eqref{condtn:2nmnt} associated with our missing-sample result, Theorem \ref{thm: uncertaintyrecovery}, and the denoising result, Theorem \ref{thm:recovfromnoisy}, are impractical for real-world applications. Consequently, expecting perfect recovery from the optimization problems \eqref{opti1} and \eqref{eq:optinoisy} is unrealistic in both scenarios.
		In the next section, we demonstrate that the recovery of a signal in practical settings can be significantly improved by utilizing the information provided by the filter bank based on Ramanujan sums \(c_{q_1}, c_{q_2}, \dots, c_{q_K}\).

		\section{Applications}\label{S4}
		In this section, we provide the applications of the filter bank based on Ramanujan sums $c_{q_1},c_{q_2}, \hdots, c_{q_K}$ in further improving the optimization problems \eqref{opti1} and \eqref{eq:optinoisy}. We show that incorporation of the periodicity information of signals using the properties of
		Ramanujan sums leads to greater signal-to-noise ratio (SNR) gains for each of the optimization
		problems \eqref{opti1} and \eqref{eq:optinoisy}.
		
		The SNR gain metric measures the improvement in the signal after applying a particular signal processing technique, such as filtering, noise reduction, or amplification. It compares the quality of the signal before and after processing, in terms of how much the signal's power is boosted relative to the noise.
		SNR is defined as the ratio of the power of the signal ($P_{\text{signal}}$) to the power of the noise and is  typically expressed in decibels (dB), using the formula:
		$
		\text{SNR (dB)} = 10 \log_{10}\left( P_{\text{signal}}/P_{\text{noise}} \right),
		$ where the power $P$ of a signal $x \in \ell^2(\mathbb Z_N)$ is given by the formula $
		P = \frac{1}{N} \|x\|^2.$
		The SNR gain is the difference between the SNR after processing (e.g., filtering) and the SNR before processing:
		\[
		\text{SNR gain (dB)} = \text{SNR after (dB)} - \text{SNR before (dB)}.
		\]
		
		%
		%
		%
		
		We provide the numerical implementations of our results Theorem \ref{thm: uncertaintyrecovery} and Theorem \ref{thm:recovfromnoisy}. We provide examples to explain the reconstruction of a given signal in view of the uncertainty principle  from $a)$ truncated signal $\mc T_{\mc J}x$ defined in \eqref{eq:TMx} in Subsection \ref{sub: missingdata} and $b)$ from noisy signal $y=x+\eta$ in Subsection \ref{sub:noisy}. We use the SNR gain metric to compare the outputs from the optimization problems \eqref{opti1} and \eqref{eq:optinoisy} with truncated signal $\mc T_{\mc J}x$ and the noisy signal $y=x+\eta$ defined in \eqref{eq:noisysignal}, respectively.

		\subsection{Application in signal recovery from missing data}\label{sub: missingdata}
		
		The condition \eqref{uncertrecovcondtn} is impractical for applications for two main reasons. First, the value of $N_B$ is often unavailable, and second, the bound on the right-hand side is typically too small to be meaningful. For instance, consider $N=38$ and $p=2$. Here, $d=19$, and the bound evaluates to $p\left(\frac{d}{\phi(N)}\right)^2 = 2\left(\frac{19}{\phi(38)}\right)^2 = 2.111$. This implies that $\#_{\mc J^c}\#_\mc C < 1.055$, which is not meaningful in practical scenarios and thus  it is reasonable to expect that the value $2\#_{\mc J^c}\#_\mc C$ will surpass the bound $p\left(\frac{d}{\phi(N)}\right)^2$. 
		

		The condition \eqref{uncertrecovcondtn} can be regarded as ideal in nature. In particular, the closer the quantity \( 2\#_{\mathcal{J}^c} \#_{\mathcal{C}} \) is to the bound \( p\left(\frac{d}{\phi(N)}\right)^2 \), the more accurately the signal \( x \) can be recovered from  \( \mathcal{T}_{\mathcal{J}} x \) via the optimization problem~\eqref{opti1}. Therefore,  the reconstruction from $\mc T_{\mc J}x$ can be expected to be worse if a significant number of samples are missing from the signal.
	In such cases, the problem~\eqref{opti1} admits further refinement under the assumption that the period of the signal (say, \(P\)) is known and satisfies \(P \mid N\). Suppose the periodic components of the signal $x$ are $P_1, P_2, \ldots, P_m$. Since $P = \text{lcm}(P_1, P_2, \dots, P_m)$, it follows that each $P_i$ also divides $N$. Consequently, $(x \ast c_{q_j})(pk) = 0$ for $q_j \notin \{P_1, P_2, \dots, P_m\}$. 
		This  information can be incorporated into the optimization problem \eqref{opti1}, modifying it as follows:
		\begin{equation}\label{opti1mod}
			\tilde{x} = \arg \min_{x'} \left\{ \|x'\|_1 \;:\; \mc T_{\mc J}x' = \mc T_{\mc J}x \text{ and } (x' \ast c_{q}) = 0 \text{ for all } q \mid N \text{ and } q \notin \{P_1, P_2, \dots, P_m\} \right\}.
		\end{equation}
		The solution provided by \eqref{opti1mod}  achieves a significantly higher SNR gain from $\mc T_{\mc J}x$ compared to the SNR gain obtained from the solution of \eqref{opti1}.
		
		In order to illustrate this, we consider an example of a signal $x$ with length $N=70$ and periodic components $P_1=5$ and $P_2=7$, represented by solid blue line in Figure~\ref{fig:thm1_plots}.
		The divisors of $N=70$ are $q_1=1, q_2=2, q_3=5, q_4=7, q_5=10, q_6=14, q_7=35, \text{ and } q_8=70$, yielding $K=8$. We consider the filter bank with decimation ratio $p=2,$ whose $i$-th channel corresponds to the $q_i$-th Ramanujan sum $c_{q_i}$ for $1 \leq i \leq 8$. For $p=2$, Theorem~\ref{thm:tight} guarantees that $\mc R_{2,70}$ forms a tight frame for $\ell^2(\mathbb Z_{70})$, satisfying the tight-frame condition of Theorem~\ref{thm: uncertaintyrecovery}. 
		Now, suppose the following samples of the signal $x$ are missing:
		\begin{equation}\label{missingsampleset2}
			\{(x\, \asterisk \,c_{q_3})(2k)\}_{k=0}^{10} \cup \{(x\, \asterisk \,c_{q_5})(2k)\}_{k=21}^{34} \cup \{(x\, \asterisk \,c_{q_7})(2k)\}_{k=10}^{34}.
		\end{equation}
	
		The set $\mc J$ for the  truncated sum $\mc T_{\mc J}x$ in this case is the complement of the set 
		\[
		\{(k,3)\}_{k=0}^{10} \cup \{(k,5)\}_{k=21}^{34}\cup \{(k,7)\}_{k=10}^{34},
		\]
		within $\{0,1,\ldots,34\} \times \{1,2,\ldots,8\}$. 
		The solution to the optimization problem \eqref{opti1} for this case is depicted by the green line with triangular markers in Figure~\ref{fig:thm1_plots}.
		It is clear from the figure that the present solution deviates significantly from the original signal, yielding a low SNR gain of $-8.4547$ dB. It is due to the reason that the value $2\#_{\mc J^c}\#_{\mc C}$ is large compared to the bound  $p(d/\phi(N)^2)$ in this case.	By incorporating the information that the original signal $x$ has periodic components $P_1=5$ and $P_2=7$, the modified optimization problem \eqref{opti1mod} becomes:
		\begin{equation*}
			\tilde{x} = \arg \min\limits_{x'} \left\{\|x'\|_1 \,:\, \mc T_{\mc J}x' =\mc T_{\mc J}x  \text{ and } (x' \ast c_{q}) = 0 \text{ for $q|70$ and $q\neq 5,7$}\right\}.
		\end{equation*}
		The solution for the missing sample set \eqref{missingsampleset2} is  illustrated using red square markers in Figure~\ref{fig:thm1_plots}.
		It is evident that this solution perfectly aligns with the original signal and achieves an SNR gain of $197.3205$ dB.

		\begin{table}[htbp]
			\centering
			\begin{tabular}{|c|c|c|c|c|}
				\hline
				\textbf{S.No.} & \textbf{Missing sample locations} & \textbf{$2\#_{\mc J^c}\#_{\mc C}$} & \textbf{SNR gain  (dB)} & \textbf{SNR gain  (dB)} \\
				& $(k,i)$ & & \textbf{(\ref{opti1})} & \textbf{(\ref{opti1mod})} \\
				\hline
				1. & $\{(k,2)\}_{k=0}^{2}, \{(k,3)\}_{k=17}^{20}, \{(k,5)\}_{k=27}^{29}$ & $20\#_{\mc C}$ & 174.8449 & 193.6376 \\
				2. & $\{(k,4)\}_{k=15}^{34}$ & $40\#_{\mc C}$ & 163.9832 & 223.7594 \\
				3. & $\{(k,3)\}_{k=0}^{24}, \{(k,8)\}_{k=0}^{24}$ & $100\#_{\mc C}$ & 2.7150 & 210.6610 \\
				4. & $\{(k,3)\}_{k=0}^{10}, \{(k,5)\}_{k=21}^{34}, \{(k,7)\}_{k=10}^{34}$ & $100\#_{\mc C}$ & -8.4547& 197.3205 \\
				5. & $\{(k,4)\}_{k=6}^{34}, \{(k,5)\}_{k=0}^{34}, (12,6), \{(k,7)\}_{k=0}^{34}$ & $200\#_{\mc C}$ & -12.3595 & 207.8186 \\
				\multirow{2}{*}{6.} & $\{(k,1)\}_{k=0}^{9}, \{(k,2)\}_{k=5}^{14}, \{(k,3)\}_{k=11}^{30}$ & \multirow{2}{*}{$200\#_{\mc C}$} & \multirow{2}{*}{-6.1835} & \multirow{2}{*}{224.2178} \\
				& $\{(k,4)\}_{k=21}^{34}, \{(k,5)\}_{k=17}^{34}, \{(k,7)\}_{k=6}^{34}$ & & & \\
				\hline
			\end{tabular}
			\vspace{2mm}
			\caption{ Comparison of SNR gains obtained from the solutions of optimization problems \eqref{opti1} and \eqref{opti1mod} across different cases of missing samples.}
			\label{table:missing}
		\end{table}

		The SNR gains obtained from optimization problems \eqref{opti1} and \eqref{opti1mod} for various cases of missing samples of the signal \(x\) are compared in Table \ref{table:missing}. 
		It is clear from the table that the solutions obtained from problem \eqref{opti1mod} consistently outperform those obtained from problem \eqref{opti1} for all cases. This  indicates that even when a significant number of samples are missing---or equivalently, when the value of \(2\#_{\mc J^c}\#_{\mc C}\) deviates substantially from the bound \(p(d/\phi(N))^2\)---the signal \(x\) can still be recovered with high accuracy from \(\mc T_{\mc J}x\), provided that the periodicity information  of the signal is known in advance.
		\begin{figure}[!htb]
			\centering
			\includegraphics[scale=0.5]{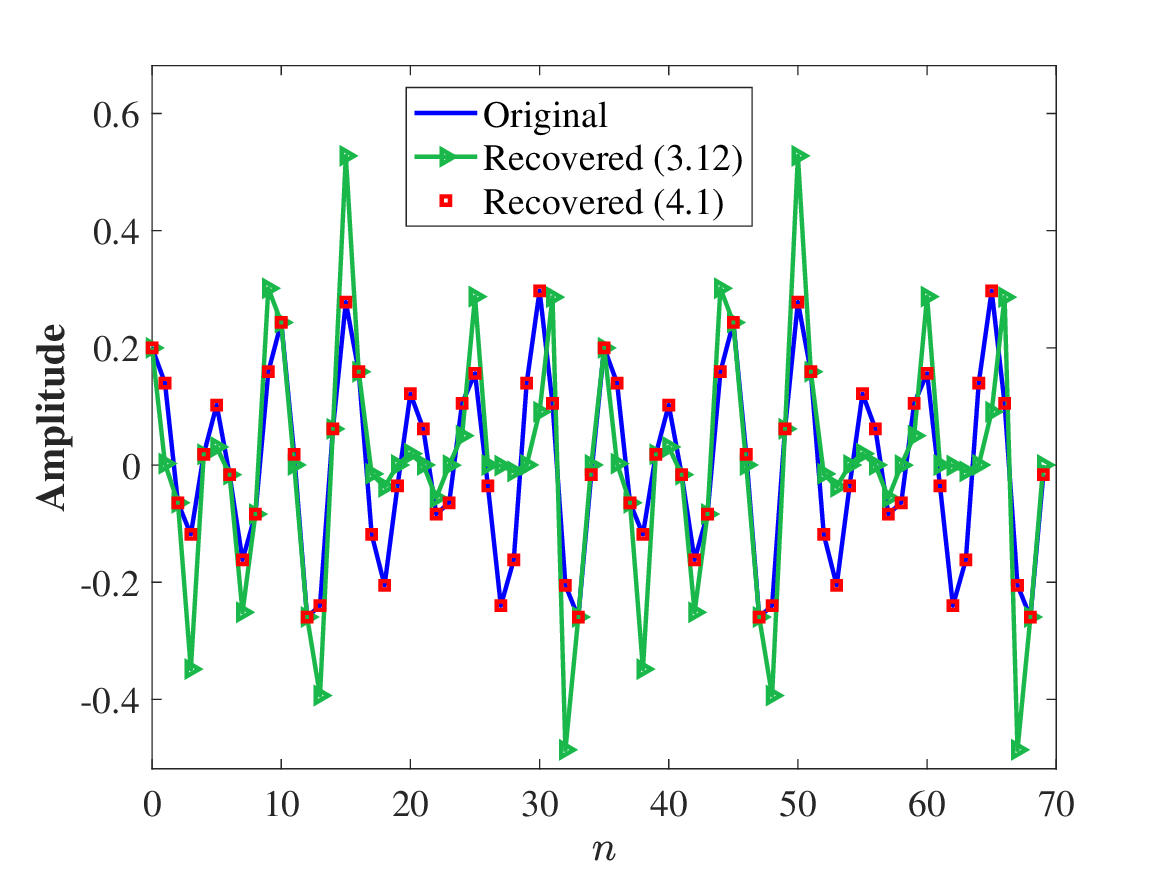}
			\label{fig:1thm1}
			\captionsetup{width=0.88\textwidth}
			\caption{Original signal (blue line) of length $N=70$ with periodic components 5 and 7, the solution of the optimization problem \eqref{opti1} (green line with triangular markers) for missing sample set \eqref{missingsampleset2}  with SNR gain  -8.4547 dB, and  the solution of the optimization problem \eqref{opti1mod} for missing sample set \eqref{missingsampleset2} (red square markers) with SNR gain 197.3205~dB.}
			\label{fig:thm1_plots}
		\end{figure}

In some scenarios, there
	may be no a priori knowledge of the periodic components of a signal. This raises the question: 
		\textit{Is it still possible to recover the original signal $x$ from $\mc T_{\mc J}x$ with high-accuracy  if a significant amount of samples are missing?} The answer to this question is ``yes". In such cases, the periodic components can be extracted from \(\mc T_{\mc J}x\), provided  the available samples in \(\mc T_{\mc J}x\) are sufficient to accurately predict the periodic components of the original signal and the  period of the signal is a divisor of \(N\). The energy outputs for the truncated sum \(\mc T_{\mc J}x\), corresponding to the missing sample sets listed in Table \ref{table:missing}, are shown in Figure~\ref{fig:pmxplots}.
		\begin{figure}[!htb]
			\centering
			\begin{minipage}{1.3\textwidth}
				\begin{subfigure}{0.265\textwidth}
					\includegraphics[width=\textwidth]{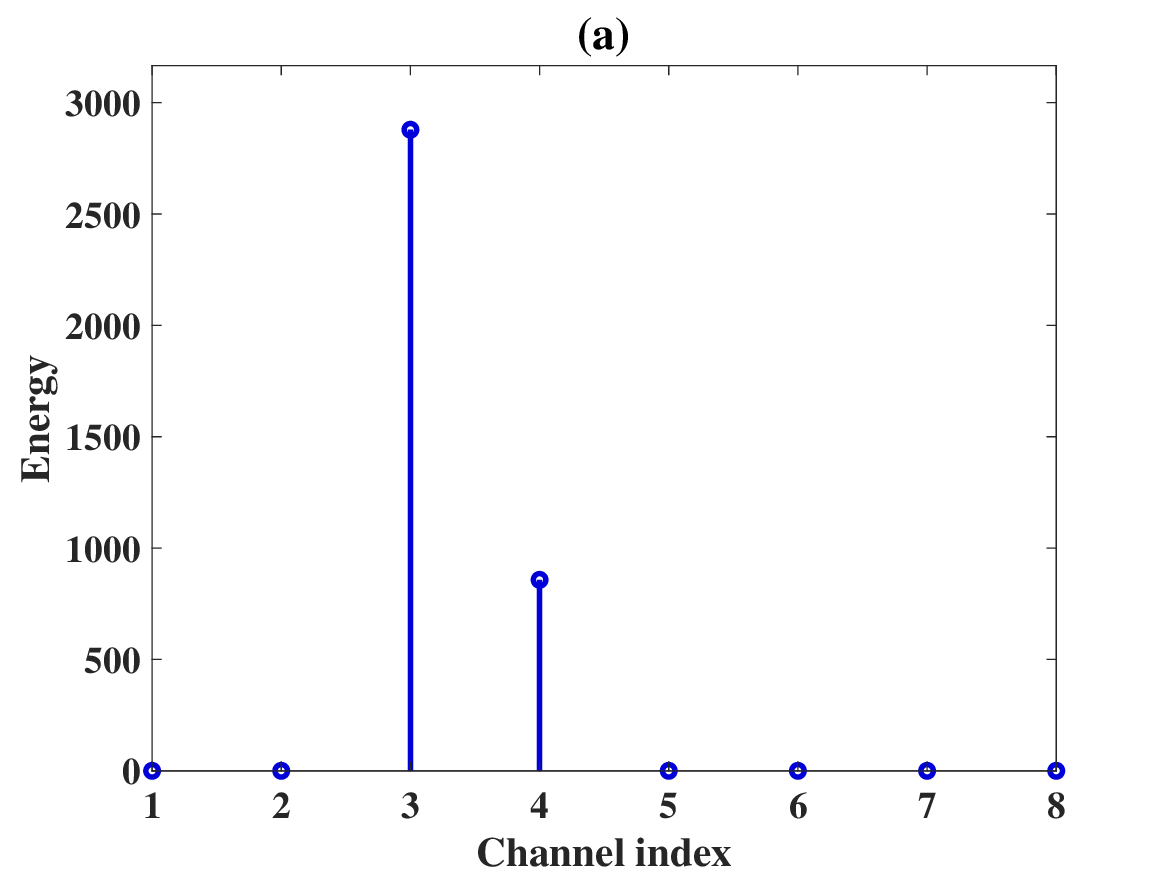}
					\label{fig:pmx1}
				\end{subfigure}
				\begin{subfigure}{0.265\textwidth}
					\includegraphics[width=\textwidth]{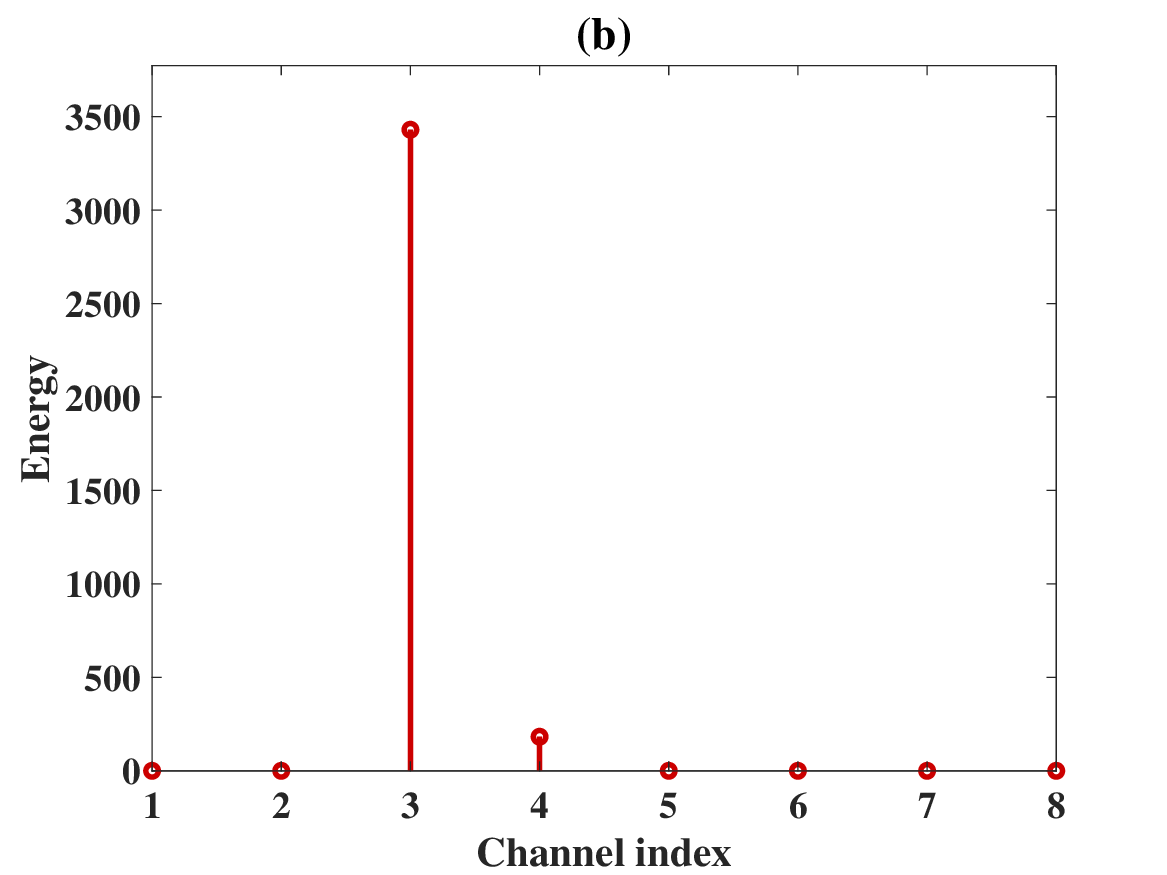}
					\label{fig:pmx2}
				\end{subfigure}
				\begin{subfigure}{0.265\textwidth}
					\includegraphics[width=\textwidth]{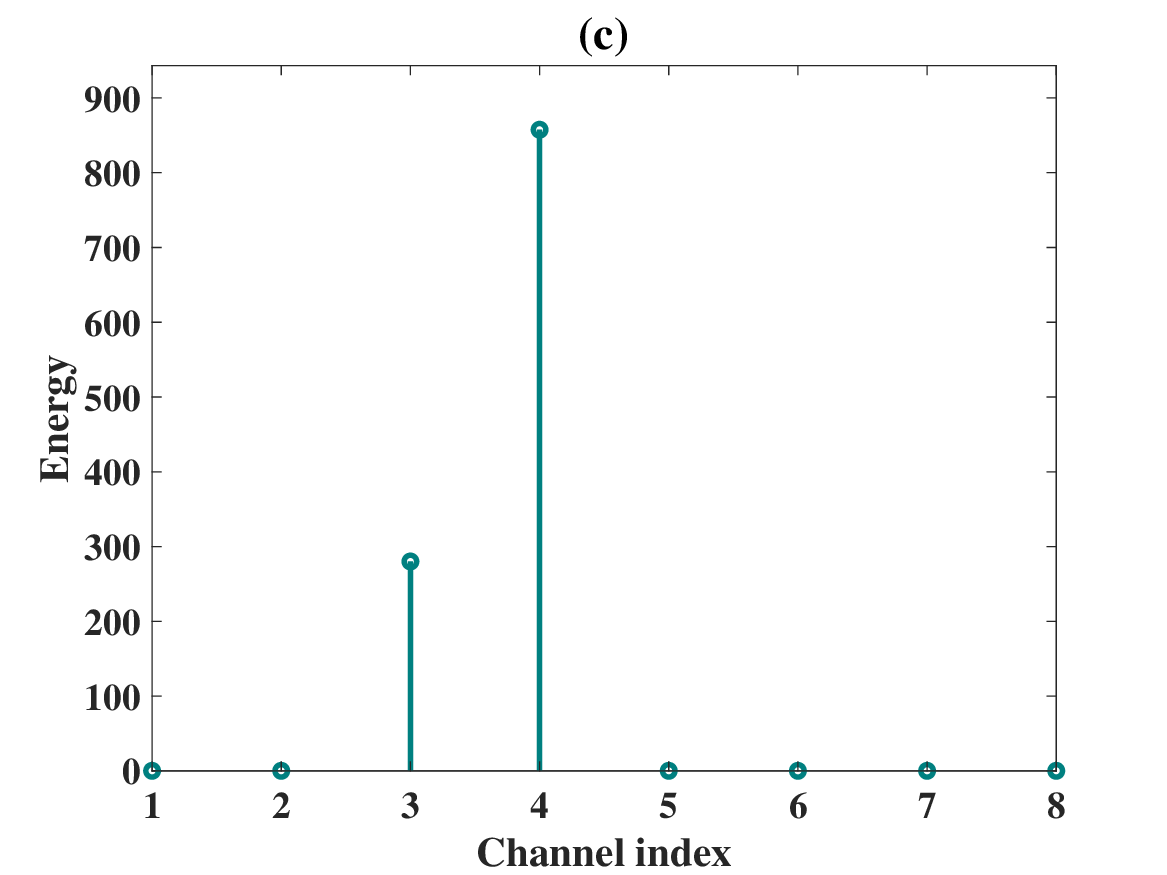}
					\label{fig:pmx3}
				\end{subfigure}
			\end{minipage}
			\begin{minipage}{1.3\textwidth}
				\begin{subfigure}{0.265\textwidth}
					\includegraphics[width=\textwidth]{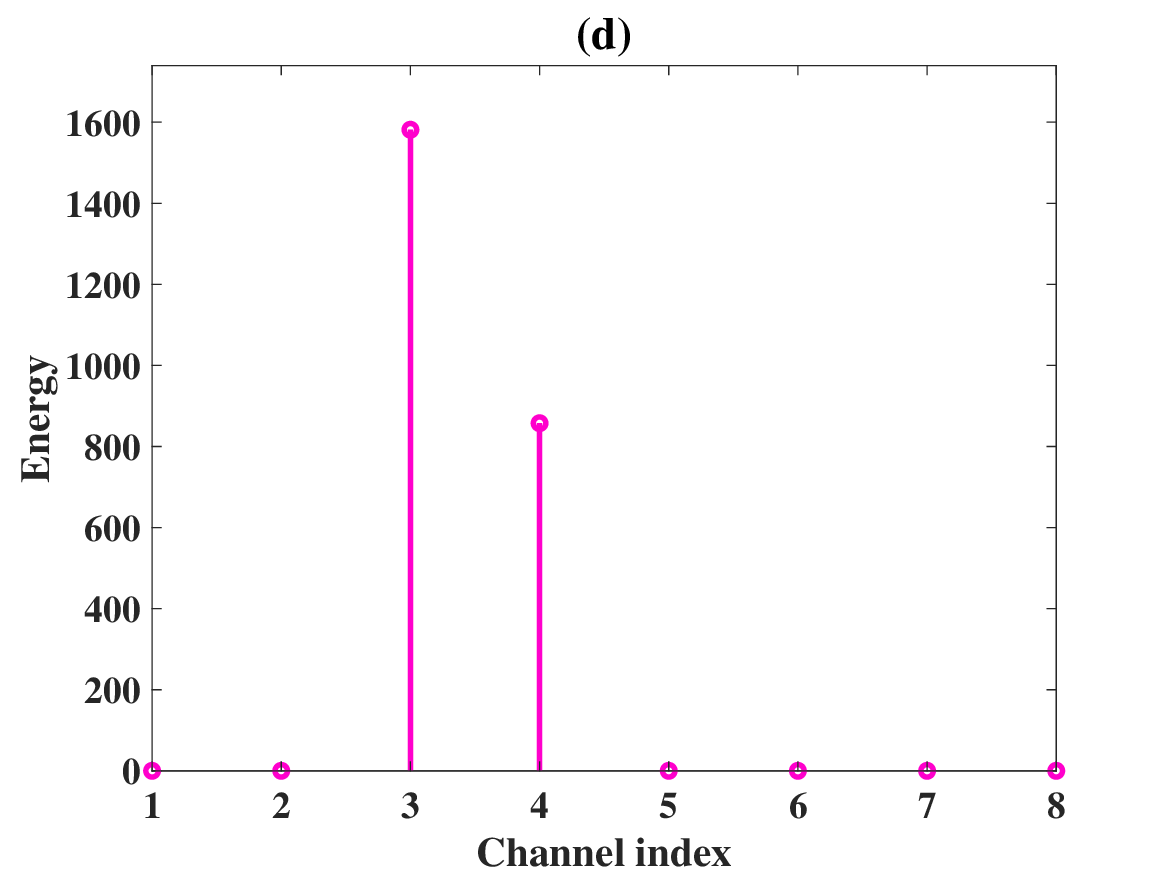}
					\label{fig:pmx4}
				\end{subfigure}
				\begin{subfigure}{0.265\textwidth}
					\includegraphics[width=\textwidth]{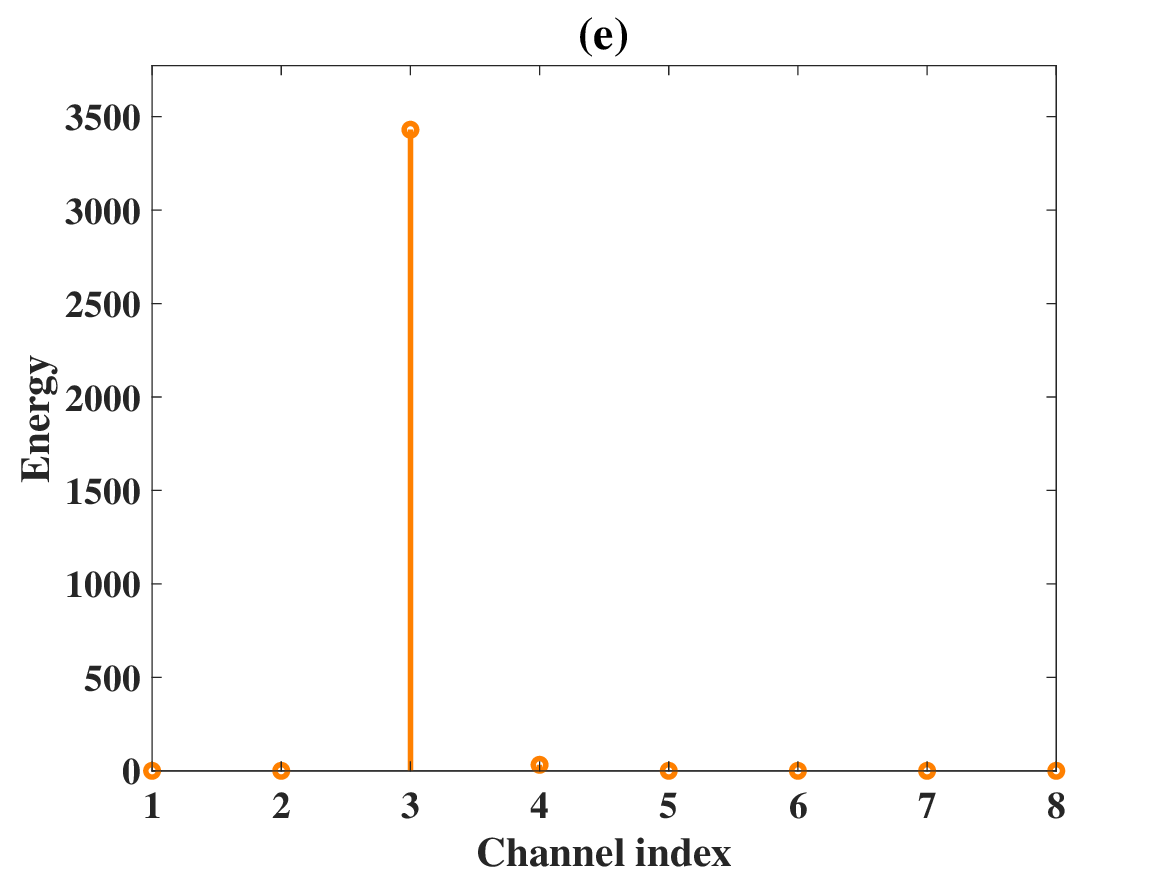}
					\label{fig:pmx5}
				\end{subfigure}
				\begin{subfigure}{0.265\textwidth}
					\includegraphics[width=\textwidth]{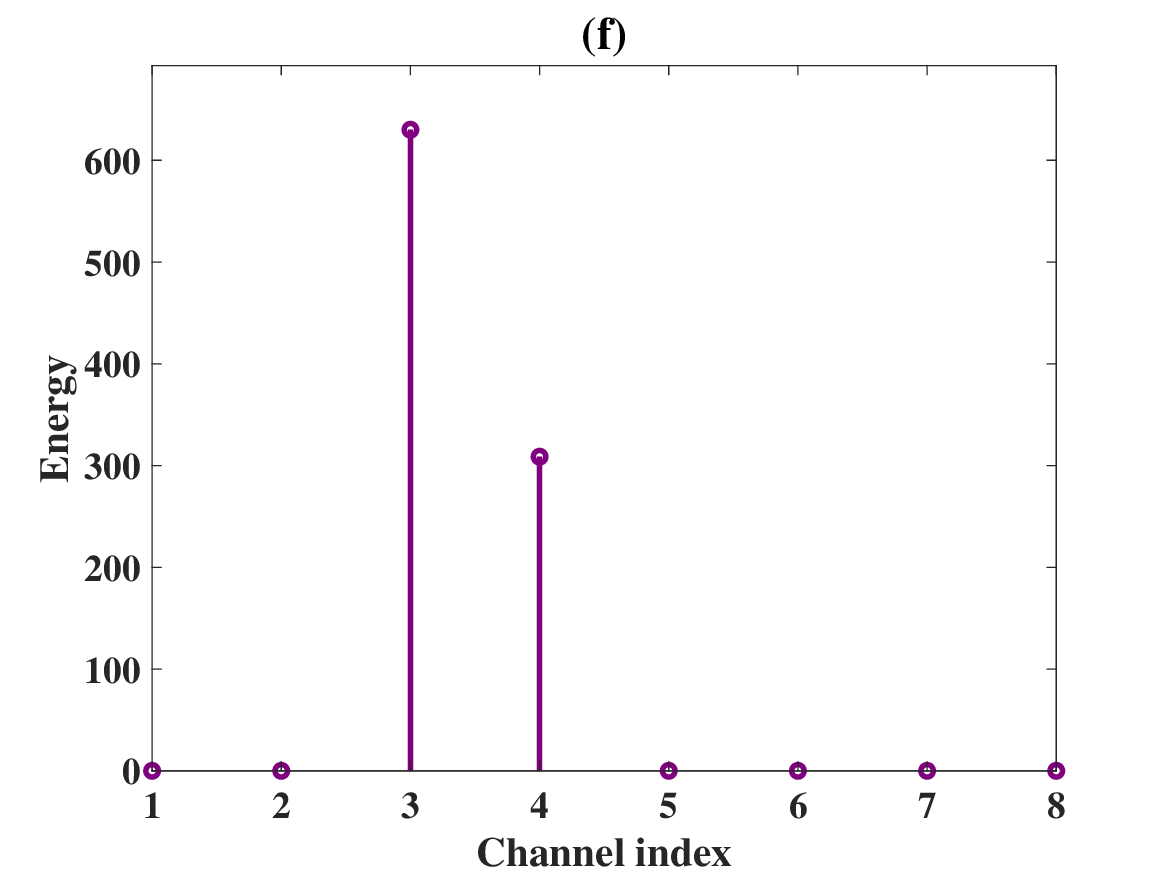}
					\label{fig:pmx6}
				\end{subfigure}
			\end{minipage}
			%
			\captionsetup{width=0.88\textwidth}
			\caption{Energy outputs for  \,$\mathcal{T}_{\mathcal{J}}x$ for various choices of $\mc J$  
				associated with missing sample sets listed in Table \ref{table:missing}.}
			\label{fig:pmxplots}
		\end{figure}
		It is evident from the figure that $\mc T_{\mc J}x$ accurately identifies the periodic components $q_3=5$ and $q_5=7$ of the original signal in all cases listed in Table \ref{table:missing}, except for  case $5$. In this instance, the energy from the fourth channel, corresponding to $c_7$, is notably low, as shown in Figure \ref{fig:pmxplots}(e). This occurs because the missing sample set for case $5$, specified in Table \ref{table:missing}, includes a significant portion of the output $\{(k,4)\}_{k=6}^{34}$, from the Ramanujan sum $c_7$. Consequently, $\mc T_{\mc J}x$ cannot reliably predict a periodic component $P$ if a substantial portion of the output from the corresponding Ramanujan sum $c_P$ is missing. 
		
		In summary, when more samples are missing, utilizing information about the periodic components of the original signal can  enhance the SNR gain.
		\subsection{Application in Denoising}\label{sub:noisy}
		
		In this subsection, we explore the recovery of the original signal \( x \in \ell^2(\mathbb{Z}_N) \) from the noisy signal \( y = x + \eta \) using a filter bank based on Ramanujan sums corresponding to the divisors of \( N \). Additionally, we examine the relationship between the condition \eqref{condtn:2nmnt} in our denoising result, Theorem \ref{thm:recovfromnoisy}, and the SNR gain for various signals.

		For recovery of the original signal $x$ from noisy signal $y$ using Theorem \ref{thm:recovfromnoisy},  $x \in \mc S_{\mc M}$ for some $\mc M \subset \mc J_d \times \mc I_K.$ In practice, it is possible that $\mc S_{\mc M}$ is unknown.  In such scenarios, a filter bank based on Ramanujan sums \( c_{q_1}, c_{q_2}, \ldots, c_{q_K} \), where \( K \) denotes the number of divisors of \( N \), can be employed to recover the support set \( \mathcal{S}_{\mathcal{M}} \) from the observation \( y \) under the assumption that the (unknown) period \( P \) of the signal satisfies \( P \mid N \).
	To this end, consider a signal \( x \in \ell^2(\mathbb{Z}_N) \), with the assumption that its (unknown) period \( P \) is a divisor of \( N \). Let the noisy signal $y=x+\eta$ be input to the  filter bank with channels based on  Ramanujan sums $c_{q_i}, i\in \mc I_K.$   Under moderate noise conditions, it is reasonable to expect that the channels corresponding to divisors of $P,$ say,  \(q_1, q_2, \dots, q_m\) will produce outputs with significantly higher energies compared to the remaining channels in the filter bank due to Theorem \ref{thm: period}. The small non-zero outputs produced by other channels are primarily attributable to the presence of noise. Since \( P \mid N \), it follows that each of the divisors \( q_1, q_2, \ldots, q_m \) of \( P \) also divides \( N \).  To filter out  the low energies, we apply a threshold to the output energies. An appropriately chosen threshold should ideally preserve energies from those channels \(c_{q_1}, c_{q_2}, \dots, c_{q_m}\), which  would have yielded non-zero outputs if the noiseless signal \(x\) had been passed through the filter bank instead of the noisy signal \(y\). Once the divisors $q_i, 1\leq i \leq m$ are known, the set $\mc M$  is given by $\mc M = \{ (k,i) : k \in \mc J_d, \; q_i \in \{ q_1, q_2, \ldots, q_m \} \}$
		and	the set $\mc S_{\mc M}$ can be computed as follows:
	\[
	\mc S_{\mc M} = \left\{ x \in \ell^2(\mathbb{Z}_N) \mid (x' * c_q)(p k) = 0, \, k \in \mc J_d, \, \text{for } q \mid N \text{ with } q \neq q_1, q_2, \ldots, q_m \right\}.
	\]
		Finally, the denoised signal $\tilde{x}$ is given by:			
		\begin{equation}\label{eq:optiden}
			\tilde{x} = \arg \min\limits_{x'\in \mc S_{\mc M}}\|y-x'\|_1.
		\end{equation}
		
		Table \ref{table:noisy} presents cases of signals with length \( N = 30 \), each having different periodic components, which are input to the filter bank with Ramanujan sums \( c_1, c_2, c_3, c_5, c_6, c_{10}, c_{15}, c_{30} \), and a decimation ratio \( p = 1 \). The table demonstrates that as the number of periodic components in the signal increases, the SNR gain from the noisy signals (with varying SNR's) decreases. This drop in SNR gain is attributed to the increasing value of \( \#_{\mathcal{M}} \), which leads to \( 2 \#_{\mathcal{M}} \#_{\mathcal{T}} \) deviating from the bound \( p \left( \frac{d}{\phi(N)} \right)^2 \), as shown in Table \ref{table:noisy}.
		
		\begin{table}[htbp]
			\centering
		\begin{tabular}{|c|p{3.0cm}|p{4.3cm}|c|c|c|}		\hline
				\textbf{\small{S.No.}} & \textbf{\small{True Periodic}} & \textbf{\small{Estimated Components}} & \textbf{$2\#_{\mc M}\#_{\mc T}$} &\textbf{\small{SNR (dB)}}& \textbf{\small{SNR gain (dB)}} \\
				& \textbf{\small{Components}} & \textbf{\small(via Filter Energies)} & &(\bf{\small{Noisy Signal}})& \textbf{(\ref{eq:optiden})} \\
				\hline
				1. & 1,3 & 3 & $60\#_{\mc T}$ &0.0007& 11.1245 \\
				2. & 3,5 & 3,5 & $120\#_{\mc T}$ &0.0006& 9.4778 \\
				3. & 2,15 & 2,15 & $120\#_{\mc T}$&0.0009& 6.7987 \\
				4. & 3,5,10 & 5,10 & $120\#_{\mc T}$&0.0004& 6.6875 \\
				5. & 1,2,3,6 & 3,6 &$120\#_{\mc T}$ &0.0008& 6.3936 \\
				6. & 1,3,5,6,10 & 5,6,10,30 & $240\#_{\mc T}$ &0.0005& 1.4456 \\
				7. & 2,3,5,6,10,15 & 5,6,10,15,30 & $300\#_{\mc T}$ &0.0008& 0.3772 \\
				\hline
			\end{tabular}
			\vspace{2mm}
			\caption{Comparison of  SNR gains from optimization problem \eqref{eq:optiden} with the values $2\#_{\mc M}\#_{\mc T}$ corresponding to the signals with various periodic components.}
			\label{table:noisy}
		\end{table}

			We explain the denoising process using an example.
				\begin{figure}[!b]
				\centering
				\begin{minipage}{0.9\textwidth}
					\begin{subfigure}{0.5\textwidth}
						\centering
						\includegraphics[width=\textwidth]{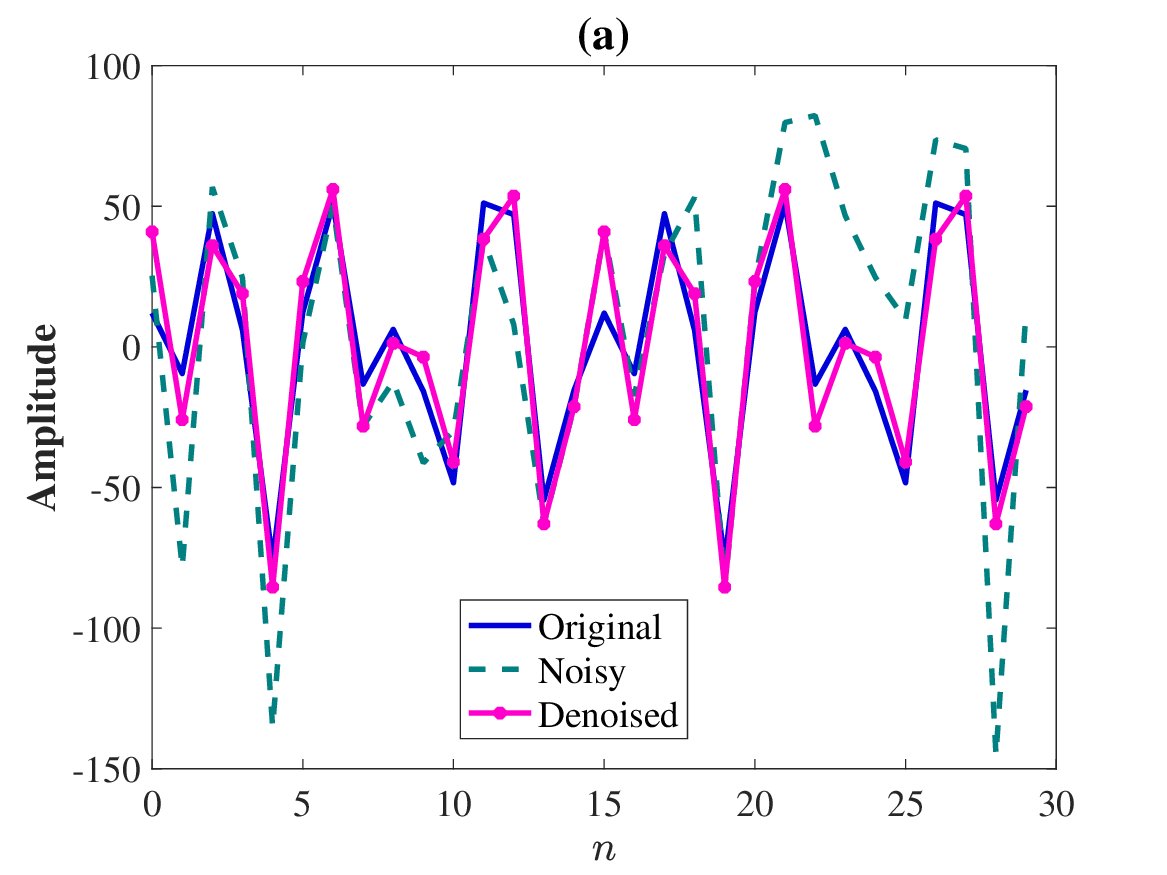}
						\label{fig:1thm2}
					\end{subfigure}
					\quad
					\begin{subfigure}{0.5\textwidth}
						\centering
						\includegraphics[width=\textwidth]{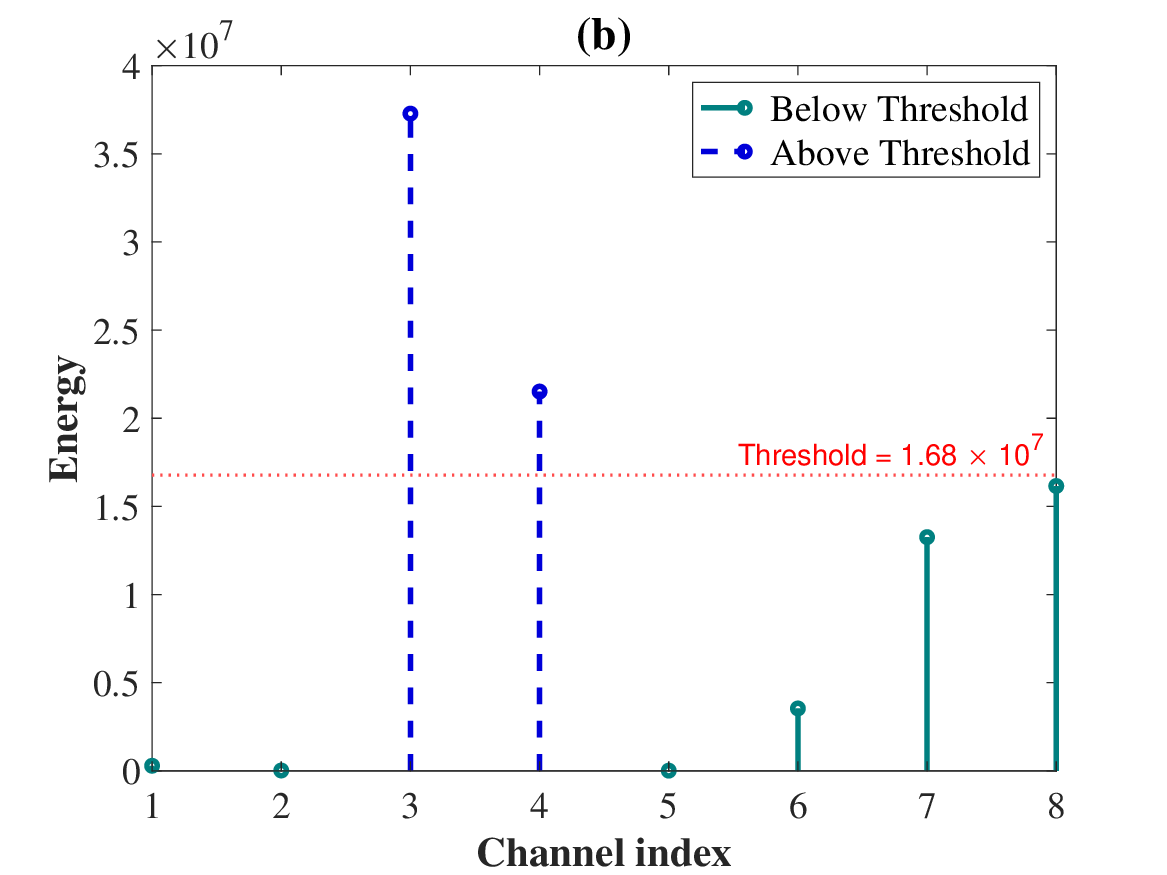}
						\label{fig:2thm2}
					\end{subfigure}
				\end{minipage}
				\captionsetup{width=0.88\textwidth}
				\caption{(a)  Original signal (solid blue line) with length  $N=30,$  the noisy signal (dashed green line) obtained by adding white Gaussian noise with SNR $6 \times 10^{-4}$ dB  and   the denoised signal (solid magenta line with circular markers), i.e., the solution  of the optimization problem  \eqref{eq:optiden}  with SNR 9.4778 dB. (b) Energy output from each channel of the filter bank and the corresponding  threshold value.}
				\label{fig:thm2_plots}
			\end{figure} 
			
			Consider a signal \( x \) with length \( N = 30 \) depicted by the blue solid line in   Figure \ref{fig:thm2_plots}(a).  The divisors of \( N = 30 \) are \( q_1 = 1, q_2 = 2, q_3 = 3, q_4 = 5, q_5 = 6, q_6 = 10, q_7 = 15, \) and \( q_8 = 30 \), resulting in \( K = 8 \). We consider a filter bank with decimation ratio \( p = 1 \), where the \( i \)-th channel corresponds to the \( q_i \)-th Ramanujan sum \( c_{q_i} \) for \( 1 \leq i \leq 8 \). For \( p = 1 \), Theorem~\ref{thm:tight} ensures that \( \mc R_{1,30} \) forms a tight frame for \( \ell^2(\mathbb Z_{30}) \), thus satisfying the tight-frame condition of Theorem~\ref{thm:recovfromnoisy}.
			The green dashed line in Figure \ref{fig:thm2_plots}(a) shows the noisy signal obtained by adding white Gaussian noise with SNR $6 \times 10^{-4} $ dB.  The noisy signal is then passed through the filter bank and the energy of each channel is recorded. The threshold value is set to 0.45 times the maximum energy output of the filter bank. The energy output from each channel of the filter bank is illustrated 
			in Figure \ref{fig:thm2_plots}(b).  It is evident from the figure that the channels corresponding to  $c_3$  and  $c_5$ (indicated by the dashed blue lines) exhibit energies above the threshold  whereas the channels corresponding to $c_1,c_2, c_6, c_{10}, c_{15}$, and $c_{30}$ (indicated by the solid green  lines) are below the threshold value $1.68 \times 10^7.$ Consequently, the set $\mc M$ in this case is given by: $\mc M=\{(k,i): 0 \leq k \leq 29,\, i=3,4\}$ and
			the solution of the optimization problem \eqref{eq:optiden}  after taking 	$\mc S_{\mc M} = \{ x \in \ell^2(\mathbb{Z}_N) \mid (x' * c_q)(k) = 0, \; 0 \leq k \leq 29, \text{ for } q \mid 30 \text{ with } q \neq 3,5 \}$ is depicted in  Figure \ref{fig:thm2_plots}(a) using solid magenta line with circular markers. 
			
			From the figure, it is clear that the solution from the optimization problem \eqref{eq:optiden} resembles  the original signal, achieving an SNR gain of 9.4778 dB. However, the solution is not perfect as  the value $2 \#_{\mc M} \#_{\mc T}=120 \#_{\mc T}$ exceeds the bound $p(d/\phi(N))^2=7.031$ for any noise vector $n$ with $\#_{\mc T}>1.$ 
			
			\vspace{4mm}
		\noindent	\textbf{Acknowledgment:} The authors thank the anonymous reviewers for carefully reading the Manuscript and for their valuable comments and suggestions, which have helped to improve the quality of this work.  The authors are also thankful to Asik Doctor for fruitful discussions at the beginning of this work.

				\bibliographystyle{abbrv}
				\bibliography{Ramanujanrefinal}
			\end{document}